\documentclass[]{article}

\usepackage[a4paper, margin=3cm]{geometry}
\usepackage[T1]{fontenc}
\usepackage[utf8]{inputenc}
\DeclareUnicodeCharacter{2010}{-}
\usepackage{microtype}
\usepackage{amsmath,amssymb,mathtools,amsthm}
\usepackage{mathrsfs}
\usepackage{graphicx}
\usepackage{subcaption}
\usepackage[table]{xcolor}
\usepackage{booktabs}
\usepackage{multirow}
\usepackage[title]{appendix}
\usepackage{algorithm}
\usepackage{algpseudocode}

\usepackage[numbers,sort&compress]{natbib}
\usepackage{url}
\usepackage[hidelinks]{hyperref}
\usepackage{authblk}
\usepackage{orcidlink}
\let\cite\citep
\newcommand{\email}[1]{\href{mailto:#1}{#1}}
\newenvironment{keywords}{\paragraph{Keywords:}}{}

\newcommand{\E}{\mathbb{E}}
\newcommand{\IR}{\mathbb{R}}
\newcommand{\F}{\mathcal{F}}

\theoremstyle{plain}
\newtheorem{theorem}{Theorem}
\newtheorem{proposition}{Proposition}
\newtheorem{lemma}{Lemma}

\newtheorem{assumption}{Assumption}
\theoremstyle{definition}
\newtheorem{definition}{Definition}

\theoremstyle{remark}
\newtheorem{remark}{Remark}

\hypersetup{
    pdftitle={Single-Loop Gradient Algorithms for Pessimistic Bilevel Optimization Problems},
    pdfauthor={Qichao Cao, Bo Zeng, Shangzhi Zeng, Jin Zhang}
}

\title{Single-Loop Gradient Algorithms for Pessimistic Bilevel Optimization Problems%
}

\author{
    Qichao Cao\thanks{Department of Mathematics, Southern University of Science and Technology, Shenzhen 518055, Guangdong, China. (\email{caoqc2024@mail.sustech.edu.cn}).} \quad
    Bo Zeng\thanks{Department of Industrial Engineering, University of Pittsburgh, Pittsburgh, Pennsylvania 15261, USA. (\email{bzeng@pitt.edu}).} \quad
    Shangzhi Zeng\thanks{National Center for Applied Mathematics Shenzhen, and Department of Mathematics, Southern University of Science and Technology, Shenzhen 518055, Guangdong, China. (\email{zengsz@sustech.edu.cn}).} \quad
    Jin Zhang\,\orcidlink{0000-0002-6691-5612}\thanks{Corresponding author. Department of Mathematics, and National Center for Applied Mathematics Shenzhen, Southern University of Science and Technology, Shenzhen 518055, Guangdong, China. (\email{zhangj9@sustech.edu.cn}).}
}
\date{}

\begin{document}

\maketitle
\begingroup
\renewcommand{\thefootnote}{}
\footnotetext{This paper is an extended version of our NeurIPS 2025
	conference paper \cite{qichao2025single}.}
\endgroup
\begin{abstract}
Bilevel optimization has recently attracted growing attention, particularly in the development of efficient numerical methods. Despite substantial progress on optimistic bilevel optimization, pessimistic bilevel optimization (PBO) remains much less explored, especially the design of fully first-order, single-loop gradient-based methods. To address this gap, we propose a smooth approximation of PBO through reformulation, penalization and regularization, and establish convergence guarantees in terms of both minimizers and stationarity. Building on this framework, we then develop two single-loop algorithms for deterministic and stochastic PBOs, respectively. Both use only first-order gradient information and avoid second-order derivatives and inner-loop subproblem solves. Non-asymptotic convergence rates for the proposed algorithms are established to provide theoretical guarantees. Through a systematic empirical study of both synthetic and practical problem instances, we demonstrate that our algorithms are highly effective and efficient. In particular, our results on spam classification and Smart Predict-then-Optimize further illustrate PBO’s advantages over its classical optimistic bilevel counterpart, highlighting its strong potential for practical modeling and the delivery of robust solutions.

\end{abstract}

\begin{keywords}
Pessimistic Bilevel Optimization, Gradient Method, Single-loop, Value Function, Smooth Approximation
\end{keywords}

\section{Introduction}\label{sec1}

Consider bilevel optimization, a hierarchical optimization problem, formulated as follows:
\begin{equation*}
	\min_{x\in{X}} \;F(x, y)\quad
	s.t.\quad y\in\mathcal{S}(x): =\underset{{y^\prime \in {Y}}}{\arg\min} f(x, y^\prime),
\end{equation*}
where $x\in\mathbb{R}^n$ and $y\in\mathbb{R}^m$ represent the upper-level and lower-level decision variables, respectively, and ${X}\subseteq\mathbb{R}^n$ and ${Y}\subseteq\mathbb{R}^m$ are nonempty, closed and convex sets. The functions $F:\mathbb{R}^{n}\times\mathbb{R}^{m}\to\mathbb{R}$ and $f:\mathbb{R}^{n}\times\mathbb{R}^{m}\to\mathbb{R}$ are the upper-level and lower-level objective functions, respectively. Bilevel optimization naturally models the  non-cooperative leader-follower game, often referred to as a Stackelberg game \cite{vonStackelbergHeinrich1953TTot}. 
Note that $\mathcal{S}(x)$ is not a singleton when the lower‐level  (i.e., the follower) problem admits multiple optimal solutions for $x$, rendering the upper‐level's (the leader's) objective  undetermined. To resolve this issue, the bilevel optimization literature generally distinguishes between two settings: Optimistic Bilevel Optimization (OBO) and Pessimistic Bilevel Optimization (PBO).

In the OBO setting, the upper-level believes that the lower-level yields a solution $y \in \mathcal{S}(x)$ that is most favorable to its objective $F$. The OBO formulation is thus:
\begin{equation*}
	\min_{x\in{X}} \min_{y\in \mathbb{R}^m} \;F(x, y)\quad
	s.t.\quad y\in\mathcal{S}(x),
\end{equation*}

Conversely, the PBO setting considers a cautious or adversarial scenario, where the lower-level is assumed to act against the upper-level by feeding back the least-favorable solution from $\mathcal{S}(x)$ to the upper-level.  Hence, the PBO formulation in fact has a tri-level structure, as follows:
\begin{equation*}
	\min_{x\in{X}} \max_{y\in \mathbb{R}^m} \;F(x, y)\quad
	s.t.\quad y\in\mathcal{S}(x).
\end{equation*}
Clearly, between OBO and PBO, the latter, once solved, yields a  conservative or robust solution when the lower-level behavior is unpredictable. 

\subsection{Relevant Literature}
Bilevel optimization has long served as a fundamental framework for hierarchical decision-making problems and has received renewed attention in recent years due to its broad range of applications, including data-driven optimization and learning problems. Yet, most existing studies focus on the optimistic  setting, whose theory, algorithms, and applications have been extensively reviewed in \cite{colson2007overview, dempe2013bilevel, dempe2020bilevel}. Classical approaches for OBO often reduce the bilevel problem to a single-level reformulation, for instance via Karush-Kuhn-Tucker (KKT) conditions leading to mathematical programs with complementarity constraints \cite{allende2013solving, luo1996mathematical}, or via value-function-based constraints \cite{ye1995optimality, outrata1990numerical}. Another line of work approximates the lower-level solution map through finite optimization trajectories and differentiates through the resulting computational process \cite{maclaurin2015gradient,pmlr-v70-franceschi17a}. Building on these reformulations and differentiation-based ideas, numerous algorithms have been developed for OBO under a variety of structural assumptions \cite{pedregosa2016hyperparameter, franceschi2018bilevel, lorraine2020optimizing, sow2022convergence, hong2023two, arbel2021amortized,  ji2021bilevel}.
Among them, first-order methods, especially their single-loop variants, have become a mainstream class of algorithms for OBO, as their low per-iteration computational cost and avoidance of nested lower-level solves make them well suited to large-scale problems \cite{lu2023slm, shen2023penalty, kwon2023fully, liu2022bome, liu2020generic, liu2021towards, lu2024first, kwon2023penalty}. These methods have also been successfully applied to learning tasks such as hyperparameter optimization, meta-learning, and hyper-representation learning \cite{pedregosa2016hyperparameter, franceschi2018bilevel, lorraine2020optimizing, liu2022bome}.

As mentioned earlier, PBO provides a conservative modeling scheme  to handle the ambiguity in the lower-level's response.  
For instance, in a couple of smart predict-then-optimize models, PBO has been used to account for the uncertainty induced by multiple optimal solutions of downstream optimization problems, with practical applications in path planning 
\cite{bucareydecision,jimenez2025pessimistic}. Also, to model the adaptive behavior of strategic agents, e.g., email senders and software developers who may adjust their actions in response to classifiers or detectors, PBO formulations have been introduced in \cite{bruckner2011stackelberg,benfield2024classification}, with numerical evidence supporting their effectiveness. Additionally, in hyperparameter optimization, where a fixed hyperparameter configuration may lead to multiple optimal training solutions, \cite{ustun2024hyperparameter} proposed a PBO formulation and showed that it can yield more robust solutions. Beyond these examples, PBO has been applied to a variety of practical problems, including but not limited to demand response management \cite{kis2021optimistic}, rank pricing and second-best toll pricing \cite{calvete2024novel,ban2009risk}, production-distribution planning \cite{zheng2016pessimistic}, and gene knockout models \cite{zeng2020practical}; see also \cite{liu2018pessimistic, dempe2020bilevel} for comprehensive surveys.

To gain a deep understanding on PBO, its mathematical structure has been studied from different perspectives. The research in \cite{aboussoror2001existence} investigated sufficient conditions for the existence of optimal solutions, while \cite{loridan1988approximate} analyzed properties of approximate solutions. Optimality conditions for PBO have also been developed, including KKT-type conditions for both smooth and nonsmooth settings \cite{dempe2014necessary,dempe2019two}. In addition, \cite{aussel2019pessimistic} studied the relationship between PBO and its MPCC reformulation. We note, however, that the algorithmic developments for solving PBO remain relatively limited. For linear PBO problems, several penalty-based methods have been proposed \cite{aboussoror2005weak, zheng2013exact}. For more general nonlinear PBO problems, existing approaches mainly rely on reformulating PBO into alternative but still challenging problems. For example, \cite{wiesemann2013pessimistic} introduced a semi-infinite programming reformulation of PBO; \cite{lampariello2019standard} reformulated PBO as an optimistic bilevel problem with a two-follower Nash game and solved the resulting problem through an MPCC reformulation; and \cite{zeng2020practical} transformed PBO into a computationally more friendly  minimax problem, together with solution methods for the linear case.  During the final preparation of this manuscript, a very recent study \cite{steintrue} appeared online, deriving a single-level constrained reformulation of PBO that avoids value functions and complementarity constraints, and proposing a penalty-based method for solving the resulting formulation. However, all the aforementioned reformulations either preserve nested or multi-level structures, or introduce complicated constraints, making the design of efficient algorithms rather challenging. 

Among the few computational methods available for nonlinear PBO,  \cite{benchouk2026relaxation,benchouk2025scholtes} applied relaxation methods to the KKT-type optimality
conditions of PBO and employed standard nonlinear equation solvers to compute the resulting nonlinear systems, building on the notions of C- and M-stationarity introduced in \cite{dempe2014necessary}. However, these systems can be large and require both first- and second-order derivative information about the problem functions, making the methods difficult to scale to high-dimensional settings.
More recently,  \cite{guanadaprox} developed the first principled gradient-based approach to nonlinear PBO. By combining a regularized lower-level value function with the KKT conditions of the inner maximization problem, they derive a constrained single-level reformulation that approximates PBO and propose two constrained optimization schemes for solving the resulting proximal subproblems within an outer loop that updates the relaxation level. However, their algorithms require nested loops and second-order information, and the convergence analysis relies on difficult-to-verify assumptions, such as the uniform boundedness of the dual variables across iterations, limiting their applicability to high-dimensional problems.

\subsection{Motivation and Our Contributions}
To solve PBO formulations arising from large-scale and data-driven applications, we identify two critical challenges that must be addressed. First, computational burden remains a major obstacle: existing double-loop second-order schemes are often computationally prohibitive, noting that Hessian-related operations and linear-system solves scale poorly with problem dimension, while repeatedly solving inner subproblems to sufficient accuracy introduces substantial additional overhead. These costs severely limit the scalability of such methods, particularly for modern machine learning applications.  Second, many practical applications are inherently stochastic: the upper- and lower-level objectives are often defined as expectations over data distributions and are accessible only through samples or stochastic first-order oracles. In such settings, exact-gradient deterministic methods become impractical or even inapplicable. Nevertheless, the development of algorithms for stochastic PBO has received comparatively little attention. 
In particular, fully first-order, single-loop methods with rigorous convergence guarantees remain largely underdeveloped for PBO.   Motivated by this gap, we ask:

\textit{Can we design fully first-order, single-loop algorithms for PBO that are applicable to both deterministic and stochastic oracle settings, with rigorous convergence guarantees?}	

This paper demonstrates that the answer is affirmative. To this end, we consider the pessimistic value-function reformulation of PBO:
\begin{equation}\label{pessimisticBLO}
	\min_{x \in {X}} \; \phi(x), \quad \text{where} \; \phi(x) := \max_{y\in \mathbb{R}^m}\; \left\{ F(x, y) \quad \text{s.t.} \; y \in \mathcal{S}(x) \right\}.
\end{equation}
Here, 
$\phi(x)$ is the value function of a maximization problem whose feasible region depends on the solution set 
$S(x)$ of another optimization problem. Note that 
$\phi(x)$
is generally non-smooth \cite{guo2024sensitivity}, and evaluating its value and (sub)gradient is computationally demanding, often requiring an accurate solution of the lower-level problem to characterize 
$S(x)$. This non-smoothness is the primary obstacle to directly minimizing 
$\phi(x)$, and it motivates the algorithmic approach we develop in this paper.	

Specifically, we construct a smooth approximation of $\phi(x)$ via penalization and regularization, reformulating PBO as a tractable smooth optimization problem that allows for efficient gradient-based methods. However, evaluating the gradient of the smoothed objective requires solving an associated minimax subproblem to obtain its saddle point, which is computationally expensive and impedes practical implementation. To overcome this issue, we employ a one-step gradient ascent-descent update to generate an inexact saddle-point estimate, which is then used to construct an inexact gradient for minimizing the smoothed objective. Based on this inexact-gradient mechanism, we develop a single-loop algorithmic framework for PBO. In particular, we propose fully first-order, single-loop methods for solving problem \eqref{pessimisticBLO} under both deterministic and stochastic oracle settings.

We next present the main contributions of this paper:

\noindent$(i)$ Smooth approximation and theoretical validation --
We introduce a novel smooth approximation for PBO by constructing a
continuously differentiable approximation  for the potentially non-smooth value
function $\phi(x)$, through penalization and regularization. The strength of this approximation is justified theoretically from two perspectives, i.e., the minimizer convergence and stationarity convergence. First, we establish the epi-convergence of the smooth approximation to 
$\phi(x)$, which yields the convergence of  global minimizers. 
Second, under explicit constraint representations and suitable constraint qualifications and a uniform parametric lower-level solution-set error bound condition, we establish stationarity convergence. Specifically, we show that any accumulation point of a sequence with vanishing first-order stationarity residuals for the smooth approximation problems is C-stationary for the original PBO.

\noindent$(ii)$ Single-loop algorithms and convergence guarantees-- Based on the smooth approximation, we develop two algorithms, namely SiPBA and Sto-SiPBA, for deterministic and stochastic PBO, respectively. These algorithms use only first-order information and update the upper-level and saddle-point variables within a single loop, without any subroutine to iteratively solve auxiliary subproblems. To handle stochasticity, Sto-SiPBA incorporates a recursive variance-reduced estimator for the stochastic upper-level direction. For both algorithms, under mild conditions, we first establish non-asymptotic bounds on the projected-gradient residual and the saddle-point tracking error, with the stochastic bounds holding in expectation. 
Moreover, when the feasible sets admit explicit inequality representations and suitable regularity conditions hold, we prove that the upper-level iterates generated by SiPBA possess a subsequence converging to a C-stationary point of the original PBO, and that the same conclusion holds almost surely for Sto-SiPBA.

\noindent$(iii)$ Empirical verification -- 
We evaluate the proposed framework on a variety of synthetic and practical problems, including spam classification and Smart Predict-then-Optimize. Obtained results  demonstrate superior performance across these settings, particularly for the spam-classification and Smart Predict-then-Optimize problems, providing strong empirical support for the effectiveness of the PBO scheme, the quality of our smooth approximation, and the proposed algorithms. 

\textbf{Organization of the paper.}
The remainder of this paper is organized as follows.
Section~\ref{sec:notation-prelim} introduces the notation and standing assumptions. Section~\ref{sec:smoothapproximation} develops the
penalized and regularized smooth approximation of the pessimistic value
function and establishes its differentiability, epi-convergence, and
stationarity convergence. Section~\ref{sec:deterministic-algorithm} presents SiPBA for deterministic
PBO and summarizes its non-asymptotic and subsequential convergence
guarantees.
Section~\ref{sec:stochastic-algorithm} develops Sto-SiPBA for stochastic PBO and provides the non-asymptotic
convergence rates for expected residual, and subsequential convergence guarantee
in almost surely sense.
Section~\ref{sec:numericalexperiments} reports empirical results on synthetic
problems, cross-corpus spam classification, and smart
predict-then-optimize. The final section concludes the paper, and additional
technical proofs are collected in the appendices.

\section{Notation and Preliminaries}\label{sec:notation-prelim}

This section introduces the notation and standing assumptions used throughout the paper and recalls the variational analysis concepts and stationarity notions needed for the subsequent analysis.

\subsection{Notation and Standing Assumptions}

Throughout, $\|\cdot\|$ and $\langle\cdot,\cdot\rangle$ denote the Euclidean
norm and inner product, respectively. For vectors
$a$ and $b$ of the same dimension, $a\circ b$ denotes their Hadamard
(componentwise) product. For a nonempty closed convex set $C\subseteq\IR^d$,
let $\operatorname{Proj}_{C}:\IR^d\to C$ denote the Euclidean projection onto
$C$, and let $N_C:\IR^d\rightrightarrows\IR^d$ denote the normal cone mapping
defined by
\(N_C(z):=\bigl\{v\in\IR^d:\langle v,z'-z\rangle\leq 0
\text{ for all }z'\in C\bigr\}\) for \(z\in C\), and
\(N_C(z):=\emptyset\) for \(z\notin C\).

We impose the following standing assumptions on the PBO problem
\eqref{pessimisticBLO} throughout the paper.

\begin{assumption}\label{assum1}
	The upper-level objective function $F$ is continuously differentiable, and
	its gradient $\nabla F$ is Lipschitz continuous on $X\times Y$. For every
	$x\in X$, the function $F(x,\cdot)$ is concave on $Y$. Moreover,
	$-F(x,\cdot)$ is coercive on $Y$ for every $x\in X$; equivalently,
	\(\lim_{\|y\|\to\infty,\,y\in Y}F(x,y)=-\infty\) for every
	\(x\in X\).
\end{assumption}

\begin{assumption}\label{assum2}
	The lower-level objective function $f$ is continuously differentiable, and
	its gradient $\nabla f$ is Lipschitz continuous on
	$X\times Y$. For every $x\in X$, the function $f(x,\cdot)$ is convex on
	$Y$, and $\mathcal{S}(x)$ is nonempty. Moreover, for every bounded set
	$D_x \subseteq X$, there exists a bounded set $D_y\subseteq Y$ such that
	$\mathcal{S}(x)\cap D_y\neq\emptyset$ for every $x\in D_x$.
\end{assumption}

By Assumption~\ref{assum2}, $\mathcal{S}(x)$ is nonempty for every $x\in X$.
Since $Y$ is closed and convex and $f(x,\cdot)$ is continuous and convex on
$Y$, the set $\mathcal{S}(x)$ is closed and convex. Furthermore, because
$\mathcal{S}(x)$ is nonempty and closed, the continuity of $F(x,\cdot)$ and
the coercivity of $-F(x,\cdot)$ imply that $F(x,\cdot)$ attains its maximum
over $\mathcal{S}(x)$. We therefore define the pessimistic solution mapping
$\mathcal S_p:X\rightrightarrows Y$ by
\[
\mathcal S_p(x)
:=
\arg\max_{y\in\mathcal S(x)}F(x,y).
\]
For every $x\in X$, the set $\mathcal S_p(x)$ is nonempty, compact and convex.

\subsection[Semicontinuity and Epi-convergence]{Semicontinuity and Epi-convergence}
\label{subsec:variational-preliminaries}

To prepare for the convergence analysis of the smooth approximation developed
in Section~\ref{sec:smoothapproximation}, we recall lower semicontinuity for
real-valued functions and inner semicontinuity for set-valued mappings.

\begin{definition}
	A function \(\phi:X\to\IR\) is said to be lower semicontinuous at
	\(\bar{x}\in X\) if, for every sequence \(\{x_k\}\subset X\) with
	\(x_k\to\bar{x}\), we have 
	\(
	\liminf_{k\to\infty}\phi(x_k)
	\geq
	\phi(\bar{x}).
	\)
\end{definition}

\begin{definition}
	A set-valued mapping \(\Gamma:X\rightrightarrows Y\) is said to be
	inner semicontinuous at \(\bar{x}\in X\) if, for every
	\(\bar{y}\in \Gamma(\bar{x})\) and every sequence
	\(\{x_k\}\subset X\) with \(x_k\to\bar{x}\), there exists a sequence
	\(\{y_k\}\subset Y\) such that
	\(
	y_k\in \Gamma(x_k)\) and \(	y_k\to\bar{y}.
	\)
\end{definition}

The following result shows that inner semicontinuity of the lower-level
solution mapping guarantees lower semicontinuity of the pessimistic value
function; see \citep[Lemma~B.8]{qichao2025single}.

\begin{lemma}\label{lem:isc-implies-lsc}
	If \(\mathcal{S}\) is inner semicontinuous at
	\(\bar{x}\in X\), then \(\phi\) is lower semicontinuous at
	\(\bar{x}\).
\end{lemma}

We next recall epi-convergence, which is defined through set convergence of
epigraphs. Let $\overline{\IR}:=\IR\cup\{-\infty,+\infty\}$. For an
extended-real-valued function \(\phi:X\to\overline{\IR}\), its epigraph
relative to \(X\) is defined by
\[
\operatorname{epi}_{X}\phi
:=
\left\{
(x,r)\in X\times\IR
:
r\ge\phi(x)
\right\}.
\]

The set convergence underlying epi-convergence is understood in the
Painlev\'e--Kuratowski sense. For a sequence of sets
\(C_k\subseteq X\times\IR\), its outer and inner limits are defined,
respectively, by
\[
\begin{aligned}
	\operatorname*{Limsup}_{k\to\infty}C_k
	&:=
	\left\{
	w
	\;\middle|\;
	\exists\,k_j\to\infty,\ 
	\exists\,w_j\in C_{k_j}
	\text{ such that }w_j\to w
	\right\},
	\\
	\operatorname*{Liminf}_{k\to\infty}C_k
	&:=
	\left\{
	w
	\;\middle|\;
	\exists\,w_k\in C_k
	\text{ for all sufficiently large }k
	\text{ such that }w_k\to w
	\right\}.
\end{aligned}
\]
Accordingly, epi-convergence is defined as follows; see
\citep[Definition~7.1]{rockafellar2009variational}.
\begin{definition}\label{def:epiconvergence}
	Let \(\phi_k,\phi:X\to\overline{\IR}\). The sequence
	\(\{\phi_k\}\) is said to epi-converge to \(\phi\) on \(X\),
	denoted by \(\phi_k\xrightarrow{e}\phi\), if
	\(
	\operatorname*{Liminf}_{k\to\infty}
	\operatorname{epi}_{X}\phi_k = 	\operatorname*{Limsup}_{k\to\infty}
	\operatorname{epi}_{X}\phi_k
	=
	\operatorname{epi}_{X}\phi.	
	\)
\end{definition}
Epi-convergence admits the following equivalent sequential characterization;
see
\citep[Proposition~7.2]{rockafellar2009variational}.
\begin{lemma}\label{lem:epi-sequential}
	Let \(\phi_k,\phi:X\to\overline{\IR}\). Then
	\(\phi_k\xrightarrow{e}\phi\) on \(X\) if and only if, for every
	\(\bar{x}\in X\), the following conditions hold:
	\begin{enumerate}
		\item for every sequence \(\{x_k\}\subset X\) with
		\(x_k\to\bar{x}\),
		\[
		\liminf_{k\to\infty}\phi_k(x_k)
		\ge
		\phi(\bar{x});
		\]
		\item there exists a sequence \(\{x_k\}\subset X\) with
		\(x_k\to\bar{x}\) such that
		\[
		\limsup_{k\to\infty}\phi_k(x_k)
		\leq
		\phi(\bar{x}).
		\]
	\end{enumerate}
\end{lemma}

A key consequence of epi-convergence is that, if a sequence of global
minimizers of \(\phi_k\) converges, then its limit is a global minimizer of
\(\phi\); see
\citep[Proposition~4.6]{bonnans2013perturbation}.

\begin{lemma}\label{lem:epi-minimizers}
	Suppose that \(\phi_k\xrightarrow{e}\phi\) on \(X\),
	\(x_k\in\arg\min_{x\in X}\phi_k(x)\), and
	\(x_k\to\bar{x}\). Then
	\[
	\bar{x}\in\arg\min_{x\in X}\phi(x).
	\]
\end{lemma}

\subsection{Constraint Qualifications and Stationarity Conditions}
\label{subsec:stNationarity-preliminaries}

This subsection introduces the
C-stationarity notion used in the convergence analysis. Suppose that the
upper- and lower-level feasible sets admit the representations
\[
X=\{x\in\mathbb R^n:G(x)\leq0\},\qquad
Y=\{y\in\mathbb R^m:g(y)\leq0\},
\]
where \(G:\mathbb R^n\to\mathbb R^q\) is continuously differentiable
with convex components and \(g:\mathbb R^m\to\mathbb R^p\) is twice
continuously differentiable with convex components.  With these explicit
representations, the pessimistic bilevel problem takes the form
\begin{equation}\label{PBOwithexplicitconstraint}
	\min_{G(x)\leq0}\ \max_{y\in\mathcal S(x)}F(x,y),
	\qquad
	\mathcal S(x):=\arg\min_{g(y)\leq0}f(x,y).
\end{equation}
For the lower-level problem, define the stationarity mapping
\[
\mathcal L(x,y,\mu)
:=\nabla_y f(x,y)+\nabla g(y)^\top\mu
\]
and the KKT multiplier set
\[
\Lambda(x,y):=\bigl\{\mu\in\mathbb R^p:
\mathcal L(x,y,\mu)=0,\ \mu\geq0,\ g(y)\leq0,
\ \mu^\top g(y)=0\bigr\}.
\]
We recall two constraint qualifications: the upper-level
Mangasarian-Fromovitz constraint qualification (MFCQ) and the lower-level
linear independence constraint qualification (LICQ).

\begin{definition}\label{def:level-mfcq}
	The upper-level MFCQ holds at a feasible point \(\bar x\) of PBO \eqref{PBOwithexplicitconstraint} if
	\[
	\pi\geq0,\qquad \pi^\top G(\bar x)=0,
	\qquad \nabla G(\bar x)^\top\pi=0
	\quad\Longrightarrow\quad \pi=0.
	\]
	The lower-level LICQ holds at $(\bar x,y)$, with
	$y\in\mathcal S(\bar x)$, if the active constraint gradients
	\[
	\{\nabla g_i(y):i=1,\ldots,p,\ g_i(y)=0\}
	\]
	are linearly independent.
\end{definition}

Fix \((\bar x,\bar y,\bar\mu)\) with
\(\bar y\in\mathcal S_p(\bar x)\) and
\(\bar\mu\in\Lambda(\bar x,\bar y)\).  The associated complementarity index sets are
\begin{equation}\label{index}
	\begin{aligned}
		\mathcal I_1(\bar y,\bar\mu)
		&:=\{i:\bar\mu_i=0,\ g_i(\bar y)<0\},\\
		\mathcal I_2(\bar y,\bar\mu)
		&:=\{i:\bar\mu_i=0,\ g_i(\bar y)=0\},\\
		\mathcal I_3(\bar y,\bar\mu)
		&:=\{i:\bar\mu_i>0,\ g_i(\bar y)=0\}.
	\end{aligned}
\end{equation}
We write $\mathcal I_j:=\mathcal I_j(\bar y,\bar\mu)$ for $j=1,2,3$.

We next recall C-stationarity, which is a necessary condition for local optimality of smooth
pessimistic bilevel programs under the assumptions stated in
\citep[Theorem~4.5]{dempe2014necessary} and
\citep[Theorem~4.1]{benchouk2025scholtes}.
The following definition specializes this condition to our setting and notation, where $f$ is assumed to be twice
continuously differentiable in a neighborhood of the reference
pair $(\bar x,\bar y)$.
\begin{definition}[C-stationarity]\label{def:C-stationarity}
	A feasible point \(\bar x\) is C-stationary of PBO \eqref{PBOwithexplicitconstraint} if there exist
	\(\bar y\in\mathcal S_p(\bar x)\),
	\(\bar\mu\in\Lambda(\bar x,\bar y)\), and multipliers
	\(\pi\in\mathbb R^q\), \(\omega\in\mathbb R^p\), and
	\(\gamma\in\mathbb R^m\) such that
	\begin{equation}\label{C-stationarity-system}
		\begin{gathered}
			\nabla_xF(\bar x,\bar y)
			+\nabla_x\mathcal L(\bar x,\bar y,\bar\mu)^\top\gamma
			+\nabla G(\bar x)^\top\pi=0,\\
			\nabla_yF(\bar x,\bar y)
			+\nabla_y\mathcal L(\bar x,\bar y,\bar\mu)^\top\gamma
			+\nabla g(\bar y)^\top\omega=0,\\
			\pi\geq0,\quad \pi^\top G(\bar x)=0,\qquad
			\omega_i=0,\ i\in\mathcal I_1,\quad
			\nabla g_i(\bar y)^\top\gamma=0,\ i\in\mathcal I_3,\\[-1mm]
			\omega_i\nabla g_i(\bar y)^\top\gamma\geq0,\quad i\in\mathcal I_2.
		\end{gathered}
	\end{equation}
\end{definition}

\section{Smooth Approximation of PBO}
\label{sec:smoothapproximation}
This section develops a differentiable approximation of the pessimistic
value function $\phi$.
We first construct the approximation and derive an explicit gradient formula,
then analyze its limiting behavior in terms of global minimizers and
stationary points.

\subsection{Construction of Smooth Approximation}	

Recall that, for every $x\in X$, the pessimistic value function $\phi(x)$ in
\eqref{pessimisticBLO} admits the equivalent constrained minimax
representation
\begin{equation}\label{constrainedminimax}
	\phi(x)=\min_{z\in Y}\max_{y\in Y}
	\bigl\{F(x,y): f(x,y)\leq f(x,z)\bigr\}.
\end{equation}
This reformulation technique was explored in \cite{zeng2020practical}.
It supports a value-function-based approach to designing computational
methods for PBO.

Based on this reformulation, we develop a smooth approximation of $\phi(x)$.
To remove the coupled constraint in \eqref{constrainedminimax}, we introduce
a penalty parameter $\rho>0$ and consider the penalized problem
\[
\min_{z\in Y}\max_{y\in Y}
\bigl\{F(x,y)-\rho\bigl(f(x,y)-f(x,z)\bigr)\bigr\}.
\]
Although penalization removes the coupled constraint, the saddle-point
variables need not be unique, and the resulting value function may therefore
remain nonsmooth in $x$. To obtain a smooth and well-posed approximation,
we introduce quadratic regularization in both saddle-point variables,
together with a bilinear coupling term, to define the regularized saddle
function
\begin{equation}\label{psi}
	\psi_{\rho,\sigma,\delta}(x,y,z)
	:=F(x,y)-\rho\bigl(f(x,y)-f(x,z)\bigr)
	+\frac{\sigma}{2}\|z\|^2-\sigma\langle y,z\rangle
	-\frac{\delta}{2}\|y\|^2,
\end{equation}
where $\sigma>0$ and $\delta>0$ are regularization parameters. For each
fixed $x\in X$, $\psi_{\rho,\sigma,\delta}(x,\cdot,\cdot)$ is
$\delta$-strongly concave in $y$ and $\sigma$-strongly convex in $z$.
Under Assumptions~\ref{assum1} and \ref{assum2}, this strong
concavity--convexity ensures a unique saddle point, denoted by
$(y_{\rho,\sigma,\delta}^*(x),z_{\rho,\sigma,\delta}^*(x))$.
The regularized value function is therefore finite and well defined, and
the minimization and maximization can be interchanged:
\begin{equation}\label{phi}
	\phi_{\rho,\sigma,\delta}(x)
	:=\min_{z\in Y}\max_{y\in Y}\psi_{\rho,\sigma,\delta}(x,y,z)
	=\max_{y\in Y}\min_{z\in Y}\psi_{\rho,\sigma,\delta}(x,y,z).
\end{equation}

The following result establishes the differentiability needed for the
first-order methods developed later.
\begin{theorem}\label{differentiable}
	Let $\rho,\sigma,\delta>0$. Then the mapping
	$x\mapsto\phi_{\rho,\sigma,\delta}(x)$ is differentiable at every $x\in X$,
	with gradient
	\begin{equation}\label{gradient_phi}
		\nabla\phi_{\rho,\sigma,\delta}(x)
		=\nabla_xF\bigl(x,y_{\rho,\sigma,\delta}^*(x)\bigr)
		-\rho\nabla_xf\bigl(x,y_{\rho,\sigma,\delta}^*(x)\bigr)
		+\rho\nabla_xf\bigl(x,z_{\rho,\sigma,\delta}^*(x)\bigr),
	\end{equation}
	where $(y_{\rho,\sigma,\delta}^*(x),z_{\rho,\sigma,\delta}^*(x))$
	denotes the unique saddle point of the minimax problem in \eqref{phi}.
\end{theorem}

\begin{remark}
	The bilinear coupling term $-\sigma\langle y,z\rangle$ in
	$\psi_{\rho,\sigma,\delta}$ provides the key link between the two
	saddle-point variables. Under the additional selection condition in the
	second assertion of Theorem~\ref{lim_yz}, it yields the common-limit
	property: if $y_k^*(x_k)\to\bar y$, then $z_k^*(x_k)\to\bar y$.
	This conclusion is subsequently used in Lemma~\ref{multiplierconverge}
	and Theorem~\ref{thm:Cstationarity}.
\end{remark}

\begin{remark}
	The construction in \eqref{psi} extends that of the conference
	version~\cite{qichao2025single} by adding the regularization term
	$-\frac{\delta}{2}\|y\|^2$. In that version, uniqueness of the saddle
	point relied on the strong concavity of $F(x,\cdot)$. Here,
	$F(x,\cdot)$ need only be concave. This weaker assumption covers a broader class
	of problems, including those in which the upper-level objective is affine
	in $y$; see, e.g., \cite{bucareydecision,kis2021optimistic}.
\end{remark}

\subsection{Convergence of Minimizers}
\label{subsec:epiconvergence}

Using the smooth approximation $\phi_{\rho,\sigma,\delta}$, we consider
the approximate problems
\begin{equation}\label{MMrho}
	\min_{x\in X}\;\phi_{\rho,\sigma,\delta}(x).
\end{equation}
To analyze their limiting behavior, let the parameter sequences satisfy
\begin{equation}\label{paramcriterion}
	\rho_k,\sigma_k,\delta_k>0,\qquad
	\rho_k\to\infty,\qquad
	\sigma_k\to0,\qquad
	\delta_k\to0,\qquad
	\limsup_{k\to\infty}\frac{\sigma_k}{\delta_k}<\infty.
\end{equation}
Here, $\rho_k\to\infty$ asymptotically enforces lower-level optimality
through penalization, while $\sigma_k,\delta_k\to0$ eliminate the
regularization bias. The final condition controls the $\sigma_k$-weighted
coupling and regularization terms relative to the $\delta_k$-regularization.
For notational convenience, throughout the remainder of the paper we write
$\phi_k(x)$, $\psi_k(x,y,z)$, $y_k^*(x)$, and $z_k^*(x)$ for
$\phi_{\rho_k,\sigma_k,\delta_k}(x)$,
$\psi_{\rho_k,\sigma_k,\delta_k}(x,y,z)$,
$y_{\rho_k,\sigma_k,\delta_k}^*(x)$, and
$z_{\rho_k,\sigma_k,\delta_k}^*(x)$, respectively.

We begin by establishing the limiting upper-bound relationship between
$\phi_k(x)$ and $\phi(x)$ required for epi-convergence in
Lemma~\ref{lem:epi-sequential}.

\begin{lemma}\label{lem:limsup}
	Suppose that the parameters satisfy \eqref{paramcriterion}. Then, for every
	$x\in X$,
	\begin{equation}\label{limsupphiapp}
		\limsup_{k\to\infty}\phi_k(x)\leq\phi(x).
	\end{equation}
\end{lemma}

The complementary lower estimate requires lower semicontinuity of 
pessimistic value function.
\begin{lemma}\label{lem_liminf}
	Suppose that the parameters satisfy \eqref{paramcriterion} and that $\phi$
	is lower semicontinuous on $X$. Then, for every $\bar x\in X$ and every
	sequence $\{x_k\}\subset X$ satisfying $x_k\to\bar x$,
	\begin{equation}\label{limsupepi}
		\liminf_{k\to\infty}\phi_k(x_k)\geq\phi(\bar x).
	\end{equation}
\end{lemma}

Suppose that $\phi$ is lower semicontinuous on $X$. For every
$\bar x\in X$, Lemma~\ref{lem_liminf} gives the lower-bound condition
in Lemma~\ref{lem:epi-sequential}, whereas Lemma~\ref{lem:limsup},
applied to the constant sequence $x_k\equiv\bar x$, gives the upper-bound
condition. Consequently, $\phi_k\xrightarrow{e}\phi$ on $X$.
Applying Lemma~\ref{lem:epi-minimizers} then yields the convergence
result for global minimizers stated below.

\begin{theorem}\label{convergethm}
	Suppose that the parameters satisfy \eqref{paramcriterion} and that $\phi$
	is lower semicontinuous on $X$. For each $k$, let
	$x_k\in\arg\min_{x\in X}\phi_k(x)$. Then every accumulation point $\bar x$
	of $\{x_k\}$ is a global solution of the original PBO; that is,
	$\bar x\in\arg\min_{x\in X}\phi(x)$.
\end{theorem}

The preceding result concerns only the upper-level variable. Since the
gradient formula~\eqref{gradient_phi} also involves the associated
saddle-point variables, we next characterize their behavior along a
convergent upper-level sequence.

\begin{theorem}\label{lim_yz}
	Suppose that the parameters satisfy \eqref{paramcriterion}. Let
	$\{x_k\}\subset X$ satisfy $x_k\to\bar x$, and assume that $\phi$ is lower
	semicontinuous at $\bar x$. Then every accumulation point of
	$\{y_k^*(x_k)\}$ belongs to $\mathcal S_p(\bar x)$.
	
	Moreover, suppose that $y_k^*(x_k)\to\bar y$ and that there exists a
	sequence $\{\hat y_k\}$ such that $\hat y_k\in\mathcal S(x_k)$ and
	$\hat y_k\to\bar y$. Then $z_k^*(x_k)\to\bar y$. In particular, if
	$\mathcal S$ is inner semicontinuous at $\bar x$, then the lower
	semicontinuity hypothesis follows from Lemma~\ref{lem:isc-implies-lsc}, and
	the existence of such a sequence $\{\hat y_k\}$ is automatic whenever
	$y_k^*(x_k)\to\bar y$.
\end{theorem}

\subsection{Convergence of Stationary Points}\label{sec:stationary}
We now complement the convergence of minimizers in
Subsection~\ref{subsec:epiconvergence} with a convergence result for approximate
stationary points. Under suitable local regularity assumptions, their
accumulation points are C-stationary for the original PBO. We use the
explicit inequality representations of $X$ and $Y$ from
Subsection~\ref{subsec:stNationarity-preliminaries}.

We first establish convergence of the saddle-point	variables and their associated lower-level KKT multipliers.
\begin{lemma}\label{multiplierconverge}
	Suppose that the parameters satisfy \eqref{paramcriterion}, that
	$\{x_k\}\subset X$ satisfies $x_k\to\bar x$, and that $\mathcal S$ is inner semicontinuous at
	$\bar x$. Let $\bar y$ be an accumulation point of
	$\{y_k^*(x_k)\}$ and, after passing to a subsequence, suppose that
	$y_k^*(x_k)\to\bar y$. Assume that $f$ is twice continuously
	differentiable in a neighborhood of $(\bar x,\bar y)$, and that lower-level LICQ holds at $(\bar x,\bar y)$.
	Then there exists a unique $\bar\mu\in\Lambda(\bar x,\bar y)$, and
	$z_k^*(x_k)\to\bar y$ with $\bar y\in\mathcal S_p(\bar x)$.
	Moreover, for all sufficiently large $k$, there exist unique
	multipliers $\lambda_k,\mu_k\in\mathbb R^p$ satisfying
	\begin{align}
		&\begin{cases}
			-\nabla_yF(x_k,y_k^*(x_k))
			+\rho_k\mathcal L(x_k,y_k^*(x_k),\lambda_k)
			+\sigma_k z_k^*(x_k)+\delta_k y_k^*(x_k)=0,\\[0.3em]
			\lambda_k\geq0,\quad
			\lambda_k^\top g(y_k^*(x_k))=0,\quad
			g(y_k^*(x_k))\leq0,
		\end{cases}\label{optconditiony}\\
		&\begin{cases}
			\rho_k\mathcal L(x_k,z_k^*(x_k),\mu_k)
			+\sigma_k\bigl(z_k^*(x_k)-y_k^*(x_k)\bigr)=0,\\[0.3em]
			\mu_k\geq0,\quad
			\mu_k^\top g(z_k^*(x_k))=0,\quad
			g(z_k^*(x_k))\leq0.
		\end{cases}\label{optconditionz}
	\end{align}
	Furthermore, $\lambda_k\to\bar\mu$ and $\mu_k\to\bar\mu$.
\end{lemma}

To control the scaled saddle-point differences, we impose a
\emph{uniform parametric error bound} in the sense of
\citep{ye1997exact}, using the lower-level stationarity residual.

\begin{assumption}\label{ass:lower-error-bound}
	At a given pair $(\bar x,\bar y)$ with $\bar y\in\mathcal S(\bar x)$,
	there exist neighborhoods $U$ of $\bar x$ and $V$ of $\bar y$, and a
	constant $\kappa>0$, such that
	\begin{equation}\label{eq:lower-solution-error-bound}
		\operatorname{dist}(y,\mathcal S(x))
		\leq\kappa\,\operatorname{dist}\bigl(0,\nabla_y f(x,y)+N_Y(y)\bigr),
		\quad x\in X\cap U,\ y\in Y\cap V.
	\end{equation}
\end{assumption}

The next lemma establishes boundedness of the
multiplier sequences and scaled differences, together with their limiting
C-stationarity relations.

\begin{lemma}\label{regularity}
	Suppose that the parameters satisfy \eqref{paramcriterion} and that
	$\{x_k\}\subset X$ satisfies $x_k\to\bar x$,
	$y_k^*(x_k),z_k^*(x_k)\to\bar y$, and
	$\lambda_k,\mu_k\to\bar\mu$, where $\bar\mu\in\Lambda(\bar x,\bar y)$ and
	$\lambda_k,\mu_k$ satisfy \eqref{optconditiony}--\eqref{optconditionz}.
	Assume that $f$ is twice continuously differentiable near
	$(\bar x,\bar y)$, that upper-level MFCQ holds at $\bar x$, and that
	lower-level LICQ and Assumption~\ref{ass:lower-error-bound}
	hold at $(\bar x,\bar y)$. If $r_k\to0$ and
	$\pi_k\in\mathbb R^q$ satisfy
	\begin{equation}\label{upper-kkt}
		\begin{gathered}
			\nabla_xF(x_k,y_k^*(x_k))
			\hspace{-0.1em}+\hspace{-0.1em}\rho_k(\nabla_xf(x_k,z_k^*(x_k))
			\hspace{-0.1em}-\hspace{-0.1em}\nabla_xf(x_k,y_k^*(x_k)))
			\hspace{-0.1em}+\hspace{-0.1em}\nabla G(x_k)^\top\pi_k\hspace{-0.1em}=\hspace{-0.1em}r_k,\\
			\pi_k\geq0,
			\qquad
			\pi_k^\top G(x_k)=0,
		\end{gathered}
	\end{equation}
	then $\{\pi_k\}$, $\{\gamma_k\}$, and $\{\omega_k\}$ are bounded, where
	$\gamma_k:=\rho_k\bigl(z_k^*(x_k)-y_k^*(x_k)\bigr)$ and
	$\omega_k:=\rho_k(\mu_k-\lambda_k)$.
	Every accumulation point $(\pi,\gamma,\omega)$ of
	$\{(\pi_k,\gamma_k,\omega_k)\}$ satisfies $\pi\geq0$,
	$\pi^\top G(\bar x)=0$, and
	\begin{equation}\label{scaled-limit-C-system}
		\begin{aligned}
			&\nabla_xF(\bar x,\bar y)
			+\nabla_x\mathcal L(\bar x,\bar y,\bar\mu)^\top\gamma
			+\nabla G(\bar x)^\top\pi=0,\\
			&\nabla_yF(\bar x,\bar y)
			+\nabla_y\mathcal L(\bar x,\bar y,\bar\mu)^\top\gamma
			+\nabla g(\bar y)^\top\omega=0,\\
			&\omega_i=0,\quad i\in\mathcal I_1,\qquad
			\nabla g_i(\bar y)^\top\gamma=0,\quad i\in\mathcal I_3,\qquad
			\omega_i\nabla g_i(\bar y)^\top\gamma\geq0,\quad i\in\mathcal I_2,
		\end{aligned}
	\end{equation}
	where $\mathcal I_j:=\mathcal I_j(\bar y,\bar\mu)$, $j=1,2,3$, are
	defined in \eqref{index}.
\end{lemma}

\begin{proof}
	For brevity, write $y_k^*:=y_k^*(x_k)$ and
	$z_k^*:=z_k^*(x_k)$. We restrict attention to sufficiently large $k$
	so that the interpolation segments below lie in the neighborhood where
	$f$ is twice continuously differentiable. Subtracting the stationarity equation in
	\eqref{optconditiony} from that in \eqref{optconditionz} gives
	\begin{equation}\label{yz-difference-eq}
		\nabla_yF(x_k,y_k^*)-(\delta_k+\sigma_k)y_k^*
		+\rho_k\bigl(
		\mathcal L(x_k,z_k^*,\mu_k)
		-\mathcal L(x_k,y_k^*,\lambda_k)\bigr)=0.
	\end{equation}
	The corresponding complementarity conditions yield
	\begin{equation}\label{comp-difference-eq}
		\rho_k\bigl(\mu_k\circ g(z_k^*)
		-\lambda_k\circ g(y_k^*)\bigr)=0.
	\end{equation}
	Recall that $\circ$ denotes the componentwise product.
	For $t\in[0,1]$, set
	\[
	y_k(t):=y_k^*+t(z_k^*-y_k^*),
	\qquad
	\mu_k(t):=\lambda_k+t(\mu_k-\lambda_k),
	\]
	and define the averaged derivative matrices
	\[
	H_k:=\int_0^1\nabla_{xy}^2f(x_k,y_k(t))\,dt
	\]
	and
	\[
	J_k:=\int_0^1
	\begin{bmatrix}
		\nabla_y\mathcal L(x_k,y_k(t),\mu_k(t))^\top
		&\nabla g(y_k(t))^\top\\
		\operatorname{Diag}(\mu_k(t))\nabla g(y_k(t))
		&\operatorname{Diag}(g(y_k(t)))
	\end{bmatrix}\,dt
	=:
	\begin{bmatrix}
		J_k^{11}&J_k^{12}\\
		J_k^{21}&J_k^{22}
	\end{bmatrix},
	\]
	where $\operatorname{Diag}(a)$ denotes the diagonal matrix with diagonal
	$a$. Recall that
	$\gamma_k=\rho_k(z_k^*-y_k^*)$ and
	$\omega_k=\rho_k(\mu_k-\lambda_k)$. Since $f$ and $g$ are twice
	continuously differentiable near the limit point, the integral form of
	the mean-value theorem, applied along the segments defined above, gives
	\begin{align*}
		&\rho_k\bigl(\nabla_xf(x_k,z_k^*)-\nabla_xf(x_k,y_k^*)\bigr)
		=H_k\gamma_k,\\
		&\rho_k
		\begin{bmatrix}
			\mathcal L(x_k,z_k^*,\mu_k)-\mathcal L(x_k,y_k^*,\lambda_k)\\
			\mu_k\circ g(z_k^*)-\lambda_k\circ g(y_k^*)
		\end{bmatrix}
		=J_k\begin{bmatrix}\gamma_k\\\omega_k\end{bmatrix}.
	\end{align*}
	Combining these identities with
	\eqref{upper-kkt}, \eqref{yz-difference-eq}, and
	\eqref{comp-difference-eq}, we obtain
	\begin{equation}\label{Mk-system}
		M_k
		\begin{bmatrix}\pi_k\\\gamma_k\\\omega_k\end{bmatrix}
		=-
		\begin{bmatrix}
			\nabla_xF(x_k,y_k^*)-r_k\\
			\nabla_yF(x_k,y_k^*)-(\delta_k+\sigma_k)y_k^*\\
			0
		\end{bmatrix},\quad
		M_k:=
		\begin{bmatrix}
			\nabla G(x_k)^\top&H_k&0\\
			0&J_k^{11}&J_k^{12}\\
			0&J_k^{21}&J_k^{22}
		\end{bmatrix}.
	\end{equation}
	The segments $y_k(t)$ and $\mu_k(t)$ converge to $\bar y$ and
	$\bar\mu$, respectively, uniformly for $t\in[0,1]$. Continuity of
	the derivatives therefore yields
	\begin{equation}\label{eq:Mk-limit}
		M_k\to\bar M:=
		\begin{bmatrix}
			\nabla G(\bar x)^\top
			&\nabla_x\mathcal L(\bar x,\bar y,\bar\mu)^\top&0\\
			0&\nabla_y\mathcal L(\bar x,\bar y,\bar\mu)^\top
			&\nabla g(\bar y)^\top\\
			0&\operatorname{Diag}(\bar\mu)\nabla g(\bar y)
			&\operatorname{Diag}(g(\bar y))
		\end{bmatrix}.
	\end{equation}
	Here, $\nabla_x\mathcal L^\top = \nabla_{xy}^2f$ because $g$ is
	independent of $x$. Moreover, the right-hand side of \eqref{Mk-system}
	is bounded, since $r_k\to0$, $\delta_k+\sigma_k\to0$, and
	$\nabla F(x_k,y_k^*)\to\nabla F(\bar x,\bar y)$.
	
	To pass the complementarity conditions to the limit, for each
	$i\in\{1,\ldots,p\}$ define
	$
	B_{k,i}:=\int_0^1\nabla g_i(y_k(t))^\top\,dt.
	$
	Then $B_{k,i}\to\nabla g_i(\bar y)^\top$, and the same integral
	formula gives
	\begin{equation}\label{eq:constraint-difference}
		\eta_{k,i}:=B_{k,i}\gamma_k
		=\rho_k\bigl(g_i(z_k^*)-g_i(y_k^*)\bigr).
	\end{equation}
	It follows from
	\eqref{optconditiony}--\eqref{optconditionz} that
	$\lambda_{k,i}g_i(y_k^*)=0$ and $\mu_{k,i}g_i(z_k^*)=0$
	for $i=1,\ldots,p$. Together with $\lambda_{k,i},\mu_{k,i}\geq0$ and
	$g_i(y_k^*),g_i(z_k^*)\leq0$, these identities imply
	\begin{equation}\label{eq:C-product-identity}
		\omega_{k,i}\eta_{k,i}
		=\rho_k^2(\mu_{k,i}-\lambda_{k,i})
		\bigl(g_i(z_k^*)-g_i(y_k^*)\bigr)
		=\rho_k^2\bigl[-\mu_{k,i}g_i(y_k^*)
		-\lambda_{k,i}g_i(z_k^*)\bigr]\geq0.
	\end{equation}
	
	We first bound $\gamma_k$ using the error bound condition in
	Assumption~\ref{ass:lower-error-bound}.
	From \eqref{optconditiony} and the convexity of $g$, we have
	\[
	e_k^y:=
	\frac{\nabla_yF(x_k,y_k^*)-\sigma_kz_k^*-\delta_ky_k^*}{\rho_k}
	\in \nabla_yf(x_k,y_k^*)+N_Y(y_k^*).
	\]
	The numerator is bounded. Since $(x_k,y_k^*)\to(\bar x,\bar y)$,
	the error bound \eqref{eq:lower-solution-error-bound} applies for all
	sufficiently large $k$, with some constant $\kappa>0$ independent of $k$,
	and gives
	\begin{equation}\label{eq:lower-distance-rate}
		\operatorname{dist}(y_k^*,\mathcal S(x_k))
		\leq \kappa\|e_k^y\|=O(\rho_k^{-1}).
	\end{equation}
	The $z$-subproblem is equivalent to minimizing
	$f(x_k,z)+\frac{\sigma_k}{2\rho_k}\|z-y_k^*\|^2$ over $Y$.
	Let $\widehat y_k:=\operatorname{Proj}_{\mathcal S(x_k)}y_k^*$.
	Optimality of $z_k^*$ in this subproblem and
	$f(x_k,\widehat y_k)\leq f(x_k,z_k^*)$ imply
	\[
	\frac{\sigma_k}{2\rho_k}
	\bigl(\|z_k^*-y_k^*\|^2-\|\widehat y_k-y_k^*\|^2\bigr)
	\leq f(x_k,\widehat y_k)-f(x_k,z_k^*)\leq0.
	\]
	Since $\sigma_k/(2\rho_k)>0$, using
	$\|\widehat y_k-y_k^*\|=\operatorname{dist}(y_k^*,\mathcal S(x_k))$,
	the definition of $\gamma_k$, and \eqref{eq:lower-distance-rate} gives
	\begin{equation}\label{eq:scaled-primal-gap-bound}
		\|z_k^*-y_k^*\|\leq\operatorname{dist}(y_k^*,\mathcal S(x_k)),
		\qquad
		\|\gamma_k\|\leq\rho_k\operatorname{dist}(y_k^*,\mathcal S(x_k))=O(1).
	\end{equation}
	
	Next, let $\mathcal A^g:=\{i:g_i(\bar y)=0\}$.
	Continuity and complementarity imply
	$\lambda_{k,i}=\mu_{k,i}=\omega_{k,i}=0$ for all
	$i\notin\mathcal A^g$ and sufficiently large $k$.
	If $\mathcal A^g$ is empty, then $\omega_k=0$ eventually.
	Otherwise, let $J_{k,\mathcal A^g}^{12}$ denote the submatrix of
	$J_k^{12}$ formed by the columns indexed by $\mathcal A^g$.
	By \eqref{eq:Mk-limit} and lower-level LICQ,
	\[
	J_{k,\mathcal A^g}^{12}\longrightarrow
	\nabla g_{\mathcal A^g}(\bar y)^\top,
	\]
	whose columns are linearly independent. 
	The second block row of \eqref{Mk-system} gives
	\[
	J_{k,\mathcal A^g}^{12}\omega_k^{\mathcal A^g}
	=-\nabla_yF(x_k,y_k^*)+(\delta_k+\sigma_k)y_k^*
	-J_k^{11}\gamma_k.
	\]
	Its right-hand side is bounded, so $\{\omega_k\}$ is bounded.
	
	Finally, the first block row of \eqref{Mk-system} shows that
	$\nabla G(x_k)^\top\pi_k$ is bounded.
	If $\{\pi_k\}$ were unbounded, a subsequence of
	$\pi_k/\|\pi_k\|$ would converge to a vector $\widehat\pi$ satisfying
	\[
	\|\widehat\pi\|=1,\qquad
	\nabla G(\bar x)^\top\widehat\pi=0,\qquad
	\widehat\pi\geq0,\qquad
	\widehat\pi^\top G(\bar x)=0.
	\]
	This contradicts upper-level MFCQ.
	Thus $\{(\pi_k,\gamma_k,\omega_k)\}$ is bounded.

	Let $(\pi,\gamma,\omega)$ be any accumulation point of this sequence
	and pass to a subsequence converging to it. Taking limits in
	\eqref{Mk-system}, using \eqref{eq:Mk-limit} and
	$r_k\to0$, $\delta_k+\sigma_k\to0$, gives the two stationarity
	equations in \eqref{scaled-limit-C-system}. The upper-level conditions
	$\pi\geq0$ and $\pi^\top G(\bar x)=0$ follow from \eqref{upper-kkt}.
	For $i\in\mathcal I_1$, the constraint is strictly inactive at both
	saddle components for sufficiently large $k$, so $\omega_{k,i}=0$ and hence $\omega_i=0$.
	For $i\in\mathcal I_3$, the common positive multiplier limit implies
	$\lambda_{k,i},\mu_{k,i}>0$ eventually. Complementarity then gives
	$g_i(y_k^*)=g_i(z_k^*)=0$, so $B_{k,i}\gamma_k=0$ and
	$\nabla g_i(\bar y)^\top\gamma=0$ in the limit.
	Finally, since $\eta_{k,i}=B_{k,i}\gamma_k\to
	\nabla g_i(\bar y)^\top\gamma$, taking limits in
	\eqref{eq:C-product-identity} yields the C-type product inequality
	for $i\in\mathcal I_2$. This proves \eqref{scaled-limit-C-system}.
\end{proof}

Combining the preceding lemmas gives the following
convergence result for sequences with vanishing first-order stationarity
residuals.

\begin{theorem}[C-stationarity of accumulation points]
	\label{thm:Cstationarity}
	Let $\{x_k\}\subset X$ satisfy
	\[
	x_k\to\bar x,
	\qquad
	\varepsilon_k
	:=\operatorname{dist}\bigl(0,\nabla\phi_k(x_k)+N_X(x_k)\bigr)
	\to0.
	\]
	Let $\bar y$ be an accumulation point of
	$\{y_k^*(x_k)\}$. Suppose that the parameters satisfy
	\eqref{paramcriterion}, that $\mathcal S$ is inner semicontinuous at
	$\bar x$, and that $f$ is twice continuously differentiable in a
	neighborhood of $(\bar x,\bar y)$. Assume further that upper-level
	MFCQ holds at $\bar x$ and that lower-level LICQ and the uniform parametric error bound
	condition in Assumption~\ref{ass:lower-error-bound} hold at
	$(\bar x,\bar y)$.
	Then $\bar y\in\mathcal S_p(\bar x)$ and $\bar x$ is a C-stationary
	point of the original PBO in the sense of
	Definition~\ref{def:C-stationarity}.
\end{theorem}

\begin{proof}
	Pass to a subsequence along which
	$y_k^*(x_k)\to\bar y$. Inner semicontinuity of $\mathcal S$ and
	Lemma~\ref{lem:isc-implies-lsc} imply that $\phi$ is lower
	semicontinuous at $\bar x$. Theorem~\ref{lim_yz} therefore gives
	$\bar y\in\mathcal S_p(\bar x)\subseteq\mathcal S(\bar x)$.
	By assumption, lower-level LICQ holds at this response.
	Lemma~\ref{multiplierconverge} then yields $z_k^*(x_k)\to\bar y$
	and multipliers $\lambda_k,\mu_k\to\bar\mu$ satisfying
	\eqref{optconditiony}--\eqref{optconditionz}.
	
	For each $k$, choose $v_k\in N_X(x_k)$ such that the residual
	\[
	r_k:=\nabla\phi_k(x_k)+v_k
	\quad\text{satisfies}\quad
	\|r_k\|\leq\varepsilon_k+\frac1k.
	\]
	Then $r_k\to0$.
	Since MFCQ is stable, upper-level MFCQ holds
	at $x_k$ for all sufficiently large $k$. Hence there exist multipliers
	$\pi_k$ such that
	\[
	v_k=\nabla G(x_k)^\top\pi_k,
	\qquad
	\pi_k\geq0,
	\qquad
	\pi_k^\top G(x_k)=0.
	\]
	Combining this representation with the gradient formula
	\eqref{gradient_phi} shows that \eqref{upper-kkt} holds. All hypotheses of Lemma~\ref{regularity} are now satisfied,
	so it provides an accumulation point
	$(\pi,\gamma,\omega)$ satisfying $\pi\geq0$,
	$\pi^\top G(\bar x)=0$, and \eqref{scaled-limit-C-system}.
	Together with
	$\bar y\in\mathcal S_p(\bar x)$ and
	$\bar\mu\in\Lambda(\bar x,\bar y)$, these are precisely the relations
	in Definition~\ref{def:C-stationarity}.
\end{proof}

\section{A Single-Loop Algorithm for PBO}
\label{sec:deterministic-algorithm}
The smooth approximation developed in Section~\ref{sec:smoothapproximation} provides a way to solve PBO through a sequence of smooth optimization problems $\min_{x\in X}\phi_k(x)$. In this section, we first consider the deterministic setting and develop a single-loop algorithm that approximates $\nabla\phi_k$ by tracking the associated saddle point. The algorithm and much of its convergence analysis follow the framework of our conference version \cite{qichao2025single}. The additional regularization in $y$, however, allows us to relax the strong concavity assumption on $F(x,\cdot)$ to concavity.

\subsection{Algorithm Description}\label{subsec:deterministic-sipba}
For the smooth approximation problem
\(\min_{x\in X}\phi_k(x)\), the ideal projected-gradient method would update $x^k$ according to
\begin{equation*}
	x^{k+1}
	=
	\operatorname{Proj}_X
	\left(
	x^k-\alpha_k\nabla\phi_k(x^k)
	\right).
\end{equation*}
By Theorem~\ref{differentiable}, the gradient of \(\phi_k\) is
$
\nabla\phi_k(x^k)
=
\nabla_x\psi_k
\bigl(
x^k,y_k^*(x^k),z_k^*(x^k)
\bigr).$
Evaluating this gradient exactly requires solving the inner minimax problem at every iteration, which can be computationally expensive. To approximate this gradient with a single saddle-point update per iteration, we propose the Single-loop Pessimistic Bilevel Algorithm (SiPBA), described below.

\textbf{Saddle-point tracking step.}
Given the current iterate \((x^k,y^k,z^k)\) and parameters
\(\rho_k,\sigma_k,\delta_k>0\), we track the saddle point of
$\underset{z\in Y}{\min}\underset{y\in Y}{\max}\;\psi_k(x^k,y,z)$
by taking one projected gradient-ascent step in \(y\) and one projected gradient-descent step in \(z\), with directions
\begin{equation}\label{det_direction_yz}
	\begin{aligned}
		d_y^k
		&:=
		\nabla_y\psi_k(x^k,y^k,z^k)=
		\nabla_yF(x^k,y^k)
		-\rho_k\nabla_yf(x^k,y^k)
		-\sigma_k z^k
		-\delta_k y^k,\\
		d_z^k
		&:=
		\nabla_z\psi_k(x^k,y^k,z^k)=
		\rho_k\nabla_yf(x^k,z^k)
		+\sigma_k(z^k-y^k).
	\end{aligned}
\end{equation}
The saddle variables are then updated according to
\begin{equation}\label{det_update_yz}
	\begin{aligned}
		y^{k+1}
		&=
		\operatorname{Proj}_Y
		\left(y^k+\beta_k d_y^k\right),\quad
		z^{k+1}
		=
		\operatorname{Proj}_Y
		\left(z^k-\beta_k d_z^k\right),
	\end{aligned}
\end{equation}
where \(\beta_k>0\) is the stepsize and
\(\operatorname{Proj}_Y\) denotes the Euclidean projection onto \(Y\).

\textbf{Inexact projected-gradient step.}
We use the updated pair \((y^{k+1},z^{k+1})\) as an approximation to the exact saddle point
\(\bigl(y_k^*(x^k),z_k^*(x^k)\bigr)\).
Substituting this pair into~\eqref{gradient_phi} gives the inexact upper-level gradient
\begin{equation}\label{det_direction_x}
	\begin{aligned}
		d_x^k
		&:=
		\nabla_x\psi_k
		\bigl(x^k,y^{k+1},z^{k+1}\bigr)\\
		&=
		\nabla_xF(x^k,y^{k+1})
		-\rho_k
		\left(
		\nabla_xf(x^k,y^{k+1})
		-\nabla_xf(x^k,z^{k+1})
		\right).
	\end{aligned}
\end{equation}
We then update the upper-level variable by
\begin{equation}\label{det_update_x}
	x^{k+1}
	=
	\operatorname{Proj}_X
	\left(x^k-\alpha_k d_x^k\right),
\end{equation}
where \(\alpha_k>0\) is the upper-level stepsize.

To ensure consistency of the smooth approximation with the original PBO, we choose the parameter sequences $\{\rho_k\}$, $\{\sigma_k\}$, and $\{\delta_k\}$ to satisfy~\eqref{paramcriterion}. The parameter and stepsize schedules used in the convergence analysis are specified in Theorem~\ref{thm:det-concave-rates}. Combining saddle-point tracking with the inexact upper-level update yields SiPBA, summarized in Algorithm~\ref{algorithm:det-sipba}.
\begin{algorithm}[htbp]
	\caption{
		\textbf{Si}ngle-loop \textbf{P}essimistic
		\textbf{B}ilevel \textbf{A}lgorithm (\textbf{SiPBA})}
	\label{algorithm:det-sipba}
	\begin{algorithmic}[1]
		\Require Initial points
		$(x^0,y^0,z^0)\in X\times Y\times Y$
		\For{$k=0,1,\dots,K-1$}
		\State Update the stepsizes $\alpha_k,\beta_k>0$, the penalty
		parameter $\rho_k>0$, and the regularization parameters
		$\sigma_k,\delta_k>0$ according to \eqref{deter:par}.
		\State Compute the one-step ascent-descent directions $d_y^k$ and $d_z^k$ as
		in~\eqref{det_direction_yz}.
		\State Update $y^{k+1}$ and $z^{k+1}$ using \eqref{det_update_yz}.
		\State Compute the inexact gradient direction $d_x^k$ as in~\eqref{det_direction_x}.
		\State Update $x^{k+1}$ using~\eqref{det_update_x}.
		\EndFor
	\end{algorithmic}
\end{algorithm}

\subsection{Convergence Results}\label{subsec:deterministic-convergence}
We establish convergence of SiPBA in two steps. First, we give nonasymptotic bounds on the stationarity residual for the current smooth approximation and the saddle-point tracking error. We then show that, under suitable local regularity conditions, the iterates admit a subsequence converging to a C-stationary point of the original PBO.
The following compactness assumption is used throughout
Sections~\ref{sec:deterministic-algorithm}
and~\ref{sec:stochastic-algorithm}.

\begin{assumption}\label{ass:Xcompact}
	The set \(X\) is compact.
\end{assumption}

We evaluate the convergence of SiPBA using two residuals. The first is the
projected-gradient mapping
\[
\mathcal G_k(x)
:=
\frac{1}{\alpha_k}
\left(
x-\operatorname{Proj}_X
\bigl(x-\alpha_k\nabla\phi_k(x)\bigr)
\right),
\]
which measures stationarity of the current smooth approximation $\phi_k(x)$. In
particular,
\(\mathcal G_k(x)=0\) if and only if
\(0\in\nabla\phi_k(x)+N_X(x)\).
The second is the saddle-point tracking error
\(
\|(y^k,z^k)-(y_k^*(x^k),z_k^*(x^k))\|,
\)
which measures how accurately $(y^k,z^k)$ tracks the exact saddle point of the current smooth approximation. Through the contraction of the saddle-point update, this error also controls the discrepancy between the gradient estimate $d_x^k$, evaluated at $(y^{k+1},z^{k+1})$, and the exact gradient $\nabla\phi_k(x^k)$.

Our nonasymptotic analysis follows the saddle-point contraction
and descent arguments in \citep[Appendix~C]{qichao2025single}. The saddle-point contraction estimate now uses $\bar\sigma_k:=\min\{\sigma_k,\delta_k\}$, and the additional terms arising from changes in $\delta_k$ have the same order as the corresponding $\sigma_k$-terms under the schedules below. We state the resulting rates and common-subsequence conclusion below and omit the repeated details.
\begin{theorem}
	\label{thm:det-concave-rates}
	Let $\{(x^k,y^k,z^k)\}$ be generated by SiPBA with
	\begin{equation}\label{deter:par}
		\begin{aligned}
			&\alpha_k=\alpha_0(k+1)^{-s},\;
			\beta_k=\beta_0(k+1)^{-3t},\;\\
			&	\sigma_k=\sigma_0(k+1)^{-t},\; \delta_k=\delta_0(k+1)^{-t},\;
			\rho_k=\rho_0(k+1)^t ,
		\end{aligned}
	\end{equation}
	with $\alpha_0,\beta_0,\sigma_0,\delta_0,\rho_0>0$, and $s,t>0$.
	Suppose that $\phi$ is bounded below on $X$ and that
	$
	s>7t,
	\;
	s+2t<1.
	$
	If $\beta_0/\sigma_0$ is sufficiently small, then
	\begin{equation*}
		\label{eq:det-concave-rate-x}
		\begin{aligned}
					&\min_{0\le k\le K}\|\mathcal G_k(x^k)\|^2
			=O(K^{-(1-s-2t)}),\\
			&\min_{0\le k\le K}
			\|(y^k,z^k)-(y_k^*(x^k),z_k^*(x^k))\|^2
			=O(K^{-(1-7t)}).
		\end{aligned}
	\end{equation*}
	Moreover,
	\[
	\liminf_{k\to\infty} \; \max\bigl\{
	\|\mathcal G_k(x^k)\|^2,
	\|(y^k,z^k)-(y_k^*(x^k),z_k^*(x^k))\|^2
	\bigr\}=0.
	\]
\end{theorem}

Along the common subsequence identified above, the stationarity residual for the smooth approximation vanishes, and the computed saddle variables approach the corresponding exact saddle points. The stationarity consistency result in Theorem~\ref{thm:Cstationarity} then yields the following conclusion for the original PBO.
\begin{theorem}
	\label{thm:sipba-subsequential-C}
	Let $\{(x^k,y^k,z^k)\}$ be generated by SiPBA under the
	assumptions of Theorem~\ref{thm:det-concave-rates}.
	Then there exists a subsequence $\{k_j\}$ such that
	\[
	\mathcal G_{k_j}(x^{k_j})\to0,\qquad
	(y^{k_j},z^{k_j})-
	\bigl(y_{k_j}^*(x^{k_j}),z_{k_j}^*(x^{k_j})\bigr)\to0.
	\]
	Let $(\bar x,\bar y)$ be any accumulation pair of
	$\bigl\{
	(x^{k_j},y_{k_j}^*(x^{k_j}))
	\bigr\}.$
	Suppose, in addition, that $X$ and $Y$ admit the explicit
	inequality representations specified in
	Subsection~\ref{subsec:stNationarity-preliminaries}.
	Assume further that the local assumptions of
		Theorem~\ref{thm:Cstationarity} are satisfied:
		$\mathcal S$ is inner semicontinuous at $\bar x$;
		$f$ is twice continuously differentiable in a neighborhood
		of $(\bar x,\bar y)$;
		upper-level MFCQ holds at $\bar x$; and lower-level LICQ
		and Assumption~\ref{ass:lower-error-bound} hold at
		$(\bar x,\bar y)$.
	Then, along the corresponding subsequence,
	\[
	x^{k_j}\to\bar x,\qquad
	(y^{k_j},z^{k_j})\to(\bar y,\bar y),\qquad
	\bar y\in\mathcal S_p(\bar x),
	\]
	and $\bar x$ is a C-stationary point of the original PBO.
\end{theorem}

\begin{proof}
	For brevity, write $u^k:=(y^k,z^k)$ and
	$u_k^*(x):=\bigl(y_k^*(x),z_k^*(x)\bigr)$ throughout this proof.
	Since Theorem~\ref{thm:det-concave-rates} gives
	$\liminf\limits_{k\to\infty}\max\{\|\mathcal G_k(x^k)\|^2,\|u^k-u_k^*(x^k)\|^2\}=0$,
	there exists a subsequence $\{k_j\}$ such that
	\[
	\mathcal G_{k_j}(x^{k_j})\to0,
	\qquad
	u^{k_j}-u_{k_j}^*(x^{k_j})\to0.
	\]
	Let $(\bar x,\bar y)$ be any accumulation pair of
	$\{(x^{k_j},y_{k_j}^*(x^{k_j}))\}$. Passing to a further subsequence, we have
	$x^{k_j}\to\bar x$ and $y_{k_j}^*(x^{k_j})\to\bar y$.
	The parameter subsequences
	$\{\rho_{k_j}\}$, $\{\sigma_{k_j}\}$, and
	$\{\delta_{k_j}\}$ still satisfy
	\eqref{paramcriterion}. The assumed inner semicontinuity of $\mathcal S$ and
		lower-level LICQ allow us to apply Lemma~\ref{multiplierconverge}, which gives
	$
	z_{k_j}^*(x^{k_j})\to\bar y,
	\;
	\bar y\in\mathcal S_p(\bar x).
	$
	Thus,
	$u_{k_j}^*(x^{k_j})\to(\bar y,\bar y).$
	Combining this limit with the vanishing tracking error gives
	$u^{k_j}\to(\bar y,\bar y).$
	
	It remains to establish C-stationarity of $\bar x$. Define
	\[	\widehat x_j
	:=
	\operatorname{Proj}_X(
	x^{k_j}
	-
	\alpha_{k_j}
	\nabla\phi_{k_j}(x^{k_j})
	).\]
	By the definition of $\mathcal G_{k_j}$,
	$
	\widehat x_j-x^{k_j}
	=
	-\alpha_{k_j}\mathcal G_{k_j}(x^{k_j}),$
	and hence $\widehat x_j\to\bar x$. The projection optimality
	condition gives
	\[
	\mathcal G_{k_j}(x^{k_j})
	-
	\nabla\phi_{k_j}(x^{k_j})
	\in N_X(\widehat x_j).
	\]
	Using Lemma~\ref{Lipshitzofphi} in Appendix, we obtain
	\begin{align}
		&
		\operatorname{dist}\left(
		0,
		\nabla\phi_{k_j}(\widehat x_j)
		+
		N_X(\widehat x_j)
		\right)
		\le
		\left(
		1+\alpha_{k_j}L_{\phi_{k_j}}
		\right)
		\|\mathcal G_{k_j}(x^{k_j})\|
		\to0,
		\label{eq:sto-as-normal-residual}
	\end{align}
	where $L_{\phi_k}=O(k^{3t})$ and hence
	$\alpha_kL_{\phi_k}=O(k^{3t-s})\to0$, since $s>7t$.
	
	Moreover, let $L_F$ and $L_f$ denote the Lipschitz constants of $\nabla F$ and $\nabla f$, respectively. Lemma~\ref{uxk+1xk} in Appendix gives
	\begin{align*}
		\left\|
		u_{k_j}^*(\widehat x_j)
		-
		u_{k_j}^*(x^{k_j})
		\right\|
		&\le
		\frac{L_F+2\rho_{k_j}L_f}
		{\bar\sigma_{k_j}}
		\alpha_{k_j}
		\|\mathcal G_{k_j}(x^{k_j})\|
		\to0,
	\end{align*}
	since
	$\frac{L_F+2\rho_kL_f}{\bar\sigma_k}\alpha_k
	=
	O(k^{2t-s}).$
	Consequently,
	$
	y_{k_j}^*(\widehat x_j)\to\bar y.$
	The error bound in Assumption~\ref{ass:lower-error-bound} is uniform
		for $x$ near $\bar x$ and therefore also applies at $\widehat x_j$ and
		$y_{k_j}^*(\widehat x_j)$ for all sufficiently large $j$.
		After reindexing the subsequence by $j$, all assumptions of
		Theorem~\ref{thm:Cstationarity} are satisfied by
	$\{\widehat x_j\}$ and the approximation sequence
	$\{\phi_{k_j}\}$. Therefore, $\bar x$ is a C-stationary point
	of the original PBO.
\end{proof}

\section{A Single-Loop Algorithm for Stochastic PBO}
\label{sec:stochastic-algorithm}

In this section, we extend SiPBA to stochastic PBO in which both the upper- and lower-level objectives are expressed as expectations:
\begin{equation}\label{stochasticpessimisticBLO}
	F(x,y)=\E_{\xi\sim D}[F(x,y;\xi)],
	\quad
	f(x,y)=\E_{\xi\sim D}[f(x,y;\xi)],
\end{equation}
where $D$ denotes the data distribution, and $F(x,y;\xi)$ and $f(x,y;\xi)$ are the corresponding sample realizations for $\xi\sim D$.

\subsection[Algorithm for Stochastic PBO]{Algorithm for Stochastic PBO}\label{subsec:stochastic-sipba}
The Stochastic Single-loop Pessimistic Bilevel Algorithm (Sto-SiPBA) retains the single-loop structure of SiPBA, combining stochastic saddle-point tracking with a recursive estimator of the upper-level gradient.

\textbf{Stochastic saddle-point tracking step.} To approximate $\nabla\phi_k(x^k)$, we track the associated saddle point using one stochastic projected ascent--descent step. At iteration $k$, we draw a sample $\xi_k^u$ and evaluate the stochastic directions at $(x^k,y^k,z^k)$:
\begin{equation}\label{direction_yz}
	d_y^k = \nabla_y \psi_k(x^k, y^k, z^k; \xi^u_k), \quad
	d_z^k = \nabla_z \psi_k(x^k, y^k, z^k; \xi^u_k)\textcolor{blue!70!black}{,}
\end{equation}
where $\psi_k(x, y, z; \xi):=F(x,y; \xi)
-\rho_k\bigl(f(x,y; \xi)-f(x,z; \xi)\bigr)	+\frac{\sigma_k}{2}\|z\|^2
-\sigma_k\langle y,z\rangle
-\frac{\delta_k}{2}\|y\|^2$is the sample realization of $\psi_k$.
The iterates are updated by
\begin{equation}\label{update_yz}
	\begin{aligned}
		y^{k+1}
		&=
		\operatorname{Proj}_Y
		\left(y^k+\beta_k d_y^k\right),\quad
		z^{k+1}
		=
		\operatorname{Proj}_Y
		\left(z^k-\beta_k d_z^k\right)\textcolor{blue!70!black}{.}
	\end{aligned}
\end{equation}

\textbf{Upper-level step with a variance-reduced estimator.}
We use the updated pair $(y^{k+1},z^{k+1})$ to construct the upper-level gradient estimator. Even at the exact saddle point $(y_k^*(x^k),z_k^*(x^k))$, a direct stochastic gradient estimate contains sampling noise whose variance need not vanish. To control this error across iterations, we use a STORM-type recursive estimator \cite{cutkosky2019momentum}. Specifically, we draw a fresh sample $\xi_k^x$ and set $d_x^0=\nabla_x\psi_0(x^0,y^{1},z^{1};\xi^x_0)$ and
\begin{equation}
	\begin{aligned}
		\label{direction_x}
		&d_x^k=\nabla_x\psi_k(x^k,y^{k+1},z^{k+1};\xi^x_k)
		+(1-\eta_k)(d_x^{k-1}-\nabla_x \psi_{k-1}(x^{k-1},y^{k},z^{k};\xi^x_k)),
	\end{aligned}
\end{equation}
for $k\ge1$.
Both gradient evaluations use the same fresh sample $\xi_k^x$, with $\psi_k$ and $\psi_{k-1}$ evaluated at their respective iterates and parameter values.

Finally, the upper-level variable is updated by
\begin{equation}\label{update_x}
	x^{k+1}=\mathrm{Proj}_X(x^k-\alpha_k d_x^k).
\end{equation}
The convergence analysis uses the strategies in Proposition~\ref{meritfunctionproposition} for the stepsizes $\alpha_k$ and $\beta_k$, the momentum coefficient $\eta_k$, and the approximation parameters $\rho_k$, $\sigma_k$, and $\delta_k$. The complete procedure is given in Algorithm~\ref{algorithm:ssipba}.

\begin{algorithm}[htbp]
	\caption{
		\textbf{Sto}chastic \textbf{Si}ngle-loop
		\textbf{P}essimistic \textbf{B}ilevel
		\textbf{A}lgorithm (\textbf{Sto-SiPBA})}
	\label{algorithm:ssipba}
	\begin{algorithmic}[1]
		\Require Initial points
		$(x^0,y^0,z^0)\in X\times Y\times Y$
		\For{$k=0,1,\dots,K-1$}
		\State Set the stepsizes $\alpha_k,\beta_k>0$, the momentum coefficient $\eta_k$, the penalty parameter $\rho_k>0$, and the regularization parameters $\sigma_k,\delta_k>0$ according to \eqref{par}.
		\State Sample $\xi_k^u$ and compute the stochastic directions $d_y^k,d_z^k$ using \eqref{direction_yz}.
		\State Update $y^{k+1}$ and $z^{k+1}$ using \eqref{update_yz}.
		\State Sample $\xi_k^x$ and compute $d_x^k$ using \eqref{direction_x}, reusing this sample for both gradient evaluations when $k\ge1$.
		\State Update $x^{k+1}$ using~\eqref{update_x}.
		\EndFor
	\end{algorithmic}
\end{algorithm}

\subsection{Assumptions on Stochastic Oracles}\label{subsec:stoch-assumptions}

We now specify the stochastic-oracle assumptions used in the convergence analysis. Assumption~\ref{ass:Xcompact} remains in force.
Let \(\F_k\) denote the \(\sigma\)-algebra generated by all samples observed
before iteration \(k\), and let \(\F_{k+\frac{1}{2}}\) further include the
current sample \(\xi_k^u\):
\[
\F_k = \sigma \{ \xi_0^u, \xi_0^x, \ldots, \xi_{k-1}^u, \xi_{k-1}^x \},
\qquad
\F_{k+\frac{1}{2}} = \sigma \{ \F_k, \xi_{k}^u\}.
\]
The next assumption requires fresh samples and conditionally unbiased gradient estimates with uniformly bounded variance.
\begin{assumption}\label{ass:oracle}
	For every $k\ge0$, $\xi_k^u$ is independent of $\mathcal F_k$, and $\xi_k^x$ is independent of $\mathcal F_{k+\frac12}$. Moreover, there exists a constant $\Delta>0$ such that the following conditions hold.
	\begin{itemize}
		\item For any square-integrable $\mathcal F_k$-measurable random vectors
		$x\in X$ and $v\in Y$, and for each $h\in\{F,f\}$,
		\begin{align*}
			&\E\left[\nabla_y h(x,v;\xi_k^u)\mid \mathcal F_k\right]
			=
			\nabla_y h(x,v),\\
			&\E\left[
			\|\nabla_y h(x,v;\xi_k^u)-\nabla_y h(x,v)\|^2
			\,\middle|\, \mathcal F_k
			\right]
			\le \Delta^2.
		\end{align*}
		
		\item For any square-integrable $\mathcal F_{k+\frac12}$-measurable random vectors
		$x\in X$ and $v\in Y$, and for each $h\in\{F,f\}$,
		\begin{align*}
			&\E\left[\nabla_x h(x,v;\xi_k^x)\mid \mathcal F_{k+\frac12}\right]
			=\nabla_x h(x,v),\\
			&\E\left[
			\|\nabla_x h(x,v;\xi_k^x)-\nabla_x h(x,v)\|^2
			\,\middle|\, \mathcal F_{k+\frac12}
			\right]
			\le \Delta^2.
		\end{align*}
	\end{itemize}
\end{assumption}

\begin{assumption}\label{ass:twopoint}
	The stochastic oracle permits two-point queries using the same sample $\xi_k^x$. Moreover, the stochastic $x$-gradients satisfy the following mean-square Lipschitz bounds for some constants $\bar L_F,\bar L_f>0$ and all pairs $(x_1,y_1),(x_2,y_2)\in X\times Y$:
	\begin{align*}
		&\E_{\xi}\!\left[\|\nabla_x F(x_1,y_1;\xi)-\nabla_x F(x_2,y_2;\xi)\|^2\right]
		\le \bar{L}_F^2\!\left(\|x_1-x_2\|^2+\|y_1-y_2\|^2\right),\\
		&\E_{\xi}\!\left[\|\nabla_x f(x_1,y_1;\xi)-\nabla_x f(x_2,y_2;\xi)\|^2\right]
		\le \bar{L}_f^2\!\left(\|x_1-x_2\|^2+\|y_1-y_2\|^2\right)
	\end{align*}
\end{assumption}

Let $L_F$ and $L_f$ denote the Lipschitz constants of
$\nabla F$ and $\nabla f$ on $X\times Y$, respectively. Without loss of generality, we enlarge \(L_F,L_f\) so that
\(L_F\ge \bar L_F\) and \(L_f\ge \bar L_f\).
Define the stochastic error in the upper-level gradient estimator by $e_x^k:=d_x^k-\nabla_x\psi_k(x^k,y^{k+1},z^{k+1})$. This error is measured relative to the exact gradient at the updated saddle-point variables.

Under the parameter monotonicity conditions used below,
Lemma~\ref{lem:bounded-saddles-monotone} in
Appendix~\ref{appendix:stochastic-algorithm} gives uniform boundedness of
$(y_k^*(x),z_k^*(x))$ over $k$ and $x\in X$.
Compactness of $X$ and continuity of $f$ and $\nabla f$ therefore ensure
that the following constants are finite:
\[
\begin{aligned}
	M_y
	&:= \sup_{k,\ x\in X}
	\max\big\{\|y_k^*(x)\|,\|z_k^*(x)\|\big\},\\
	M_f
	&:= \sup_{k,\ x\in X}
	\max\big\{|f(x,y_k^*(x))|,\ |f(x,z_k^*(x))|\big\},\\
	M_{\nabla f}
	&:= \sup_{k,\ x\in X}
	\max\big\{\|\nabla f(x,y_k^*(x))\|,\ \|\nabla f(x,z_k^*(x))\|\big\}.
\end{aligned}
\]

\subsection{Auxiliary Lemmas}
This subsection establishes bounds on the saddle-point tracking error, the change in the smooth approximation value, and the mean-square error of the upper-level gradient estimator. The proofs are given in Appendix~\ref{appendix:stochastic-algorithm}.
For notational convenience in the convergence analysis, write $u^k:=(y^k,z^k)$ and $u_k^*(x):=\bigl(y_k^*(x),z_k^*(x)\bigr)$.

The following lemma bounds the conditional mean-square tracking error after one stochastic ascent--descent step, relative to the fixed target $u_k^*(x^k)$, by a contraction term and a sampling-noise term.
\begin{lemma}\label{udescent}
	Let $\{\rho_k\}$, $\{\sigma_k\}$  and $\{\delta_k\}$ be sequences such that $\rho_{k+1} \ge \rho_{k}\ge 0$, $\sigma_{k} \ge\sigma_{k+1}>0$ and $\delta_{k}\ge\delta_{k+1}>0$. Define $\bar{\sigma}_k = \min\{\sigma_k, \delta_k\}$. Suppose the step-size sequence $\{\beta_k\}$ satisfies $0<\beta_k<\frac{\bar{\sigma}_k}{(L_F + \rho_k L_f + 2\sigma_k+\delta_{k})^2}$ for each $k$. Let $\{(x^k, y^k, z^k)\}$ be the sequence generated by Sto-SiPBA. Then
	\begin{equation}\label{udescentequ}
		\E[\|u^{k+1}- u_k^*(x^k)\|^2\mid\F_k]\le(1-\bar{\sigma}_k \beta_k)\|u^{k}-u_k^*(x^k)\|^2+(1+2\rho_k+2\rho_k^2)\beta_k^2\Delta^2.
	\end{equation} 
\end{lemma}

The saddle point changes with both $x^k$ and the parameters $\rho_k,\sigma_k,\delta_k$. Combining Lemma~\ref{udescent} with the sensitivity estimates from Lemmas~\ref{uxk+1xk} and~\ref{uk+1uk}  in Appendix~\ref{appendix:stochastic-algorithm} accounts for these changes and yields the following bound relative to  $u_{k+1}^*(x^{k+1})$.
\begin{lemma}\label{udescentlemma}
	Let $\{\rho_k\}$, $\{\sigma_k\}$  and $\{\delta_k\}$ be sequences such that $\rho_{k+1} \ge \rho_{k}>0$, $\sigma_{k} \ge\sigma_{k+1}>0$ and $\delta_{k}\ge\delta_{k+1}>0$. Define $\bar{\sigma}_k = \min\{\sigma_k, \delta_k\}$. Suppose the step-size sequence $\{\beta_k\}$ satisfies $0<\beta_k<\frac{\bar{\sigma}_k}{(L_F + \rho_k L_f + 2\sigma_k+\delta_k)^2}$ for each $k$. Let $\{(x^k, y^k, z^k)\}$ be the sequence generated by Sto-SiPBA.
	\textcolor{blue!70!black}{Then}
	\begin{equation}\label{udescentlemmaeq}
		\begin{aligned}
			& \E[\|u^{k+1}- u_{k+1}^*(x^{k+1})\|^2\mid \F_k]-\|u^{k}-u_{k}^*(x^{k})\|^2 \\\le&-\frac{1}{2}\beta_k\bar{\sigma}_k \|u^{k}-u_{k}^*(x^{k})\|^2 + 2(1+\frac{2}{\beta_k\bar{\sigma}_k})\frac{(L_{F}+2\rho_kL_{f})^2}{\bar{\sigma}_k^2}\E[\|x^{k+1}-x^k\|^2\mid \F_k] \\
			&
			+2(1+\frac{2}{\beta_k\bar{\sigma}_k}) \left( \frac{12(\rho_{k+1} - \rho_k)^2}{\bar{\sigma}_k^2}M_{\nabla f}^2 + \frac{27(\sigma_{k} - \sigma_{k+1})^2}{\bar{\sigma}_k^2}M_y^2\right.\\
			&\left.+\frac{3(\delta_{k} - \delta_{k+1})^2}{\bar{\sigma}_k^2}M_y^2 \right)+(1+\frac{1}{2}\beta_k \bar{\sigma}_k)(1+2\rho_k+2\rho_k^2)\beta_k^2 \Delta^2.
		\end{aligned}
	\end{equation}
\end{lemma}

The next lemma bounds the change in the smooth approximation value after the upper-level update, accounting for the tracking error, the stochastic estimation error, and the parameter changes.
\begin{lemma}\label{xdescentlemma}
	Let $\{\rho_k\}$, $\{\sigma_k\}$  and $\{\delta_k\}$ be sequences such that $\rho_{k+1} \ge \rho_{k}>0$, $\sigma_{k} \ge\sigma_{k+1}>0$ and $\delta_{k}\ge\delta_{k+1}>0$. Define $\bar{\sigma}_k = \min\{\sigma_k, \delta_k\}$. Suppose the step-size sequence $\{\beta_k\}$ satisfies $0<\beta_k<\frac{\bar{\sigma}_k}{(L_F + \rho_k L_f + 2\sigma_k+\delta_k)^2}$ for each $k$. Let $\{(x^k, y^k, z^k)\}$ be the sequence generated by Sto-SiPBA.  Then
	\begin{equation}\label{xdescentlemma_eq}
		\begin{aligned}
			& \E[\phi_{k+1}(x^{k+1})\mid \F_k]-\phi_{k}(x^k) + \left( \frac{1}{4\alpha_k} - \frac{L_{\phi_k}}{2}\right) \E[\|x^{k+1}-x^{k}\|^2\mid \F_k] \\ \le\,& \frac{\alpha_k}{2} (L_F + 2\rho_k L_f)^2 (1-\bar{\sigma}_k \beta_k) \| u^{k}- u_{k}^*(x^k)\|^2+\alpha_k\E[\|e_x^k\|^2\mid \F_k]  \\
			&+ \left( \sigma_k - \sigma_{k+1} \right)\frac{M_y^2 }{2} + 2\left(\rho_{k+1}-\rho_{k}\right)M_f+ \left( \delta_k - \delta_{k+1} \right)\frac{M_y^2 }{2} \\
			&+\frac{\alpha_k}{2} (L_F + 2\rho_k L_f)^2(1+2\rho_k+2\rho_k^2)\beta_k^2\Delta^2,
		\end{aligned}
	\end{equation}
	where $L_{\phi_k}:= \frac{(L_F + 2\rho_kL_f)(L_{F}+2\rho_kL_{f} + \bar{\sigma}_k)}{\bar{\sigma}_k} $ is a Lipschitz constant for $\nabla\phi_k$, and $e_x^k$ is the stochastic estimation error defined above.
\end{lemma}

Finally, the next lemma gives a recursion for the conditional mean-square estimation error, accounting for sampling noise and changes in the iterates and parameters.

\begin{lemma}\label{variancereduction}
	Let $\{(x^k,y^k,z^k)\}$ be the sequence generated by Sto-SiPBA.
	Suppose that $\{\rho_k\}$, $\{\sigma_k\}$  and $\{\delta_k\}$ satisfy
	$\rho_{k+1}\ge \rho_k>0$, $\sigma_k\ge\sigma_{k+1}>0$ and $\delta_k\ge\delta_{k+1}>0$, that
	$0<\beta_k<\frac{\bar{\sigma}_k}{(L_F + \rho_k L_f + 2\sigma_k+\delta_k)^2}$, and that $0\le \eta_k\le 1$. Then, for every
	$k>0$,
	\[
	\begin{aligned}
		&\E\!\left[\|e_x^k\|^2\mid \F_k\right]-\|e_x^{k-1}\|^2\\
		\le\;&	(-2\eta_k+\eta_k^2)	\|e_x^{k-1}\|^2
		+(6+12\rho_k^2)\eta_k^2\Delta^2
		+16(\rho_k-\rho_{k-1})^2\Delta^2 \\
		&
		+36\beta_{k}^2L_k^2[(1+2\rho_k+2\rho_k^2)\Delta^2+(L_F+\rho_{k}L_f+2\sigma_{k}+\delta_k )^2\|u^{k}-u_{k}^*(x^{k})\|^2\\
		&+ \rho_{k}^2 M_{T}]+12L_k^2\|x^k-x^{k-1}\|^2,
	\end{aligned}
	\]
	where $L_k:= L_F+2\rho_kL_f$ and $M_T>0$ is a constant independent of $k$.
\end{lemma}

\subsection{Convergence Analysis}
We first derive nonasymptotic bounds on the expected squared projected-gradient residual and saddle-point tracking error. We then establish, under the stated regularity assumptions, almost-sure subsequential convergence to C-stationary points of the original PBO.

To combine the preceding estimates, we define, for $k\ge1$, the augmented merit function with positive coefficients $a_k,b_k,c_k,d_k$:
\begin{align}\label{V_def2}
	V_k = a_k(\phi_k(x^k)- \underline{\phi}) + b_k\| u^{k}- u_{k}^*(x^k)\|^2+c_k\|e_x^{k-1}\|^2+d_k\|x^{k}-x^{k-1}\|^2,
\end{align}
Here, $\underline\phi$ is a uniform lower bound satisfying $\phi_k(x^k)\ge\underline\phi$ for all $k$. Under the assumption that $\phi$ is bounded below on $X$, such a bound exists by Lemma~\ref{lem:det-concave-lower-bound} in Appendix~\ref{appendix:stochastic-algorithm}. The four terms combine the shifted objective value, the squared tracking error, the previous squared estimation error, and the squared upper-level step, respectively. 
In particular, $V_k\ge0$.

The following proposition gives a conditional descent inequality for $V_k$ under suitable stepsize and parameter schedules.

\begin{proposition}\label{meritfunctionproposition}
	Let $\{(x^k,y^k,z^k)\}$ be the sequence generated by Sto-SiPBA with 
	\begin{equation}\label{par}
		\begin{aligned}
			&\alpha_k=\alpha_0(k+1)^{-9t-s},\;
			\beta_k=\beta_0(k+1)^{-4t-s},\; 	\eta_k=\eta_0(k+1)^{-5t-s},
			\\
			&	\sigma_k=\sigma_0(k+1)^{-t},\; \delta_k=\delta_0(k+1)^{-t},\;
			\rho_k=\rho_0(k+1)^t,\;
		\end{aligned}
	\end{equation}
	with $\alpha_0, \beta_0, \sigma_0, \delta_0,\rho_0,  t >0$ and $1\ge\eta_0>0$. Suppose that the function $\phi(x)$ is bounded below on the set $X$. Let the coefficients in the merit function $V_k$ be
	$a_k=(k+1)^{-2t},\; b_k=(k+1)^{-3t},\; c_k=(k+1)^{-5t},\; d_k=(k+1)^{-2t}.$
	Assume that 
	$0<t,s<1$ satisfy $11t+s<1$. 
	If $\beta_0/\sigma_0$ is sufficiently small, then for all sufficiently large $k$,
	\begin{equation}\label{meritfunctionproposition_eq}
		\begin{aligned}
			&\E[V_{k+1}\mid \F_k]-V_k\\
			\le\;&
			-\frac{a_k\alpha_k}{24}\|\mathcal G_k(x^k)\|^2
			-\frac{b_k\beta_k\bar\sigma_k}{4} \|u^k-u_k^*(x^k)\|^2
			+Ck^{-9t-2s}+\zeta_k,
		\end{aligned}
	\end{equation}
	where $V_k$ is defined in \eqref{V_def2}, $\bar\sigma_k:=\min\{\sigma_k,\delta_k\}$, $C>0$ is independent of $k$, and $\{\zeta_k\}$ is a nonnegative summable sequence, i.e., $\sum_{k=0}^{\infty}\zeta_k<\infty$.
\end{proposition}

Summing the descent inequality in Proposition~\ref{meritfunctionproposition} and substituting the parameter schedules yields the following bounds on the minimum expected squared projected-gradient residual and tracking error.
\begin{theorem}\label{thm:nonasymtoticconbergence}
	Let $\{(x^k,y^k,z^k)\}$ be the sequence generated by Sto-SiPBA with parameters
	selected as in \eqref{par}. Suppose that the function $\phi(x)$ is bounded below
	on $X$. Assume further that
	$0<2t<s<1,\; 11t+s<1$ and $9t+2s\neq 1$.
	If $\beta_0/\sigma_0$ is sufficiently small, then
	\begin{align*}
		&\min_{0\le k\le K}\E[\|\mathcal G_k(x^k)\|^2]
		=
		O\!\left(\frac{1}{K^{1-11t-s}}\right)
		+
		O\!\left(\frac{1}{K^{s-2t}}\right),\\
		&\min_{0\le k\le K}\E[\|u^k-u_k^*(x^k)\|^2]
		=
		O\!\left(\frac{1}{K^{1-8t-s}}\right)
		+
		O\!\left(\frac{1}{K^{s+t}}\right).
	\end{align*}
\end{theorem}	
\begin{proof}
	By Proposition~\ref{meritfunctionproposition}, for all sufficiently large $k$,
	\[
	\E[V_{k+1}\mid\F_k]-V_k
	\le
	-\frac{a_k\alpha_k}{24}\|\mathcal G_k(x^k)\|^2
	-\frac14 b_k\beta_k\bar\sigma_k\|u^k-u_k^*(x^k)\|^2
	+Ck^{-9t-2s}+\zeta_k .
	\]
	Taking expectations and rearranging gives
	\[
	\frac{a_k\alpha_k}{24}\E[\|\mathcal G_k(x^k)\|^2]
	+
\frac{b_k\beta_k\bar\sigma_k}{4}
	\E[\|u^k-u_k^*(x^k)\|^2]
	\le
	\E[V_k]-\E[V_{k+1}]
	+Ck^{-9t-2s}+\zeta_k .
	\]
	Summing from a sufficiently large index $k_0\ge1$ to $K$ and using $V_{K+1}\ge0$ yields
	\[
	\begin{aligned}
		&\sum_{k=k_0}^{K}
		\frac{a_k\alpha_k}{24}\E[\|\mathcal G_k(x^k)\|^2]
		+
		\sum_{k=k_0}^{K}
		\frac14 b_k\beta_k\bar\sigma_k
		\E[\|u^k-u_k^*(x^k)\|^2] \\
		\le\;&
		\E[V_{k_0}]
		+
		C\sum_{k=k_0}^{K}k^{-9t-2s}
		+
		\sum_{k=k_0}^{K}\zeta_k .
	\end{aligned}
	\]
	Since $\sum_{k=0}^{\infty}\zeta_k<\infty$, we have
	$\sum_{k=k_0}^{K}\zeta_k=O(1).$
	Since $9t+2s\neq1$, we also have $\sum_{k=k_0}^{K}k^{-9t-2s}=O(K^{1-9t-2s})+O(1)$. Absorbing fixed factors into a constant $C>0$, independent of $K$, gives
	\[
	\sum_{k=k_0}^{K}
	a_k\alpha_k\E[\|\mathcal G_k(x^k)\|^2]
	+
	\sum_{k=k_0}^{K}
	b_k\beta_k\bar\sigma_k
	\E[\|u^k-u_k^*(x^k)\|^2]
	\le
	C\bigl(K^{1-9t-2s}+1\bigr).
	\]
	The chosen weights satisfy $a_k\alpha_k=\alpha_0(k+1)^{-11t-s}$.
	Since $11t+s<1$, we have
	$\sum_{k=k_0}^{K}a_k\alpha_k
	=
	\Theta(K^{1-11t-s}).$
	Therefore,
	\[
	\begin{aligned}
		\min_{k_0\le k\le K}\E[\|\mathcal G_k(x^k)\|^2]
		&\le
		\frac{
			\sum_{k=k_0}^{K}a_k\alpha_k\E[\|\mathcal G_k(x^k)\|^2]
		}{
			\sum_{k=k_0}^{K}a_k\alpha_k
		}\\
		&=
		O\!\left(\frac{1}{K^{1-11t-s}}\right)
		+
		O\!\left(\frac{1}{K^{s-2t}}\right).
	\end{aligned}
	\]
	Similarly, $b_k\beta_k\bar\sigma_k=\Theta(k^{-8t-s})$ as $k\to\infty$.
	Because $11t+s<1$ implies $8t+s<1$, we also have $\sum_{k=k_0}^{K}b_k\beta_k\bar\sigma_k
	=
	\Theta(K^{1-8t-s}).$
	Hence,
	\[
	\begin{aligned}
		\min_{k_0\le k\le K}
		\E[\|u^k-u_k^*(x^k)\|^2]
		&\le
		\frac{
			\sum_{k=k_0}^{K}
			b_k\beta_k\bar\sigma_k
			\E[\|u^k-u_k^*(x^k)\|^2]
		}{
			\sum_{k=k_0}^{K}b_k\beta_k\bar\sigma_k
		}\\
		&=
		O\!\left(\frac{1}{K^{1-8t-s}}\right)
		+
		O\!\left(\frac{1}{K^{s+t}}\right).
	\end{aligned}
	\]
	Finally, extending the minimization range to $0\le k\le K$ can only decrease each minimum, so
	\[
	\min_{0\le k\le K}\E[\|\mathcal G_k(x^k)\|^2]
	=
	O\!\left(\frac{1}{K^{1-11t-s}}\right)
	+
	O\!\left(\frac{1}{K^{s-2t}}\right),
	\]
	\[
	\min_{0\le k\le K}\E[\|u^k-u_k^*(x^k)\|^2]
	=
	O\!\left(\frac{1}{K^{1-8t-s}}\right)
	+
	O\!\left(\frac{1}{K^{s+t}}\right).
	\]
	This completes the proof.
\end{proof}

Theorem~\ref{thm:nonasymtoticconbergence} bounds the minimum expected value of each squared residual. The joint descent inequality in Proposition~\ref{meritfunctionproposition} further ensures that, almost surely, both residuals vanish along a common subsequence. This leads to the following  result of C-stationarity.

\begin{theorem}
	\label{thm:sto-as-local-Cstationarity}
	Let $\{(x^k,y^k,z^k)\}$ be generated by Sto-SiPBA under the
	assumptions of Theorem~\ref{thm:nonasymtoticconbergence}.
	Then, almost surely, there exists a sample-path-dependent
	subsequence $\{k_j\}$ such that
	\[
	\mathcal G_{k_j}(x^{k_j})\to0,
	\qquad
	u^{k_j}-u_{k_j}^*(x^{k_j})\to0.
	\]
	On this probability-one event, let $(\bar x,\bar y)$ be any
	accumulation pair of the sequence
	$\bigl\{
	(x^{k_j},y_{k_j}^*(x^{k_j}))
	\bigr\}.$
	Suppose, in addition, that $X$ and $Y$ admit the explicit
	inequality representations specified in
	Subsection~\ref{subsec:stNationarity-preliminaries}.
	Assume further that the local assumptions of
		Theorem~\ref{thm:Cstationarity} are satisfied:
		$\mathcal S$ is inner semicontinuous at $\bar x$;
		$f$ is twice continuously differentiable in a neighborhood
		of $(\bar x,\bar y)$;
		upper-level MFCQ holds at $\bar x$; and lower-level LICQ
		and Assumption~\ref{ass:lower-error-bound} hold at
		$(\bar x,\bar y)$. Then, along a further subsequence, relabeled $\{k_j\}$,
	\[
	x^{k_j}\to\bar x,
	\qquad
	(y^{k_j},z^{k_j})\to(\bar y,\bar y),
	\qquad
	\bar y\in\mathcal S_p(\bar x),
	\]
	and $\bar x$ is a C-stationary point of the original PBO.
\end{theorem}

\begin{proof}
	Define the joint residual
	$
	R_k
	:=
	\|\mathcal G_k(x^k)\|^2
	+
	\|u^k-u_k^*(x^k)\|^2.$
	Taking expectations in \eqref{meritfunctionproposition_eq}, summing from a sufficiently large index $k_0\ge1$ to $K$, and using $V_k\ge0$ and $\sum_k\zeta_k<\infty$, we obtain
	\begin{equation}
		\label{eq:sto-as-combined-bound}
		\sum_{k=k_0}^{K}
		k^{-11t-s}\E[R_k]
		\le
		C\left(
		1+\sum_{k=k_0}^{K}k^{-9t-2s}
		\right).
	\end{equation}
	Here we used $a_k\alpha_k=\Theta(k^{-11t-s})$ and $b_k\beta_k\bar\sigma_k=\Theta(k^{-8t-s})$ under \eqref{par}, together with $k^{-8t-s}\ge k^{-11t-s}$ for $k\ge1$, to give both residuals the common weight $k^{-11t-s}$.Here, $\Theta(\cdot)$ denotes two-sided bounds up to positive constant factors as $k\to\infty$.
	
	Set
	$
	q:=11t+s,
	$ and $
	r:=9t+2s.
	$
	The assumptions give
	$
	q<1,
	\;
	r-q=s-2t>0.$
	Hence
	\[
	\frac{
		1+\sum_{k=k_0}^{K}k^{-r}
	}{
		\sum_{k=k_0}^{K}k^{-q}
	}
	\to0.
	\]
	Dividing \eqref{eq:sto-as-combined-bound} by
	$\sum_{k=k_0}^{K}k^{-q}$ shows that the corresponding
	weighted average of $\{\E[R_k]\}$ converges to zero. Therefore,
	\[
	\liminf_{k\to\infty}\E[R_k]=0.
	\]
	Since $R_k\ge0$, Fatou's lemma yields
	\[
	0
	\le
	\E\left[\liminf_{k\to\infty}R_k\right]
	\le
	\liminf_{k\to\infty}\E[R_k]
	=0.
	\]
	Consequently,
	\[
	\liminf_{k\to\infty}R_k=0
	\qquad\text{almost surely}.
	\]
	
	Fix a sample path in this probability-one event. There exists a
	sample-path-dependent subsequence $\{k_j\}$ such that
	$R_{k_j}\to0$, and hence
	\[
	\mathcal G_{k_j}(x^{k_j})\to0,
	\qquad
	u^{k_j}-u_{k_j}^*(x^{k_j})\to0.
	\]
	Fix an accumulation pair $(\bar x,\bar y)$ as in the theorem and pass to a further subsequence converging to it. The parameter subsequences still satisfy \eqref{paramcriterion}. Moreover, \eqref{par} gives $\alpha_kL_{\phi_k}=O(k^{-6t-s})\to0$ and $\alpha_k(L_F+2\rho_kL_f)/\bar\sigma_k=O(k^{-7t-s})\to0$. Assumption~\ref{ass:lower-error-bound} is uniform in the upper-level parameter,
		so it also holds along the projected sequence used in the deterministic
		argument. Its neighborhood and constant may be chosen for the fixed
		sample path and accumulation pair. Thus, the remaining arguments in the
		proof of Theorem~\ref{thm:sipba-subsequential-C} apply on this sample path. They yield $u^{k_j}\to(\bar y,\bar y)$ with $\bar y\in\mathcal S_p(\bar x)$ and, by Theorem~\ref{thm:Cstationarity}, C-stationarity of $\bar x$.
	
\end{proof}

\section{Empirical Studies}\label{sec:numericalexperiments}
The experiments are designed to test four distinct aspects of the proposed framework. The deterministic synthetic example evaluates the accuracy and scalability of SiPBA against existing PBO solvers. The stochastic synthetic example isolates the effect of sampling noise on Sto-SiPBA. The spam classification experiment examines whether pessimistic modeling improves cross-domain robustness. The smart predict-then-optimize experiment tests Sto-SiPBA on a stochastic decision-focused inventory planning problem with a genuinely set-valued downstream solution map. All computational experiments were performed on a server
provisioned with dual Intel Xeon Gold 5218R CPUs (a total of 40 cores/80 threads, with 2.1-4.0
GHz) and an NVIDIA H100 GPU. 
The code is available on \url{https://github.com/qichaosustech/single-loop-pbo}.

\subsection{Synthetic Example}
\textbf{Deterministic case.}
We first consider a deterministic test problem for which both the pessimistic lower-level response and the global upper-level solution can be computed explicitly:
\begin{align}\label{constrainedtoy}
	&\min_{x\in X}\max_{y \in \mathbb{R}^n}\|x-e\|^2-\sqrt{n}\|y-e\|,
	\quad\text{s.t.}\;
	y\in\underset{y'\in Y}{\arg\min}
	\frac{1}{n}\bigl|\langle y',e\rangle-\|x\|^2\bigr|^2 .
\end{align}
Here $X=Y=[-0.9,0.9]^n$, and $e\in\mathbb{R}^n$ denotes the all-ones vector. For every fixed $x\in X$, the lower-level solution set is nonempty, and the pessimistic response is given by $y^*(x)=\frac{\|x\|^2}{n}e .$
The corresponding upper-level solution is $x^*=\frac{e}{2}$. Hence we report the normalized upper-level and lower-level errors
$\frac{\|x^k-x^*\|^2}{n},
\;
\frac{\|y^k-y^*(x^k)\|^2}{n},$
respectively.

\begin{figure}[htbp]
	\centering
	\begin{subfigure}[t]{0.48\linewidth}
		\centering
		\includegraphics[width=\linewidth]{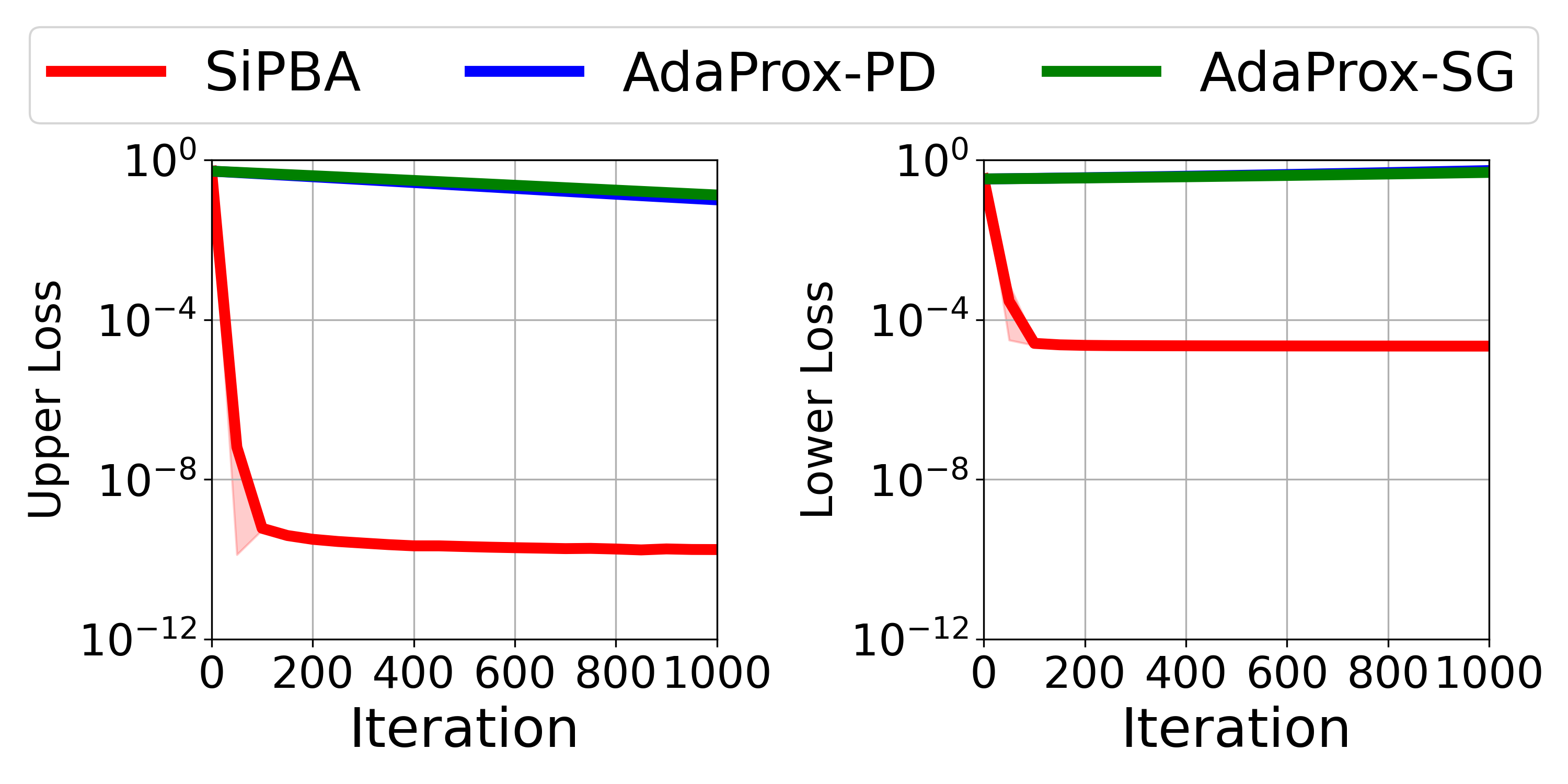}
		\caption{Deterministic case.}
		\label{fig:deterministic_compare}
	\end{subfigure}
	\hfill
	\begin{subfigure}[t]{0.48\linewidth}
		\centering
		\includegraphics[width=\linewidth]{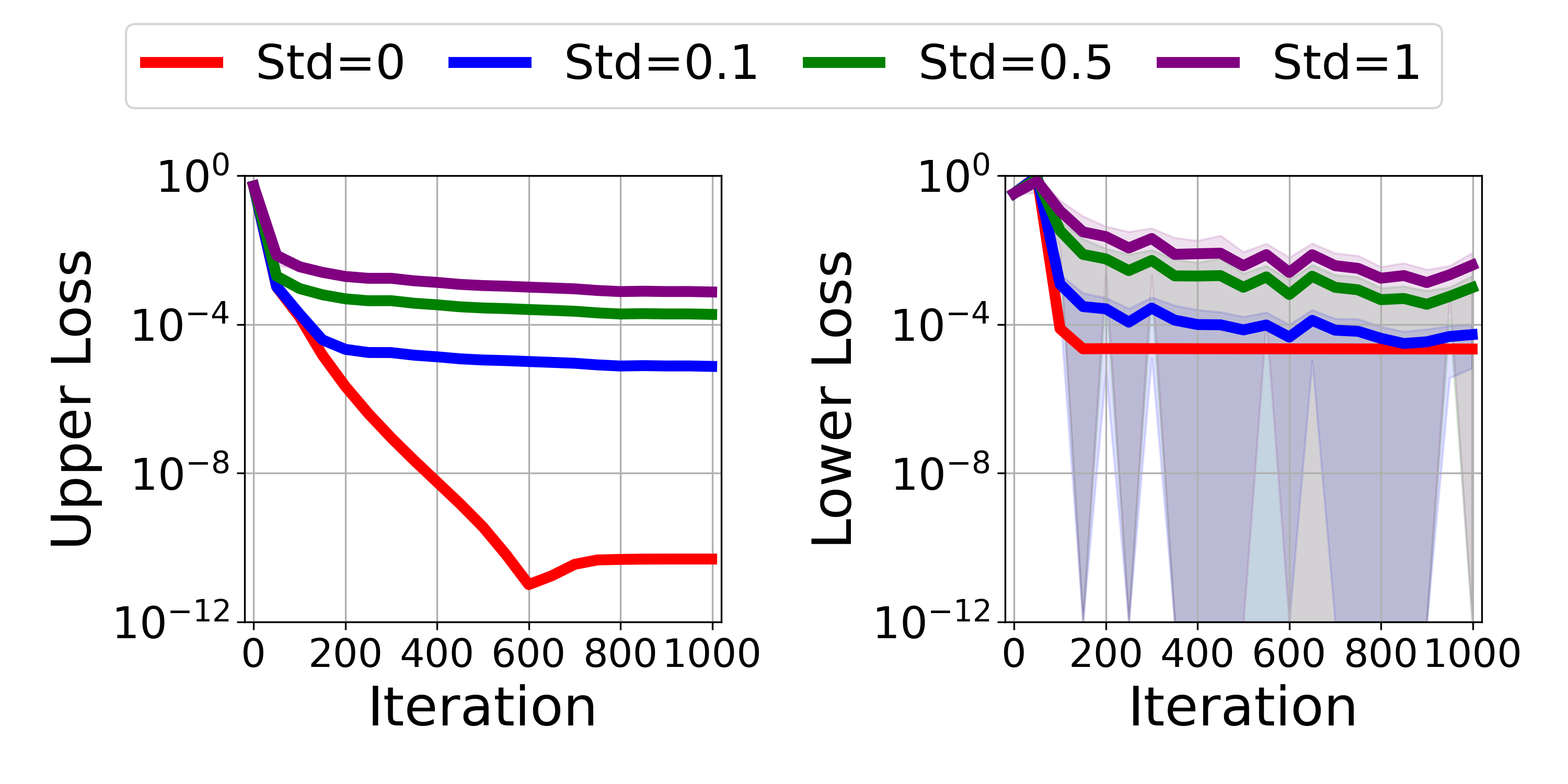}
		\caption{Stochastic case.}
		\label{fig:stochastic_convergence}
	\end{subfigure}
	\caption{Convergence behavior on the synthetic examples. 
		(a) Upper- and lower-level losses of the gradient-based methods in the deterministic setting. 
		(b) Upper- and lower-level losses of Sto-SiPBA under different noise levels in the stochastic setting.}
	\label{fig:synthetic_convergence}
\end{figure}

We compare SiPBA with two gradient-based bilevel methods, AdaProx-PD and AdaProx-SG \cite{guanadaprox}, and with two MPCC-based Scholtes relaxation methods, Compact Scholtes and Detailed Scholtes \cite{benchouk2025scholtes}, denoted by Scholtes-C and Scholtes-D. SiPBA, AdaProx-PD, and AdaProx-SG are run for 1000 iterations, while Scholtes-C and Scholtes-D are run for 10 outer iterations. Unless otherwise stated, the parameters of SiPBA are chosen as in \eqref{deter:par}: $s=0.08$, $t=0.01$, $\alpha_0=0.1$, $\beta_0=0.01$, $\rho_0=10$, and $\sigma=\delta=10^{-4}$.

\begin{table}[htbp]
	\vspace{-1em}
	\caption{Performance comparison of the SiPBA, AdaProx-PD, AdaProx-SG, Scholtes-C, and Scholtes-D with $n=100$.}
	\label{tab:dimension100}
	\footnotesize
	\setlength{\tabcolsep}{3pt}
	\begin{tabular*}{\textwidth}{@{\extracolsep{\fill}}lccccc}
		\hline
		& SiPBA & AdaProx-PD & AdaProx-SG & Scholtes-C & Scholtes-D \\
		\hline
		Upper loss (Min.) & \textbf{1.627e-10} & 8.229e-02 & 1.069e-01 & 1.846e-05 & 1.818e-05 \\
		Upper loss (Max.) & \textbf{1.767e-10} & 1.109e-01 & 1.440e-01 & 1.951e-05 & 1.820e-05 \\
		Lower loss (Min.) & 2.175e-05 & 5.061e-01 & 4.526e-01 & 1.632e-10 & \textbf{6.817e-11} \\
		Lower loss (Max.) & 2.178e-05 & 6.173e-01 & 5.550e-01 & 3.031e-09 & \textbf{1.212e-10} \\
		Ave Time (s) & \textbf{1.613} & 8.241 & 7.534 & 26.081 & 107.723 \\
		\hline
	\end{tabular*}
\end{table}
\begin{table}[htbp]
	\caption{Performance comparison of the SiPBA, AdaProx-PD, AdaProx-SG, Scholtes-C, and Scholtes-D with $n=1000$.}
	\label{tab:dimension1000}
	\footnotesize
	\setlength{\tabcolsep}{3pt}
	\begin{tabular*}{\textwidth}{@{\extracolsep{\fill}}lccccc}
		\hline
		& SiPBA & AdaProx-PD & AdaProx-SG & Scholtes-C & Scholtes-D \\
		\hline
		Upper loss (Min.) & \textbf{1.229e-10} & 1.103e-01 & 1.277e-01 & 5.012e-01 & 5.012e-01 \\
		Upper loss (Max.) & \textbf{1.996e-10} & 1.186e-01 & 1.380e-01 & 5.480e-01 & 5.480e-01 \\
		Lower loss (Min.) & \textbf{2.173e-05} & 5.041e-01 & 4.764e-01 & 3.315e-01 & 3.315e-01 \\
		Lower loss (Max.) & \textbf{2.184e-05} & 5.745e-01 & 5.426e-01 & 3.769e-01 & 3.769e-01 \\
		Ave Time (s) & \textbf{1.907} & 9.773 & 8.901 & 400.235 & 400.239 \\
		\hline
	\end{tabular*}
\end{table}
Tables~\ref{tab:dimension100} and \ref{tab:dimension1000} report the minimum and maximum errors and the average running time for $n=100$ and $n=1000$, respectively. Figure~\ref{fig:deterministic_compare} shows the corresponding convergence histories for the gradient-based methods. In both dimensions, SiPBA gives the smallest upper-level error among the tested methods. The Scholtes relaxations produce very small lower-level errors for $n=100$, but for $n=1000$ both variants reach the prescribed 400-second time limit and return substantially larger upper- and lower-level errors. These observations suggest that, on this controlled example, the smooth approximation used by SiPBA can recover accurate upper-level iterates without solving each auxiliary subproblem to high accuracy.

\textbf{Stochastic case.}
We next consider a stochastic counterpart of \eqref{constrainedtoy}:
\begin{equation} \label{toy_problem}
	\begin{aligned}
		\min_{x \in X} \max_{y \in \mathbb{R}^n} & \quad \mathbb{E}_{w} \left[ \|x+w-e\|^2 - \sqrt{n}\|y-e\| \right] \\
		\text{s.t. } & \quad y \in \arg\min_{y^\prime \in Y} \mathbb{E}_{v} \left[\frac{1}{n}\left\| \langle y^\prime+v, e \rangle - \|x\|^2 \right\|^2\right],
	\end{aligned}
\end{equation}
where $X = Y = [-0.9,0.9]^n$. The entries of $w$ and $v$ are independent Gaussian random variables with
$w_i,v_i\overset{\mathrm{i.i.d.}}{\sim}\mathcal{N}(0,\Delta^2)$. Since the stochastic terms add constants to the expected objectives, the pessimistic response and the upper-level solution remain $y^*(x)=\frac{\|x\|^2}{n}e,\;x^*=\frac{e}{2}.$
We therefore use the same normalized errors as in the deterministic case.

Unless otherwise stated, the parameters are chosen as in \eqref{par}: $s=0.5$, $t=0.01$, $\alpha_0=0.1$, $\beta_0=0.1$, $\rho_0=10$, and $\sigma=\delta=10^{-4}$. Figure~\ref{fig:stochastic_convergence} reports the convergence behavior of Sto-SiPBA for several values of the noise level $\Delta$. The curves show that the upper-level error decreases consistently over the tested noise levels, while the lower-level error remains stable up to the stochastic fluctuations induced by the sampling procedure.

Finally, we examine the sensitivity of Sto-SiPBA to selected algorithmic parameters. The noise level is fixed at $\Delta=0.1$. For each parameter configuration, we record the number of iterations and the running time required for the aggregate error
$\max\{\frac{\|x^k-x^*\|^2}{n},\frac{\|y^k-y^*(x^k)\|^2}{n}\}$ to fall below $10^{-4}$. Table~\ref{ablationanalysis} reports the mean and standard deviation over repeated runs. Within the tested range, the choices of $\sigma$, $\delta$, and $\eta$ have limited influence on the stopping time, whereas $\alpha$, $\beta$, and $\rho$ have a more visible effect. This is consistent with the role of these parameters in controlling the primal update and penalty scaling.

\begin{table}[htbp]
	\caption{Results for Sto-SiPBA with different hyperparameters.}
	\label{ablationanalysis}
	\footnotesize
	\setlength{\tabcolsep}{3pt}
	\begin{tabular*}{\textwidth}{@{\extracolsep{\fill}}cccccccc@{}}
		\toprule
		\textbf{$\alpha_0$} & \textbf{$\beta_0$} & \textbf{$\rho_0$} & \textbf{$\sigma_0$} & \textbf{$\delta_0$} & \textbf{$\eta_0$} & \textbf{Time (s)} & \textbf{Iter.} \\
		\midrule
		0.1 & 0.1 & 10 & $10^{-4}$ & $10^{-4}$ & 1 & $0.72\,(\pm 0.10)$ & $113.3\,(\pm 2.8)$ \\
		0.05 & 0.1 & 10 & $10^{-4}$ & $10^{-4}$ & 1 & $1.85\,(\pm 0.08)$ & $295.8\,(\pm 7.3)$ \\
		0.2 & 0.1 & 10 & $10^{-4}$ & $10^{-4}$ & 1 & $0.62\,(\pm 0.05)$ & $102.9\,(\pm 3.4)$ \\
		0.1 & 0.05 & 10 & $10^{-4}$ & $10^{-4}$ & 1 & $0.41\,(\pm 0.03)$ & $67.4\,(\pm 3.2)$ \\
		0.1 & 0.02 & 10 & $10^{-4}$ & $10^{-4}$ & 1 & $0.33\,(\pm 0.03)$ & $52.6\,(\pm 4.7)$ \\
		0.1 & 0.1 & 5 & $10^{-4}$ & $10^{-4}$ & 1 & $0.43\,(\pm 0.02)$ & $67.7\,(\pm 1.8)$ \\
		0.1 & 0.1 & 20 & $10^{-4}$ & $10^{-4}$ & 1 & $2.19\,(\pm 0.06)$ & $356.7\,(\pm 5.3)$ \\
		0.1 & 0.1 & 10 & $10^{-5}$ & $10^{-4}$ & 1 & $0.70\,(\pm 0.05)$ & $113.3\,(\pm 2.8)$ \\
		0.1 & 0.1 & 10 & $10^{-3}$ & $10^{-4}$ & 1 & $0.68\,(\pm 0.05)$ & $113.2\,(\pm 2.7)$ \\
		0.1 & 0.1 & 10 & $10^{-4}$ & $10^{-5}$ & 1 & $0.65\,(\pm 0.03)$ & $113.3\,(\pm 2.8)$ \\
		0.1 & 0.1 & 10 & $10^{-4}$ & $10^{-3}$ & 1 & $0.69\,(\pm 0.03)$ & $113.2\,(\pm 2.7)$ \\
		0.1 & 0.1 & 10 & $10^{-4}$ & $10^{-4}$ & 0.8 & $0.68\,(\pm 0.03)$ & $114.2\,(\pm 2.2)$ \\
		0.1 & 0.1 & 10 & $10^{-4}$ & $10^{-4}$ & 0.5 & $0.67\,(\pm 0.04)$ & $114.4\,(\pm 3.5)$ \\
		\bottomrule
	\end{tabular*}
\end{table}

\subsection{Spam Classification}
Spam distributions can vary substantially across corpora, users, and time; consequently, a classifier trained on one corpus may transfer poorly to another. We therefore use cross-corpus spam classification to examine whether pessimistic modeling reduces sensitivity to source-corpus-specific patterns. Specifically, we consider the following feature-map generalization of the spam classification model in \cite{bruckner2011stackelberg}:
\begin{equation}\label{spammodel}
	\begin{aligned}
		\min_{w\in\IR^n}\;\max_{\hat{x}\in\mathcal{X}}\quad
		& l(w,\hat{x},y)+\lambda_1\operatorname{Reg}(w)\\
		\mathrm{s.t.}\quad
		& \hat{x}\in\underset{x^\prime\in\mathcal{X}}{\arg\min}\;
		\left\{
		l^\prime(w,x^\prime)
		+\lambda_2\|\varphi(x^\prime)-\varphi(x)\|^2
		\right\}.
	\end{aligned}
\end{equation}
Here, $w$ denotes the classifier parameters, while $x$ and $y$ denote the vectorized messages and their associated labels, respectively.  The functions $l$ and $l^\prime$ are the classifier and sender losses, $\operatorname{Reg}(w)$ is the classifier regularizer, and $\varphi$ is the feature map used to measure the cost of modifying a message. For a fixed classifier, the lower-level problem models the sender's optimal adaptive response. Specifically, the sender modifies each
spam message so that the classifier is more likely to label it as
legitimate email, while keeping the modified message close to the
original one in the feature space. When this response is nonunique, the maximization in \eqref{spammodel} selects a lower-level solution that is least favorable to the classifier. The outer minimization then trains the classifier against this pessimistic response.

We compare three training configurations.  The first configuration applies SiPBA to \eqref{spammodel}, with $\varphi$ given by the projection onto the leading $100$ principal components of the source-corpus training features. Thus, the modification cost measures changes in the dominant feature subspace. The second configuration uses the sequential quadratic programming (SQP) procedure of \cite{bruckner2011stackelberg} with the identity map $\varphi(x)=x$, as in the original formulation. The third configuration is a standard single-level classifier trained on the source corpus without the adaptive-response problem and implemented using scikit-learn \cite{pedregosa2011scikit}. For the PBO methods, we use either hinge loss or cross-entropy (CE) for both $l$ and $l^\prime$; the corresponding single-level baseline uses the same classification loss. We denote the resulting methods by SiPBA-Hinge/CE, SQP-Hinge/CE, and Single-Hinge/CE, respectively.

\begin{table}[tbp]
	\caption{Mean classification accuracy (Acc.) and F1 score, in percent, over ten random seeds. Each block identifies the source corpus used for training. The final column reports the arithmetic mean over the four test corpora, comprising one in-corpus and three cross-corpus test sets.}
	\label{mutual_training_results}
	\footnotesize
	\setlength{\tabcolsep}{1.5pt}
	\begin{tabular*}{\textwidth}{@{\extracolsep{\fill}}llccccc@{}}
		\toprule
		\multirow{2}{*}{\textbf{Train Set}}
		& \multirow{2}{*}{\textbf{Method}}
		& \multicolumn{4}{c}{\textbf{Test corpus (Acc./F1)}}
		& \multirow{2}{*}{\textbf{Ave (Acc./F1)}} \\
		\cmidrule(lr){3-6}
		& & \textbf{TREC06} & \textbf{TREC07} & \textbf{EnronSpam} & \textbf{LingSpam} & \\
		\midrule
		\multirow{6}{*}{TREC06}
		& SiPBA-Hinge & 96.4/94.7 & 87.3/81.0 & 70.6/70.2 & 87.6/92.7 & 85.5/84.7 \\
		& SiPBA-CE & 94.5/92.5 & 79.5/73.0 & 70.9/71.8 & 87.6/92.8 & 83.1/82.5 \\
		& SQP-Hinge & 93.1/90.0 & 89.2/83.2 & 69.0/66.7 & 89.0/93.4 & 85.1/83.3 \\
		& SQP-CE & 93.6/91.3 & 78.9/72.4 & 70.7/71.4 & 87.2/92.6 & 82.6/81.9 \\
		& Single-Hinge & 95.4/93.1 & 89.3/82.8 & 63.9/46.5 & 75.5/82.5 & 81.0/76.2 \\
		& Single-CE & 93.8/90.4 & 88.5/79.6 & 56.9/24.1 & 55.1/62.6 & 73.6/64.2 \\
		\midrule
		\multirow{6}{*}{TREC07}
		& SiPBA-Hinge & 68.9/16.8 & 93.7/89.7 & 57.0/33.7 & 50.5/57.6 & 67.5/49.5 \\
		& SiPBA-CE & 71.7/56.9 & 98.1/97.2 & 68.3/68.8 & 64.6/75.5 & 75.7/74.6 \\
		& SQP-Hinge & 68.9/17.2 & 95.3/92.5 & 55.0/21.0 & 29.9/28.1 & 62.3/39.7 \\
		& SQP-CE & 71.3/56.9 & 97.7/96.6 & 68.4/69.7 & 70.1/80.5 & 76.9/75.9 \\
		& Single-Hinge & 65.4/1.9 & 97.7/96.4 & 50.9/0.2 & 16.6/0.3 & 57.7/24.7 \\
		& Single-CE & 66.4/3.4 & 95.7/93.0 & 51.0/0.8 & 17.3/1.8 & 57.6/24.8 \\
		\midrule
		\multirow{6}{*}{EnronSpam}
		& SiPBA-Hinge & 75.8/61.8 & 72.1/28.0 & 95.9/95.8 & 59.6/67.4 & 75.9/63.3 \\
		& SiPBA-CE & 76.3/62.8 & 74.0/34.4 & 95.2/95.0 & 64.0/72.0 & 77.4/66.1 \\
		& SQP-Hinge & 77.5/61.7 & 70.5/22.8 & 96.1/96.0 & 52.3/59.3 & 74.1/60.0 \\
		& SQP-CE & 76.0/62.6 & 73.4/32.9 & 94.9/94.8 & 63.0/71.0 & 76.8/65.3 \\
		& Single-Hinge & 76.8/56.0 & 69.3/15.0 & 95.8/95.6 & 47.2/52.3 & 72.3/54.7 \\
		& Single-CE & 76.4/55.4 & 70.0/19.2 & 95.6/95.3 & 43.1/46.9 & 71.3/54.2 \\
		\midrule
		\multirow{6}{*}{LingSpam}
		& SiPBA-Hinge & 63.4/59.1 & 66.2/51.2 & 71.1/65.4 & 99.4/99.6 & 75.0/68.8 \\
		& SiPBA-CE & 71.8/48.5 & 69.0/27.6 & 59.1/34.3 & 91.8/94.8 & 72.9/51.3 \\
		& SQP-Hinge & 42.5/53.8 & 45.3/52.0 & 72.5/65.8 & 98.2/99.0 & 64.6/67.7 \\
		& SQP-CE & 72.0/49.5 & 68.9/26.2 & 58.9/33.9 & 91.9/94.8 & 72.9/51.1 \\
		& Single-Hinge & 37.2/51.9 & 38.6/50.6 & 56.7/69.0 & 95.7/97.5 & 57.1/67.3 \\
		& Single-CE & 34.5/51.0 & 34.0/50.1 & 51.3/66.8 & 91.4/95.1 & 52.8/65.8 \\
		\bottomrule
	\end{tabular*}
\end{table}

We use four standard spam corpora: TREC 2006 (TREC06) \cite{ounis2006overview}, TREC 2007 (TREC07) \cite{macdonald2007overview}, EnronSpam \cite{metsis2006spam}, and LingSpam \cite{androutsopoulos2000evaluation}. For each random seed, one corpus is designated as the source corpus, from which $500$ emails are sampled uniformly without replacement. The same sample is used to train all methods. A \texttt{TfidfVectorizer} with \texttt{max\_features=3000} is fitted only on these $500$ training emails, so the resulting TF--IDF feature dimension $n$ is at most $3000$. The fitted vectorizer is subsequently used to transform the held-out source-corpus emails and all emails from the other three corpora. The former constitute the in-corpus test set, whereas the latter constitute the cross-corpus test sets. For the SiPBA configuration, the principal-component map is also fitted only on the source-corpus training sample. We repeat the experiment for ten random seeds and report the mean accuracy and F1 score. Under this representation, each lower-level message variable is $x\in\IR^n$ with $n\le 3000$, and the $500$ sampled training messages are stacked into $X\in\IR^{500\times n}$ in the implementation, whereas the upper-level classifier variable is $w\in\IR^n$.

Table~\ref{mutual_training_results} shows that, for every source corpus and both loss functions, each PBO configuration attains a higher four-corpus mean accuracy than its corresponding single-level baseline. The same comparison holds for the mean F1 score in seven of the eight source--loss settings; the exception is cross-entropy training on LingSpam. Relative to the SQP configuration, SiPBA achieves a higher mean accuracy in six settings, the same accuracy in one setting, and a lower accuracy in one setting. It also achieves a higher mean F1 score in seven of the eight settings. These results indicate that the tested PBO configurations generally perform favorably across heterogeneous evaluation corpora, although the improvement is not uniform across loss functions, source corpora, and evaluation metrics.

\subsection{Smart Predict-then-Optimize}

Smart Predict-then-Optimize (SPO) is a recent decision-focused learning paradigm, where the predictive model is trained according to the quality of the
downstream decisions induced by its predictions~\cite{elmachtoub2022smart}. Given contextual features, the predictor estimates the unknown parameters of an optimization problem, and these estimates are then used to compute a decision. The learning objective is therefore not prediction accuracy alone, but rather the out-of-sample performance of the decisions obtained under the true problem parameters.

When the downstream optimization problem admits multiple optimal solutions, the
decision induced by a prediction is set-valued. The standard SPO loss is then
not well defined unless a selection rule is prescribed to resolve this ambiguity. To this end, \cite{elmachtoub2022smart} introduced the unambiguous SPO loss, which evaluates the worst realized loss over all downstream optimal decisions. This construction naturally leads to a pessimistic bilevel formulation.

Directly optimizing the resulting pessimistic formulation, however, remains difficult. The few existing approaches exploit special structures of the downstream problem, such as linear optimization or integer linear regression \cite{bucareydecision,jimenez2025pessimistic}. In our experiment, we consider a continuous nonlinear SPO problem with a non-singleton lower-level solution set and examine the computational performance of our proposed Sto-SiPBA method.

We study a tolerance-insensitive inventory problem, motivated by operational
costs with an insensitive region \cite{chen2024learning}. For a production
quantity $p$ and demand $d$, the operational loss is
\[
F_{\tau}(p,d)
=
c_o [p-d-\tau]_+
+ q_o [p-d-\tau]_+^2
+
c_u [d-p-\tau]_+
+ q_u [d-p-\tau]_+^2,
\]
where $[x]_+=\max\{x,0\}$ and $\tau\ge 0$ denotes the tolerance width.
Throughout the experiment, we set $c_o=1$, $q_o=0.1$, $c_u=5$, and
$q_u=0.5$. The parameter $\tau$ specifies a no-penalty interval around the
demand; hence, increasing $\tau$ enlarges the set of downstream optimal decisions.

Let $d_{\theta}(\xi)$ be the demand prediction produced from contextual features
$\xi$. Given this prediction, the lower-level problem and its optimal solution set can be represented as 
\[
P_{\tau}^{*}(\theta)
=
\arg\min_{p\ge 0} \frac{1}{|D_{\text{train}}|}\sum_{i\in\mathcal D_{\mathrm{train}}} F_{\tau}\bigl(p_i,d_{\theta}(\xi_i)\bigr).
\]
Then, the pessimistic SPO loss
evaluates the worst realized performance over all decisions in
$P_{\tau}^{*}(\theta)$, and is given by
\[
\ell_{\mathrm{pess}}(\theta):= \max_{p\in P_{\tau}^{*}(\theta)} \frac{1}{|\mathcal D_{\mathrm{train}}|}
\sum_{i\in\mathcal D_{\mathrm{train}}} F_{\tau}\bigl(p_i,d_i\bigr).
\]
This results in the  following pessimistic bilevel optimization with a stochastic structure:
\begin{equation}\label{pessimisticspo}
	\begin{aligned}
		\min_{\theta}\max_{p\in P_{\tau}^{*}(\theta)}\;
		\frac{1}{|\mathcal D_{\mathrm{train}}|}
		\sum_{i\in\mathcal D_{\mathrm{train}}} F_{\tau}\bigl(p_i,d_i\bigr).
	\end{aligned}
\end{equation}

We compare Sto-SiPBA with TTSA \cite{hong2023two}, LANCER
\cite{zharmagambetov2023landscape}, SurCo \cite{ferber2023surco}, and an
MSE-trained predictor. Here, TTSA is used to train the optimistic variant of \eqref{pessimisticspo}, while LANCER and SurCo serve as task-aware learning
baselines. All methods use the same
embedding-MLP architecture: categorical covariates are embedded and concatenated
with standardized numerical covariates before being passed to a multilayer
perceptron. We use the hourly Bike Sharing Dataset \cite{fanaee2014event},
which contains 17,379 hourly observations from 2011 to 2012. The target is the
total number of rentals in each hour, and the inputs consist of calendar, time,
and weather covariates. The variables \texttt{casual} and
\texttt{registered} are excluded to avoid information leakage. The data are
sorted chronologically and split into training, validation, and test sets using
a 70\%/15\%/15\% day-forward split, resulting in 12,187 training samples, 2,592
validation samples, and 2,600 test samples. Under the network architecture and training split used in our
experiments, the upper-level variable is the predictor parameter vector
$\theta\in\mathbb{R}^{45{,}445}$, whereas the lower-level variable is
$p=(p_i)_{i\in\mathcal D_{\mathrm{train}}}
\in\mathbb{R}_{+}^{12{,}187}$, with one decision associated with each
training observation.

For each method, hyperparameters are selected at $\tau=10$ according
to the validation pessimistic loss. The selected hyperparameters are then fixed
and used for all values $\tau\in\{0,5,10,15,20\}$. For each value of $\tau$, we
repeat the experiment over ten random seeds. We report pessimistic,
optimistic, nominal, and MSE losses. The optimistic loss uses the best
downstream decision in $P_{\tau}^{*}(\theta)$, while the nominal loss
uses the center rule $p_{\mathrm{nom}}(\xi_i)=d_{\theta}(\xi_i)$. The MSE loss
is defined as
\[
\ell_{\mathrm{mse}}(\theta)
=
\frac{1}{|\mathcal D_{\mathrm{test}}|}
\sum_{i\in \mathcal D_{\mathrm{test}}}
\bigl(d_i-d_{\theta}(\xi_i)\bigr)^2 .
\]
\begin{figure}[t]
	\centering
	\includegraphics[width=\linewidth]{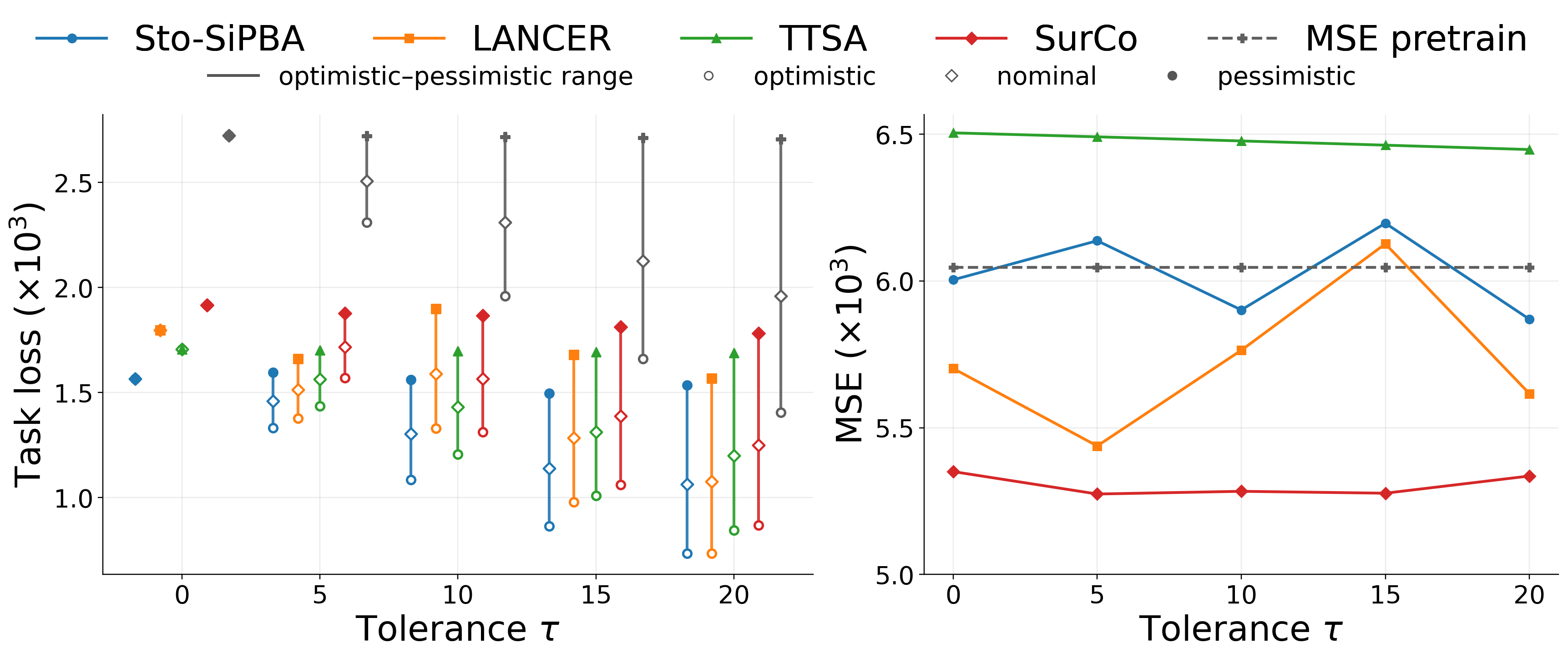}
	\caption{
		Mean task loss and MSE with varying tolerance $\tau$ over ten runs. 
	}\label{SPO}
\end{figure}

\begin{table}[t]
	\caption{Mean pessimistic task loss and MSE with varying tolerance $\tau$ over ten runs.}
	\label{tab:spo-repeat}
	\scriptsize
	\setlength{\tabcolsep}{1pt}
	\begin{tabular*}{\textwidth}{@{\extracolsep{\fill}}clccccc@{}}
		\toprule
		\textbf{$\tau$} & \textbf{Metric} & \textbf{Sto-SiPBA} & \textbf{LANCER} & \textbf{TTSA} & \textbf{SurCo} & \textbf{MSE pretrain} \\
		\midrule
		\multirow{2}{*}{$0$}
		& Task loss ($10^3$) & $\mathbf{1.566\pm 0.305}$ & $1.797\pm 0.631$ & $1.706\pm 0.352$ & $1.916\pm 0.741$ & $2.722\pm 0.916$ \\
		& MSE ($10^3$) & $6.004\pm 1.421$ & $5.702\pm 1.905$ & $6.505\pm 1.217$ & $\mathbf{5.351\pm 0.763}$ & $6.047\pm 1.333$ \\
		\midrule
		\multirow{2}{*}{$5$}
		& Task loss ($10^3$) & $\mathbf{1.597\pm 0.340}$ & $1.662\pm 0.461$ & $1.702\pm 0.352$ & $1.877\pm 0.729$ & $2.721\pm 0.916$ \\
		& MSE ($10^3$) & $6.137\pm 1.651$ & $5.438\pm 1.770$ & $6.491\pm 1.212$ & $\mathbf{5.274\pm 0.763}$ & $6.047\pm 1.333$ \\
		\midrule
		\multirow{2}{*}{$10$}
		& Task loss ($10^3$) & $\mathbf{1.561\pm 0.309}$ & $1.898\pm 0.853$ & $1.698\pm 0.352$ & $1.867\pm 0.741$ & $2.717\pm 0.914$ \\
		& MSE ($10^3$) & $5.902\pm 1.527$ & $5.764\pm 1.432$ & $6.477\pm 1.206$ & $\mathbf{5.284\pm 0.799}$ & $6.047\pm 1.333$ \\
		\midrule
		\multirow{2}{*}{$15$}
		& Task loss ($10^3$) & $\mathbf{1.495\pm 0.279}$ & $1.680\pm 0.460$ & $1.693\pm 0.352$ & $1.812\pm 0.701$ & $2.712\pm 0.912$ \\
		& MSE ($10^3$) & $6.197\pm 1.636$ & $6.127\pm 2.722$ & $6.463\pm 1.200$ & $\mathbf{5.277\pm 0.684}$ & $6.047\pm 1.333$ \\
		\midrule
		\multirow{2}{*}{$20$}
		& Task loss ($10^3$) & $\mathbf{1.535\pm 0.275}$ & $1.567\pm 0.439$ & $1.688\pm 0.352$ & $1.781\pm 0.673$ & $2.706\pm 0.909$ \\
		& MSE ($10^3$) & $5.870\pm 1.586$ & $5.616\pm 1.358$ & $6.448\pm 1.194$ & $\mathbf{5.336\pm 0.584}$ & $6.047\pm 1.333$ \\
		\bottomrule
	\end{tabular*}
\end{table}
The results are summarized in Figure \ref{SPO} and Table
\ref{tab:spo-repeat}. The MSE baseline denotes a predictor trained only by the
supervised MSE loss. Across all tested tolerance levels, Sto-SiPBA attains the
smallest test pessimistic loss. Averaged over the five tolerance levels, its
pessimistic task loss is $1550.7$, compared with $1697.5$ for TTSA, $1720.9$
for LANCER, $1850.6$ for SurCo, and $2715.5$ for the MSE-trained predictor.
This clear improvement is consistent with the pessimistic formulation: the training
criterion penalizes predictions whose induced optimal-decision set contains
solutions with poor realized performance. Hence the learned predictor is
encouraged to be stable with respect to the ambiguity introduced by the
downstream optimization problem, rather than to be merely pointwise accurate in
demand space.

SurCo has the lowest average MSE across the tolerance sweep, and achieves the lowest MSE for all five tolerance levels, while Sto-SiPBA remains comparable to the MSE-trained predictor on
average. One possible explanation is that SPO-type objectives introduce a task-dependent
bias in training: prediction errors are assessed through their effect on the
downstream loss, so errors in directions that have limited influence on the
induced decision may be penalized less. This may partly explain why such methods can exhibit competitive out-of-sample MSE performance, even though they are not
trained to minimize the MSE directly.
Finally, Sto-SiPBA has the smallest standard deviation of the pessimistic task
loss in all five tolerance settings, whereas its MSE variability is
relatively larger. This pattern is consistent with the objective optimized by
the pessimistic model, which primarily controls task performance under the
worst admissible downstream decision.
\section{Conclusions}

This paper studies pessimistic bilevel optimization and proposes a novel smooth approximation for the pessimistic value function through penalization, regularization, and a coupling mechanism between lower-level variables. 

The resulting differentiable approximation is theoretically justified by convergence of global minimizers and C-stationarity consistency under the stated constraint qualifications and uniform parametric lower-level solution-set error bound condition.

With this approximation, we developed SiPBA and Sto-SiPBA, which, to the best of our knowledge, are the first fully first-order, single-loop algorithms for deterministic and stochastic PBO, respectively.
Both methods avoid second-order computations and inner-loop subproblem solves, while admitting convergence guarantees to C-stationary points of PBO.

Through a systematic empirical study of synthetic problem instances, the spam classification problem, and a Smart Predict-then-Optimize model, we observe that our solution methods generally outperform existing approaches, with clear advantages in efficiency, effectiveness, and robustness. Building on the proposed smooth approximation framework, a promising direction for future research is to incorporate advanced optimization techniques to develop more efficient algorithms with sharper convergence rates and complexity guarantees.

\begin{appendices}

	\section{Proof for Section \ref{sec:smoothapproximation}}
	\label{appendix:smoothapproximation}
	
	\subsection{Proof for Theorem \ref{differentiable}}
	We first present a lemma showing that the coercivity of $-F(x,\cdot)$ implies a uniform upper-level boundedness property of $F(x,\cdot)$ over compact sets of $x$. 
	\begin{lemma}\label{coerciveuniform}
		Suppose that $F:\IR^n\times\IR^m\to\IR$ is continuous and that $B\subset \IR^n$ is
		compact. If for each $x\in B$, the function $-F(x,\cdot)$ is coercive and
		$F(x,\cdot)$ is concave, then for any $M>0$, there exists a constant $M_y>0$ such that
		\[
		F(x,y)<-M,
		\qquad \forall x\in B,\ \forall y\in Y \text{ with } \|y\|>M_y.
		\]
	\end{lemma}
	\begin{proof}
		We prove this by providing a  contradiction. Suppose that the conclusion fails for some
		$M>0$. Then there exist sequences $x_k\in B$ and $y_k\in Y$ such that
		\[
		\|y_k\|\to\infty,
		\qquad
		F(x_k,y_k)\ge -M .
		\]
		Since $B$ is compact, after passing to a subsequence we may assume that
		$x_k\to \bar x\in B$.
		
		Fix an arbitrary point $y^0\in Y$. Set
		$r_k:=\|y_k-y^0\|,\;
		d_k:=\frac{y_k-y^0}{r_k}.$
		Then $r_k\to\infty$. Passing to a further subsequence if necessary, we may
		assume that $d_k\to d$ with $\|d\|=1$.
		
		Let $t>0$ be fixed. For all sufficiently large $k$, $t/r_k\in(0,1)$, and hence
		\[
		y_k(t):=y^0+\frac{t}{r_k}(y_k-y^0)
		=\left(1-\frac{t}{r_k}\right)y^0+\frac{t}{r_k}y_k
		\in Y
		\]
		by convexity of $Y$. Since $Y$ is closed and $y_k(t)\to y^0+td$, we also have
		$y^0+td\in Y$.
		
		Using the concavity of $F(x_k,\cdot)$ on $Y$, we obtain
		\[
		\begin{aligned}
			F(x_k,y_k(t))
			&\ge
			\left(1-\frac{t}{r_k}\right)F(x_k,y^0)
			-\frac{tM}{r_k}.
		\end{aligned}
		\]
		Letting $k\to\infty$ and using the continuity of $F$ gives
		\[
		F(\bar x,y^0+td)\ge F(\bar x,y^0),\qquad \forall t>0 .
		\]
		However, $\|y^0+td\|\to\infty$ as $t\to\infty$, and $y^0+td\in Y$ for every
		$t>0$. This contradicts the coercivity of $-F(\bar x,\cdot)$ on $Y$, i.e.,
		$F(\bar x,y)\to-\infty$ as $\|y\|\to\infty$ along $y\in Y$.
		Thus, the original assumption must be false, which completes the proof.
	\end{proof}
	
	\begin{proof}[Proof for Theorem \ref{differentiable}]
		For brevity, write
		$\psi:=\psi_{\rho,\sigma,\delta}$.
		By Assumptions~\ref{assum1}--\ref{assum2}, $\psi$ is continuously differentiable.
		Moreover, for each fixed $(x,z)$, the map $y\to \psi(x,y,z)$ is
		$\delta$-strongly concave on $Y$, while for each fixed $(x,y)$, the map
		$z\to \psi(x,y,z)$ is $\sigma$-strongly convex on $Y$.
		
		Fix $(\bar x,\bar z)\in X\times Y$ and compact neighborhoods $B_X$ and $B_Z$ of
		$\bar x$ and $\bar z$. By Assumption~\ref{assum2}, there exists a bounded set $D$
		such that $S(x)\cap D\neq\emptyset$ for all $x\in B_X$; hence $f(x,\cdot)$ is
		uniformly bounded below on $B_X$. In addition, Lemma~\ref{coerciveuniform}
		implies that $-F(x,\cdot)$ is uniformly coercive on $B_X$. Since $z$ ranges in the
		bounded set $B_Z$, the linear term $\sigma\langle y,z\rangle$ is dominated by the
		quadratic term $\frac{\delta}{2}\|y\|^2$. Therefore the inf-compactness condition holds for $-\psi(x,\cdot,z)$. Combined with the $\delta$-strong convexity of
		$-\psi(x,\cdot,z)$, this shows that
		\[
		\widehat y(x,z):=\arg\max_{y\in Y}\psi(x,y,z)
		\]
		is well defined and unique.
		
		Define
		$
		h(x,z):=\max_{y\in Y}\psi(x,y,z)=\psi(x,\widehat y(x,z),z).
		$
		By \cite[Theorem~4.13 and Remark~4.14]{bonnans2013perturbation}, $h$ is
		differentiable and
		\[
		\nabla h(x,z)
		=
		\bigl(
		\nabla_x\psi(x,\widehat y(x,z),z),
		\nabla_z\psi(x,\widehat y(x,z),z)
		\bigr).
		\]
		By the maximum theorem and the uniqueness of the maximizer, $\widehat y$ is
		continuous; hence $h$ is continuously differentiable on $X\times Y$.
		
		Next, since $z\mapsto \psi(x,y,z)$ is $\sigma$-strongly convex for every fixed
		$(x,y)$, the pointwise maximum $h(x,\cdot)$ is also $\sigma$-strongly convex.
		Moreover, choosing $x^\prime\in S(x)\cap D$ on compact neighborhoods of $x$, we have
		\[
		h(x,z)\ge \psi(x,x^\prime,z)
		\ge
		F(x,x^\prime)+\frac{\sigma}{2}\|z-x^\prime\|^2-\frac{\sigma+\delta}{2}\|x^\prime\|^2,
		\]
		so the inf-compactness condition holds for $h(x,\cdot)$. Therefore
		\[
		z_{\rho,\sigma,\delta}^*(x):=\arg\min_{z\in Y}h(x,z)
		\]
		exists and is unique.
		
		Since
		$
		\phi_{\rho,\sigma,\delta}(x)=\min_{z\in Y}h(x,z),
		$
		applying \cite[Theorem~4.13 and Remark~4.14]{bonnans2013perturbation} once more
		yields
		\[
		\nabla\phi_{\rho,\sigma,\delta}(x)
		=
		\nabla_x h\bigl(x,z_{\rho,\sigma,\delta}^*(x)\bigr)
		=
		\nabla_x\psi\bigl(x,\widehat y(x,z_{\rho,\sigma,\delta}^*(x)),
		z_{\rho,\sigma,\delta}^*(x)\bigr).
		\]
		Because $\psi(x,\cdot,\cdot)$ is strongly concave-convex, the saddle point of
		\[
		\min_{z\in Y}\max_{y\in Y}\psi(x,y,z)
		\]
		is unique. Hence
		$
		\widehat y(x,z_{\rho,\sigma,\delta}^*(x))=y_{\rho,\sigma,\delta}^*(x).
		$
		Finally, by the definition of $\psi$,
		\[
		\nabla_x\psi(x,y,z)
		=
		\nabla_xF(x,y)-\rho\nabla_xf(x,y)+\rho\nabla_xf(x,z),
		\]
		since the regularization terms are independent of $x$.\\
		 Substituting
		$(y,z)=\bigl(y_{\rho,\sigma,\delta}^*(x),z_{\rho,\sigma,\delta}^*(x)\bigr)$ gives
		\eqref{gradient_phi}.
	\end{proof}
	\subsection{Proof for Lemma \ref{lem:limsup}}
	Before presenting the proof for Lemma \ref{lem:limsup}, we first establish some auxiliary results. First, we establish a uniform boundedness property for the saddle point components $y_k^*(x)$ and $z_k^*(x)$.
	
	\begin{lemma}\label{saddlepointboundedness}
		Let $B \subset X$ be a compact set. Suppose that the parameters satisfy criterion \eqref{paramcriterion}. Then, there exists a constant $M_B > 0$ such that for all $k$ and all $x \in B$,
		\[
		\|y^*_{\rho_k,\sigma_k,\delta_k}(x)\| \le M_B, \quad \text{and} \quad \|z^*_{\rho_k,\sigma_k,\delta_k}(x)\| \le M_B.
		\]
	\end{lemma}
	\begin{proof}
		We first prove the uniform boundedness of $\{y_k^*(x)\}$ on $B$.
		Suppose, to the contrary, that there exist $x_k\in B$ such that
		$\|y_k^*(x_k)\|\to\infty$.
		By Assumption~\ref{assum2}, we may choose
		$\hat y_k\in \mathcal S(x_k)\cap D$, where $D\subset Y$ is bounded; hence
		$\{\hat y_k\}$ is bounded.
		
		Using the saddle-point property of $(y_k^*(x_k),z_k^*(x_k))$ and the fact that
		$\hat y_k\in \mathcal S(x_k)$, we obtain after rearrangement
		\[
		\frac{2}{\delta_k}\Big(F(x_k,y_k^*(x_k))-F(x_k,\hat y_k)\Big)
		\ge
		\|y_k^*(x_k)\|^2
		+2\frac{\sigma_k}{\delta_k}\langle y_k^*(x_k),\hat y_k\rangle
		-\Bigl(2\frac{\sigma_k}{\delta_k}+1\Bigr)\|\hat y_k\|^2.
		\]
		Since $\limsup_{k\to\infty}\sigma_k/\delta_k<\infty$ and $\{\hat y_k\}$ is bounded, the
		right-hand side is bounded from below by
		\[
		\|y_k^*(x_k)\|^2-C\|y_k^*(x_k)\|-C
		\]
		for all sufficiently large $k$, and therefore tends to $+\infty$.
		
		On the other hand, $\{F(x_k,\hat y_k)\}$ is bounded because $B$ is compact and $D$
		is bounded. Moreover, Lemma~\ref{coerciveuniform} gives
		$F(x_k,y_k^*(x_k))\to -\infty$, since $-F(x,\cdot)$ is uniformly coercive on $B$
		and $\|y_k^*(x_k)\|\to\infty$. Hence the left-hand side tends to $-\infty$, a
		contradiction. Therefore, $\{y_k^*(x)\}$ is uniformly bounded on $B$.
		
		We next prove the boundedness of $\{z_k^*(x)\}$. Fix any $x\in B$ and choose
		$\hat y\in \mathcal S(x)\cap D$. Since $z_k^*(x)$ minimizes
		$z\mapsto \psi_k(x,y_k^*(x),z)$ over $Y$, comparison with $z=\hat y$ yields
		\[
		\rho_k\bigl(f(x,z_k^*(x))-f(x,\hat y)\bigr)
		+\frac{\sigma_k}{2}
		\Bigl(
		\|z_k^*(x)-y_k^*(x)\|^2-\|\hat y-y_k^*(x)\|^2
		\Bigr)\le 0.
		\]
		Because $\hat y\in \mathcal S(x)$, the first term is nonnegative, so
		\[
		\|z_k^*(x)-y_k^*(x)\|\le \|\hat y-y_k^*(x)\|.
		\]
		Consequently,
		\[
		\|z_k^*(x)\|
		\le \|y_k^*(x)\|+\|z_k^*(x)-y_k^*(x)\|
		\le 2\|y_k^*(x)\|+\|\hat y\|.
		\]
		The right-hand side is uniformly bounded on $B$, since $\{y_k^*(x)\}$ is uniformly
		bounded and $D$ is bounded. Hence $\{z_k^*(x)\}$ is uniformly bounded on $B$ as well.
	\end{proof}

	Next, we establish Lemma \ref{valueofy}, which shows that the accumulation points of $\{y^*_{k}(x_k)\}$ belong to the solution set $\mathcal{S}(\bar{x})$ when $x_k \rightarrow \bar{x}$. 
	As the proofs follow the same arguments as those of   \citep[Lemma B.3]{qichao2025single}, we omit it here.
	\begin{lemma}\label{valueofy}
		Suppose that the parameters satisfy criterion \eqref{paramcriterion}. Then, for any sequence $\{x_k\} \subset X$ such that $x_k \rightarrow \bar{x} \in X$ as $k \rightarrow \infty$, we have
		\begin{equation}
			\underset{k\to\infty}{\lim\sup}\;f(x_k,y_{k}^*(x_k))\le  \min_{y \in Y} f(\bar{x}, y).
		\end{equation}
		Consequently, for any accumulation point $\bar{y}$ of sequence $\{y_{k}^*(x_k)\}$, we have $\bar{y} \in \mathcal{S}(\bar{x})$.
	\end{lemma}

	Now, we present the proof for  Lemma \ref{lem:limsup}.
	\begin{proof}[Proof for Lemma \ref{lem:limsup}]
		We first prove \eqref{limsupphiapp}. Suppose to the contrary that there exist
		$\bar{x}\in X$ and $\varepsilon>0$ such that
		\[
		\limsup_{k\to\infty}\phi_k(\bar{x})>\phi(\bar{x})+\varepsilon.
		\]
		Then, by passing to a subsequence if necessary, we may assume that
		$
		\phi_k(\bar{x})>\phi(\bar{x})+\varepsilon, 
		$ for all $k$.
		For brevity, write
		$y_k^*:=y_k^*(\bar{x}),\;z_k^*:=z_k^*(\bar{x}).$
		Since $(y_k^*,z_k^*)$ is the saddle point defining $\phi_k(\bar{x})$, we have
		\begin{align}
			\phi_k(\bar{x})
			&=F(\bar{x},y_k^*)
			-\rho_k\bigl(f(\bar{x},y_k^*)-f(\bar{x},z_k^*)\bigr)
			+\frac{\sigma_k}{2}\|z_k^*\|^2
			-\sigma_k\langle y_k^*,z_k^*\rangle
			-\frac{\delta_k}{2}\|y_k^*\|^2\nonumber\\
			&\ge \phi(\bar{x})+\varepsilon. \label{lem:limsup_pf_1}
		\end{align}
		
		By Lemma \ref{saddlepointboundedness}, the sequence $\{y_k^*\}$ is bounded.
		Hence, by passing to a further subsequence if necessary, we may assume that
		$y_k^*\to \bar{y}$ for some $\bar{y}\in Y$.
		Applying Lemma \ref{valueofy} to the constant sequence $x_k\equiv \bar{x}$ yields
		$\bar{y}\in \mathcal{S}(\bar{x})$.
		On the other hand, since $z_k^*$ minimizes $\psi_k(\bar{x},y_k^*,z)$ over $z\in Y$,
		we have
		$
		\psi_k(\bar{x},y_k^*,z_k^*)\le \psi_k(\bar{x},y_k^*,y_k^*).
		$
		Expanding this inequality and rearranging terms, we obtain
		\[
		\rho_k\bigl(f(\bar{x},z_k^*)-f(\bar{x},y_k^*)\bigr)
		+\frac{\sigma_k}{2}\|z_k^*-y_k^*\|^2 \le 0.
		\]
		Combining this with \eqref{lem:limsup_pf_1}, we get
		\[
		F(\bar{x},y_k^*)-\frac{\sigma_k+\delta_k}{2}\|y_k^*\|^2
		\ge \phi(\bar{x})+\varepsilon.
		\]
		Letting $k\to\infty$, and using the continuity of $F$, the boundedness of $\{y_k^*\}$,
		and the facts that $\sigma_k\to0$ and $\delta_k\to0$, we arrive at
		\[
		F(\bar{x},\bar{y})\ge \phi(\bar{x})+\varepsilon.
		\]
		However, since $\bar{y}\in \mathcal{S}(\bar{x})$, by the definition of $\phi(\bar{x})$,
		we must have
		$
		F(\bar{x},\bar{y})\le \phi(\bar{x}),
		$
		which is a contradiction. Therefore,
		\[
		\limsup_{k\to\infty}\phi_k(x)\le \phi(x),\qquad \forall x\in X.
		\]
	
	\end{proof}

	\subsection{Proof for Lemma \ref{lem_liminf}}
	We first establish a uniform boundedness property for upper-level solution set when $x$ is restricted to a compact set.
	\begin{lemma}\label{lem_a4}
		Let \(B \subset X\) be a compact set, and let
		\[
		\mathcal S_p(x):=\arg\max_{y\in S(x)} F(x,y).
		\]
		Then there exists a constant \(M_B>0\) such that
		\[
		\|y\|\le M_B,\qquad \forall x\in B,\ \forall y\in \mathcal S_p(x).
		\]
		Equivalently, the set \(\bigcup_{x\in B}\mathcal S_p(x)\) is bounded.
	\end{lemma}
	
	\begin{proof}
		By Assumption \ref{assum2}, since \(B\) is bounded, there exists a bounded set \(D_B\subset \mathbb R^m\)
		such that
		\[
		S(x)\cap D_B\neq \emptyset,\qquad \forall x\in B.
		\]
		Let \(K_B:=\overline{D_B}\), where \(\overline{C}\) denotes the closure of a set \(C\). Then \(K_B\) is compact. Since \(F\) is continuous and \(B\) is compact,
		the quantity
		\[
		m_B:=\min\{F(x,y):x\in B,\ y\in K_B\}
		\]
		is well defined and finite.
		
		Set
		$C_B:=|m_B|+1>0.$
		By Lemma~\ref{coerciveuniform}, there exists \(R_B>0\) such that
		\[
		F(x,y)<-C_B<m_B,\qquad \forall x\in B,\ \forall y\in Y\ \text{with }\|y\|>R_B.
		\]
		
		Fix any $x\in B$ and $y\in\mathcal S_p(x)$, and choose
		$\hat y\in\mathcal S(x)\cap D_B$.
		By the optimality of $y$ and the definition of $m_B$,
		\[
		F(x,y)\ge F(x,\hat y)\ge m_B.
		\]
		If $\|y\|>R_B$, the preceding uniform coercivity estimate would give
		$F(x,y)<-C_B<m_B$, a contradiction. Thus $\|y\|\le R_B$
		for every $x\in B$ and $y\in\mathcal S_p(x)$, proving the claim
		with $M_B=R_B$.
	\end{proof}
	
	Next, we introduce an inequality relating $\phi_k(x)$ and $\phi(x)$. We omit the proof as it's similar as \citep[Lemma B.5.]{qichao2025single}.
	\begin{lemma}\label{lem_a5}
		For any $x \in X$, 
		\[
		\phi_{k}(x) \ge \phi(x)- \frac{\sigma_k+\delta_k}{2}\|y^*(x)\|^2,
		\]
		where $y^*(x) = \arg\min_{y \in \mathcal{S}_p(x)}\|y\|^2$.
	\end{lemma}

	Now, we are ready to provide the proof for Lemma \ref{lem_liminf}.
	\begin{proof}[Proof for Lemma \ref{lem_liminf}]
		From Lemma \ref{lem_a5}, for each $k$, we have the inequality:
		\begin{equation}\label{lem_liminf_eq1}
			\phi_{k}(x_k) \ge \phi(x_k)- \frac{\sigma_k+\delta_k}{2}\|y^*(x_k)\|^2,
		\end{equation}
		where $y^*(x) = \arg\min_{y \in \mathcal{S}_p(x)}\|y\|^2$.
		
		Since the sequence $\{x_k\}$ converges to $\bar{x}$, it is bounded. By Lemma \ref{lem_a4}, the sequence $\{y^*(x_k)\}$ is uniformly bounded. 
		Given that $\sigma_k,\delta_k \rightarrow 0$ as $k \rightarrow \infty$ and $\{y^*(x_k)\}$ is bounded, it follows that
		\[
		\lim_{k \to \infty} \frac{\sigma_k+\delta_k}{2}\|y^*(x_k)\|^2 = 0.
		\]
		Combining this with \eqref{lem_liminf_eq1}, we obtain
		\[
		\liminf_{k\to\infty }\phi(x_{k}) \le\liminf_{k\to\infty }\phi_{k}(x_k).
		\]
		Thus, by the lower semi-continuity of $\phi(x)$ at $\bar{x}$, we conclude that:
		\begin{equation*}
			\phi(\bar{x})\le\liminf_{k\to\infty }\phi(x_{k}) \le\liminf_{k\to\infty }\phi_{k}(x_k).
		\end{equation*}
		This completes the proof.
	\end{proof}
	
	\subsection{Proof for Theorem \ref{lim_yz}}
	\begin{proof}
		For brevity, denote
		$y_k^*:=y_k^*(x_k),\; z_k^*:=z_k^*(x_k).$
		First, we establish the property of the accumulation points of the sequence $\{y_k^*\}$.
		Let $\bar y$ be an accumulation point of $\{y_k^*\}$. By taking a subsequence and re-indexing the indices if necessary, we can assume without loss of generality that
		$y_k^*\to \bar y.$
		By applying Lemma \ref{valueofy}, we have $\bar y\in \mathcal S(\bar x)$.
		
		For each $k$, choose $\tilde{y}_k\in \mathcal S_p(x_k)$. Since $x_k\to \bar x$, the sequence $\{x_k\}$ is bounded. Hence, by Lemma \ref{lem_a4}, the sequence $\{\tilde{y}_k\}$ is bounded. Moreover, by Lemma \ref{saddlepointboundedness}, the sequences $\{y_k^*\}$ and $\{z_k^*\}$ are bounded.
		
		By the saddle point property, we have
		\[
		\psi_k(x_k,y_k^*,z_k^*)\ge \psi_k(x_k,\tilde{y}_k,z_k^*).
		\]
		Expanding this inequality and rearranging terms yields
		\begin{align*}
			F(x_k,y_k^*)-F(x_k,\tilde{y}_k)\ge\;&
			\rho_k\big(f(x_k,y_k^*)-f(x_k,\tilde{y}_k)\big)
			+\frac{\delta_k}{2}\big(\|y_k^*\|^2-\|\tilde{y}_k\|^2\big)  \\
			&\quad
			+\sigma_k\big(\langle y_k^*,z_k^*\rangle-\langle \tilde{y}_k,z_k^*\rangle\big).
		\end{align*}
		Since $\tilde{y}_k\in \mathcal S_p(x_k)\subset \mathcal S(x_k)$, we have
		\[
		f(x_k,y_k^*)-f(x_k,\tilde{y}_k)\ge 0.
		\]
		Together with the boundedness of $\{y_k^*\}$, $\{z_k^*\}$, and $\{\tilde{y}_k\}$, and the facts that $\delta_k\to 0$ and $\sigma_k\to 0$, this gives
		\[
		\liminf_{k\to\infty}\big(F(x_k,y_k^*)-F(x_k,\tilde{y}_k)\big)\ge 0.
		\]
		Since $y_k^*\to \bar y$ and $x_k\to \bar x$, the continuity of $F$ implies
		\[
		F(\bar x,\bar y)
		\ge \liminf_{k\to\infty}F(x_k,\tilde{y}_k)
		= \liminf_{k\to\infty}\phi(x_k).
		\]
		By the lower semi-continuity of $\phi$ at $\bar x$, we further obtain
		$F(\bar x,\bar y)\ge \phi(\bar x).$
		On the other hand, since $\bar y\in \mathcal S(\bar x)$, we have
		$F(\bar x,\bar y)\le \phi(\bar x).$
		Therefore,
		$F(\bar x,\bar y)=\phi(\bar x),$
		and hence $\bar y\in \mathcal S_p(\bar x)$.
		
		Next, we establish the convergence property of $z^*_{k}$. Recall that  $\hat y_k\in \mathcal S(x_{k})$ and $\hat y_k\to \bar y.$ 
		By Lemma \ref{saddlepointboundedness}, the sequence $\{z_k^*\}$ is bounded. It suffices to show that any accumulation point of $\{z_k^*\}$ is equal to $\bar y$.
		
		Let $\bar z$ be an arbitrary accumulation point of $\{z_k^*\}$. By taking a further subsequence and re-indexing if necessary, we can assume that
		$z_k^*\to \bar z.$
		From the other side of the saddle point property, we have
		\[
		\psi_{k}(x_k,y_k^*,z_k^*)\le \psi_{k}(x_k,y_k^*,\hat y_k).
		\]
		Expanding this inequality and completing the square gives
		\[
		\rho_{k} f(x_k,z_k^*)
		+\frac{\sigma_{k}}{2}\|z_k^*-y_k^*\|^2
		\le
		\rho_{k} f(x_k,\hat y_k)
		+\frac{\sigma_{k}}{2}\|\hat y_k-y_k^*\|^2.
		\]
		Multiplying by $2/\sigma_{k}$ and rearranging, we obtain
		\[
		\|z_k^*-y_k^*\|^2
		\le
		\frac{2\rho_{k}}{\sigma_{k}}
		\big(f(x_k,\hat y_k)-f(x_k,z_k^*)\big)
		+\|\hat y_k-y_k^*\|^2.
		\]
		Since $\hat y_k\in \mathcal S(x_k)$, it is a global minimizer of $f(x_k,\cdot)$ over $Y$. Hence
		\[
		f(x_k,\hat y_k)-f(x_k,z_k^*)\le 0,
		\]
		and therefore
		\[
		\|z_k^*-y_k^*\|^2\le \|\hat y_k-y_k^*\|^2.
		\]
		Taking the limit as $k\to\infty$, and using $\hat y_k\to \bar y$, $y_k^*\to \bar y$, and $z_k^*\to \bar z$, we obtain
		\[
		\|\bar z-\bar y\|^2
		\le
		\lim_{k\to\infty}\|\hat y_k-y_k^*\|^2
		=0.
		\]
		Thus $\bar z=\bar y$. Since every accumulation point of $\{z_k^*\}$ equals $\bar y$ and $\{z_k^*\}$ is bounded, we conclude that
		$z^*_{k}\to \bar y.$
	\end{proof}
	
	\subsection{Proof for Lemma \ref{multiplierconverge}}

	\begin{proof}
		For brevity, write
		$y_k^*:=y_k^*(x_k),\; z_k^*:=z_k^*(x_k).$
		Since $y_k^*\to\bar y$, Theorem~\ref{lim_yz} and
		Lemma~\ref{lem:isc-implies-lsc} yield
		$z_k^*\to\bar y$ and $\bar y\in\mathcal S_p(\bar x)$.
		
		Let
		\[
		\mathcal{A}^g:=\{i\in\{1,\dots,p\}: g_i(\bar y)=0\}.
		\]
		By lower-level LICQ, the family
			$\{\nabla g_i(\bar y)\}_{i\in\mathcal A^g}$ is linearly independent.
			Since $\bar y\in\mathcal S(\bar x)$, the lower-level KKT conditions admit
			a multiplier, and this independence makes it unique; denote it by
			$\bar\mu$. Thus $\Lambda(\bar x,\bar y)=\{\bar\mu\}$.
			If $\mathcal A^g$ is empty, all constraints are strictly inactive near
			$\bar y$. The unconstrained first-order conditions for the two saddle
			subproblems then give \eqref{optconditiony}--\eqref{optconditionz}
			with $\lambda_k=\mu_k=\bar\mu=0$ for sufficiently large $k$.
			The conclusions follow immediately in this case.
			We may therefore assume that $\mathcal A^g$ is nonempty.
		Hence, by continuity of
		$\nabla g$, there exists a neighborhood $U$ of $\bar y$ such that, for every $y\in U$,
		the matrix $\nabla g_{\mathcal{A}^g}(y)^\top$ has full column rank. Since
		$g_i(\bar y)<0$ for every $i\notin\mathcal{A}^g$, shrinking $U$ if necessary, we may also assume that
		\[
		g_i(y)<0,\qquad \forall\, y\in U,\ \forall\, i\notin\mathcal{A}^g.
		\]
		Because $y_k^*\to\bar y$ and $z_k^*\to\bar y$, we have $y_k^*,z_k^*\in U$ for all sufficiently large $k$.
		
		Now fix such a large $k$. Since
		$Y=\{w\in\mathbb R^m:g(w)\leq0\}$ and $\rho_k>0$, the constraint
		$\rho_k g(w)\leq0$ is an equivalent representation of $w\in Y$.
		Hence, the saddle-point property gives
		\[
		y_k^*\in\arg\max_{\rho_k g(y)\leq0}\psi_k(x_k,y,z_k^*),
		\qquad
		z_k^*\in\arg\min_{\rho_k g(z)\leq0}\psi_k(x_k,y_k^*,z).
		\]
		By the preceding local independence property and saddle-point
		optimality, there exist unique multipliers $\lambda_k$ and
		$\mu_k$ associated with the KKT systems of these constraint problems.
		Then \eqref{optconditiony}--\eqref{optconditionz} follow immediately from the
		KKT systems of the two subproblems.
		
		Moreover, because $g_i(y_k^*)<0$ and $g_i(z_k^*)<0$ for every $i\notin\mathcal{A}^g$, complementarity gives
		\[
		(\lambda_k)_i=(\mu_k)_i=0,\qquad \forall\, i\notin\mathcal{A}^g.
		\]
		Hence it suffices to study the subvectors
		\[
		\lambda_k^{\mathcal{A}^g}:=[(\lambda_k)_i]_{i\in\mathcal{A}^g},\qquad
		\mu_k^{\mathcal{A}^g}:=[(\mu_k)_i]_{i\in\mathcal{A}^g}.
		\]
		Rearranging the stationarity equations in \eqref{optconditiony} and
		\eqref{optconditionz}, we obtain
		\begin{align}
			\nabla g_{\mathcal{A}^g}(y_k^*)^\top \lambda_k^{\mathcal{A}^g}
			&=
			-\nabla_y f(x_k,y_k^*)
			+\frac{\nabla_yF(x_k,y_k^*)-\sigma_k z_k^*-\delta_k y_k^*}{\rho_k},
			\label{lambda-rearranged}
			\\
			\nabla g_{\mathcal{A}^g}(z_k^*)^\top \mu_k^{\mathcal{A}^g}
			&=
			-\nabla_y f(x_k,z_k^*)
			-\frac{\sigma_k(z_k^*-y_k^*)}{\rho_k}.
			\label{mu-rearranged}
		\end{align}
		Since $\nabla g_{\mathcal{A}^g}(y_k^*)^\top$ and $\nabla g_{\mathcal{A}^g}(z_k^*)^\top$ have full column rank,
		\eqref{lambda-rearranged}--\eqref{mu-rearranged} yield
		\begin{align*}
			\lambda_k^{\mathcal{A}^g}
			&=
			\big[\nabla g_{\mathcal{A}^g}(y_k^*)^\top\big]^+
			\left(
			-\nabla_y f(x_k,y_k^*)
			+\frac{\nabla_yF(x_k,y_k^*)-\sigma_k z_k^*-\delta_k y_k^*}{\rho_k}
			\right),
			\\
			\mu_k^{\mathcal{A}^g}
			&=
			\big[\nabla g_{\mathcal{A}^g}(z_k^*)^\top\big]^+
			\left(
			-\nabla_y f(x_k,z_k^*)
			-\frac{\sigma_k(z_k^*-y_k^*)}{\rho_k}
			\right).
		\end{align*}
		Here, $A^+$ denotes the Moore--Penrose pseudoinverse of a matrix $A$.
		Because each matrix to which it is applied above has full column rank,
		$A^+=(A^\top A)^{-1}A^\top$ in both formulas.
		
		Because $x_k\to\bar x$, $y_k^*\to\bar y$, $z_k^*\to\bar y$,
		$\frac1{\rho_k}\to 0$, $\sigma_k\to 0$ and $\delta_k\to 0,$
		the vectors inside the parentheses converge to $-\nabla_y f(\bar x,\bar y)$. On the other hand,
		for all sufficiently large $k$,
		\[
		\mathrm{rank}\big(\nabla g_{\mathcal{A}^g}(y_k^*)\big)
		=
		\mathrm{rank}\big(\nabla g_{\mathcal{A}^g}(z_k^*)\big)
		=
		|\mathcal{A}^g|
		=
		\mathrm{rank}\big(\nabla g_{\mathcal{A}^g}(\bar y)\big).
		\]
		Therefore, by the convergence of Moore--Penrose pseudoinverses \citep[Corollary 3.5]{stewart1977perturbation},
		\[
		\big[\nabla g_{\mathcal{A}^g}(y_k^*)^\top\big]^+
		\to
		\big[\nabla g_{\mathcal{A}^g}(\bar y)^\top\big]^+,
		\qquad
		\big[\nabla g_{\mathcal{A}^g}(z_k^*)^\top\big]^+
		\to
		\big[\nabla g_{\mathcal{A}^g}(\bar y)^\top\big]^+.
		\]
		Since $\bar\mu\in\Lambda(\bar x,\bar y)$ and the active-gradient
		matrix has full column rank, the common limit on
		$\mathcal A^g$ is precisely $\bar\mu^{\mathcal A^g}$. All components
		outside $\mathcal A^g$ vanish. Hence
		$\lambda_k\to\bar\mu$ and $\mu_k\to\bar\mu$, as asserted.
	\end{proof}

	\section{Proof for Section \ref{sec:stochastic-algorithm}}\label{appendix:stochastic-algorithm}

	\begin{lemma}\label{lem:bounded-saddles-monotone}
		Under Assumptions~\ref{assum1}, \ref{assum2}, and~\ref{ass:Xcompact},
		suppose that, for every $k\ge0$,
		\[
		\rho_k\ge0,\qquad
		\sigma_k\ge\sigma_{k+1}>0,\qquad
		\delta_k\ge\delta_{k+1}>0.
		\]
		Then there exists a constant $M>0$ such that
		\[
		\max\bigl\{\|y_k^*(x)\|,\|z_k^*(x)\|\bigr\}\le M,
		\qquad \forall k\ge0,\ \forall x\in X.
		\]
	\end{lemma}
	\begin{proof}
		By Assumptions~\ref{assum2} and~\ref{ass:Xcompact}, there exists a bounded
		set $D\subset Y$ such that we may choose
		$\hat y(x)\in\mathcal S(x)\cap D$ for every $x\in X$.
		The saddle-point property, comparison with $z=y_k^*(x)$, and completion
		of the square in $z$ at $y=\hat y(x)$ give
		\[
		F(x,y_k^*(x))\ge\phi_k(x)
		\ge F(x,\hat y(x))
		-\frac{\sigma_k+\delta_k}{2}\|\hat y(x)\|^2.
		\]
		The right-hand side is bounded below uniformly in $k$ and $x\in X$,
		since $X$ is compact, $D$ is bounded, and
		$\sigma_k+\delta_k\le\sigma_0+\delta_0$.
		Lemma~\ref{coerciveuniform} therefore yields uniform boundedness of
		$y_k^*(x)$. Moreover, comparison with $z=\hat y(x)$ in the
		$z$-subproblem gives
		\[
		\|z_k^*(x)-y_k^*(x)\|
		\le\|\hat y(x)-y_k^*(x)\|.
		\]
		Consequently,
		$\|z_k^*(x)\|\le2\|y_k^*(x)\|+\|\hat y(x)\|$,
		so $z_k^*(x)$ is uniformly bounded as well.
	\end{proof}

	For later use, define the saddle operator
	$T_k:\mathbb R^{n+2m}\to\mathbb R^{2m}$ by
	\[
	T_k(x,y,z)
	:=
	\big(-\nabla_y\psi_k(x,y,z),\ \nabla_z\psi_k(x,y,z)\big).
	\]
	For any fixed $x\in X$, since $\psi_k(x,\cdot,\cdot)$ is
	$\delta_k$-strongly concave in $y$ and $\sigma_k$-strongly convex in $z$,
	the operator $T_k(x,\cdot,\cdot)$ is strongly monotone on $Y\times Y$:
	\begin{equation}\label{T_mono}
		\begin{aligned}
			\big\langle T_k(x,u)-T_k(x,u'),\,u-u'\big\rangle
			& \ge
			\bar\sigma_k\|u-u'\|^2,
			\qquad \forall u,u'\in Y\times Y .
		\end{aligned}
	\end{equation}
	Moreover, the Lipschitz continuity of $\nabla F$ and $\nabla f$ implies that
	\begin{equation}\label{T_Lip}
		\|T_k(x,u)-T_k(x,u')\|
		\le (L_F+\rho_k L_f+2\sigma_k+\delta_k)\|u-u'\|,
		\quad \forall u,u'\in Y\times Y.
	\end{equation}
	The next four lemmas characterize the properties of
	\(u_k^*(\cdot)\) and \(\phi_k(\cdot)\). Their proofs follows from
	\citep[Lemmas C.2, C.3, C.4 and C.6]{qichao2025single}.
	\begin{lemma}\label{uxk+1xk}
		Let $\{\rho_k\}$, $\{\sigma_k\}$  and $\{\delta_k\}$ be sequences such that $ \rho_{k}, \sigma_{k},\delta_{k} >0$. Define $\bar{\sigma}_k = \min\{\sigma_k, \delta_k\}$. Then, for any $x, x'\in X$, the corresponding saddle points $u^*_k(x)$ and $u^*_k(x')$ satisfy:
		\begin{align}\label{uxk+1xkequ}
			\|u^*_k(x')-u^*_k(x)\|\le\frac{L_{F}+2\rho_kL_{f}}{\bar{\sigma}_k}\|x'-x\|.
		\end{align}
	\end{lemma}
	
	\begin{lemma}\label{uk+1uk}
		Let $\{\rho_k\}$, $\{\sigma_k\}$  and $\{\delta_k\}$ be sequences such that $\rho_{k+1} \ge \rho_{k}>0$, $\sigma_{k} \ge\sigma_{k+1}>0$ and $\delta_{k}\ge\delta_{k+1}>0$. Define $\bar{\sigma}_k = \min\{\sigma_k, \delta_k\}$. Then, for any fixed $x \in X$, we have 
		\begin{equation}\label{uk+1ukequ}
			\| u_{k+1}^*(x) - u_{k}^*(x)\| \le \frac{2 (\rho_{k+1} - \rho_k)}{\bar{\sigma}_k}M_{\nabla f} + \frac{3(\sigma_{k} - \sigma_{k+1})}{\bar{\sigma}_k}M_y+\frac{\delta_k-\delta_{k+1}}{\bar{\sigma}_k}M_y.
		\end{equation}
	\end{lemma}
	
	\begin{lemma}\label{Lipshitzofphi}
		Let $\{\rho_k\}$, $\{\sigma_k\}$  and $\{\delta_k\}$ be sequences such that $\rho_{k+1} \ge \rho_{k}>0$, $\sigma_{k} \ge\sigma_{k+1}>0$ and $\delta_{k}\ge\delta_{k+1}>0$. Define $\bar{\sigma}_k = \min\{\sigma_k, \delta_k\}$. Then, for any $x, x'\in X$, we have  
		\begin{align}
			\|\nabla \phi_k(x^\prime)-\nabla \phi_k(x)\|\le L_{\phi_k}\|x^\prime-x\|,
		\end{align}
		where $L_{\phi_k}:= \frac{(L_F + 2\rho_kL_f)(L_{F}+2\rho_kL_{f} + \bar{\sigma}_k)}{\bar{\sigma}_k} $.
	\end{lemma}

	\begin{lemma}\label{lem_b6}
		Let $\{\rho_k\}$, $\{\sigma_k\}$  and $\{\delta_k\}$ be sequences such that $\rho_{k+1} \ge \rho_{k}>0$, $\sigma_{k} \ge\sigma_{k+1}>0$ and $\delta_{k}\ge\delta_{k+1}>0$. Then, for any $x \in X$, we have 
		\begin{equation}
			\phi_{k+1}(x) - \phi_{k}(x) \le \left( \sigma_k - \sigma_{k+1} \right)\frac{M_y^2 }{2} + 2\left(\rho_{k+1}-\rho_{k}\right)M_f+(\delta_{k}-\delta_{k+1})\frac{M_y^2}{2}.
		\end{equation}
	\end{lemma}

	\subsection{Proof for Lemma \ref{udescent}}
	
	\begin{proof}
		By denoting $\tilde{T}(x^k, u^k)=(-d_y^k,d_z^k)$, the update rule for $u^{k+1}$ can be formulated as follows,
		\[
		u^{k+1} = \mathrm{Proj}_{Y \times Y} \left( u^k - \beta_k \tilde{T}(x^k, u^k) \right).
		\]
		where $\E[\tilde{T}(x^k, u^k)\mid \F_k]=T(x^k, u^k)$ and
		\[
		E[\|\tilde{T}(x^k, u^k)-T(x^k, u^k)\|^2\mid \F_k]\le (1+2\rho_k+2\rho_k^2)\Delta^2.
		\]
		Recall that  $u^*_k(x^k) = (y_{k}^*(x^k), z_{k}^*(x^k))$ is the unique saddle point of the minimax problem
		\[
		\min_{z \in {Y}} \max_{ y \in {Y}}\psi_{k}(x^k, y, z).
		\]
		From the first-order optimality conditions, $u^*_k(x^k)$ satisfies:
		\[
		0 \in T_k(x^k,u^*_k(x^k)) + N_{Y \times Y}(u^*_k(x^k)),
		\]
		which implies
		$
		u^*_k(x^k) = \mathrm{Proj}_{Y\times Y} \left( u^*_k(x^k) - \beta_k T_k(x^k, u^*_k(x^k)) \right).$
		Utilizing the non-expansiveness of the projection operator, the strong monotonicity of $T_k$ in \eqref{T_mono} and its Lipschitz continuity with respect to $u$ in \eqref{T_Lip}, we have
		\begin{align*}
			&\E[\|u^{k+1}-u_k^*(x^k)\|^2\mid \F_k]\\
			\le\;&\E[\|(u^k-\beta_k\tilde{T}_k(x^k,u^k))-(u_k^*(x^k)-\beta_kT_k(x^k,u_k^*(x^k)))\|^2\mid \F_k]\\
			\le \;&\|u^{k}-u_k^*(x^k)\|^2-\E[2\beta_k\langle T_k(x^k,u^k)-T_k(x^k,u_k^*(x^k)),u^{k}-u_k^*(x^k)\rangle\mid \F_k]\\
			&+\E [\beta_k^2\|\tilde{T}_k(x^k,u^k)-T_k(x^k,u^k)\|^2\mid \F_k]+\beta_k^2\|T_k(x^k,u^k)-T_k(x^k,u_k^*(x^k))\|^2 \\
			\le\;&(1-\bar{\sigma}_k\beta_k)\|u^{k}-u_k^*(x^k)\|^2+(1+2\rho_k+2\rho_k^2)\beta_k^2 \Delta^2.
		\end{align*}
		This completes the proof.
	\end{proof}	
	\subsection{Proof for Lemma \ref{udescentlemma}}
	\begin{proof}
		By applying the Cauchy-Schwarz inequality, for any $\tau>0$, we have 
		\begin{align}
			&\E[\|u^{k+1}- u_{k+1}^*(x^{k+1})\|^2\mid \F_k]\nonumber\\
			\le \;&(1+\tau)\E[\|u^{k+1}- u_{k}^*(x^{k})\|^2\mid \F_k]+(1+\frac{1}{\tau})\E[\|u_{k}^*(x^k)- u_{k+1}^*(x^{k+1})\|^2\mid \F_k]\nonumber\\
			\le\; &(1+\tau)\E[\|u^{k+1}- u_{k}^*(x^{k})\|^2\mid \F_k]+2(1+\frac{1}{\tau})\E[\|u_{k}^*(x^k)- u_{k}^*(x^{k+1})\|^2\mid \F_k]\nonumber\\
			&+2(1+\frac{1}{\tau})\E[\|u_{k}^*(x^{k+1})- u_{k+1}^*(x^{k+1})\|^2\mid \F_k].\label{udescentlemmaeq1}
		\end{align}
		Next, take $\tau=\frac{1}{2}\beta_k \bar{\sigma}_k$ in the above inequality. By applying Lemma \ref{udescent}, we obtain the following bound:
		\begin{align*}
			(1+\tau)\E[\|u^{k+1}- u_{k}^*(x^{k})\|^2\mid \F_k]\le& (1-\frac{1}{2}\beta_k \bar{\sigma}_k)\|u^{k}- u_{k}^*(x^{k})\|^2\\
			&+(1+\frac{1}{2}\beta_k \bar{\sigma}_k)(1+2\rho_k+2\rho_k^2)\beta_k^2 \Delta^2.
		\end{align*}
		Using Lemma \ref{uxk+1xk}, we can further bound the second term as follows:
		\begin{align*}
			&2(1+\frac{1}{\tau})\E[\|u_{k}^*(x^{k+1})-u_{k}^*(x^{k})\|^2\mid \F_k]\\
			\le\;& 2(1+\frac{2}{\beta_k\bar{\sigma}_k})\frac{(L_{F}+2\rho_kL_{f})^2}{\bar{\sigma}_k^2}\E[\|x^{k+1}-x^k\|^2\mid \F_k].
		\end{align*}
		Next, applying Lemma \ref{uk+1uk} with $x=x^{k+1}$, we obtain
		\[
		\begin{aligned}
			&2(1+\frac{1}{\tau})\E[\|u_{k+1}^*(x^{k+1})-u_{k}^*(x^{k+1})\|^2\mid \F_k] \\
			\le& 2(1+\frac{2}{\beta_k\bar{\sigma}_k}) ( \frac{12(\rho_{k+1} - \rho_k)^2}{\bar{\sigma}_k^2}M_{\nabla f}^2 + \frac{27(\sigma_{k} - \sigma_{k+1})^2}{\bar{\sigma}_k^2}M_y^2\hspace{-0.1em}+\hspace{-0.1em}\frac{3(\delta_{k} - \delta_{k+1})^2}{\bar{\sigma}_k^2}M_y^2 ).
		\end{aligned}
		\]        
		Finally, combining the above three inequalities with \eqref{udescentlemmaeq1}, we arrive at the desired inequality.
	\end{proof}
	
	\subsection{Proof for Lemma \ref{xdescentlemma}}

	\begin{proof}
		We decompose the total difference as follows:
		\begin{equation*}
			\begin{aligned}
							&\E[\phi_{k+1}(x^{k+1})\mid \F_k]-\phi_{k}(x^k)\\
				=\;&\E[\phi_{k+1}(x^{k+1})-\phi_{k}(x^{k+1})\mid \F_k] + \E[\phi_{k}(x^{k+1})\mid \F_k]-\phi_{k}(x^k).
			\end{aligned}
		\end{equation*}
		For the first term, applying Lemma \ref{lem_b6} with $x = x^{k+1}$:
		\begin{equation}\label{lem_b7_eq3}
			\begin{aligned}
							&\E[\phi_{k+1}(x^{k+1}) - \phi_{k}(x^{k+1})\mid \F_k] \\
				\le\; &\left( \sigma_k - \sigma_{k+1} \right)\frac{M_y^2 }{2} + 2\left(\rho_{k+1}-\rho_{k}\right)M_f+\left( \delta_k - \delta_{k+1} \right)\frac{M_y^2 }{2}.
			\end{aligned}
		\end{equation}
		
		For the second term, $\phi_{k}(x^{k+1})-\phi_{k}(x^k)$, we use the $L_{\phi_k}$-Lipschitz continuity of $\nabla \phi_{k}(x)$ established in Lemma \ref{Lipshitzofphi}. A standard descent inequality (cf. \cite[Lemma 5.7]{beck2017first} for smooth functions) states:
		\begin{equation}\label{lem_b7_eq1}
			\begin{aligned}
							&\E[\phi_{k}(x^{k+1})\mid \F_k]-\phi_{k}(x^k) \\\le\;& \E[\langle\nabla\phi_{k}(x^k),x^{k+1}-x^{k}\rangle\mid \F_k]+\frac{L_{\phi_k}}{2}\E[\|x^{k+1}-x^{k}\|^2\mid \F_k].
			\end{aligned}
		\end{equation}
		Next, applying the update rule for $x^{k+1}$, we get
		\[
		\frac{1}{\alpha_k} \E[\|x^{k+1}-x^{k}\|^2\mid \F_k] \le \E[\langle - d_{x}^k,x^{k+1}-x^{k}\rangle\mid \F_k].
		\]
		By combining this inequality with the previous one, and using the formula for $\nabla\phi_{k}(x^k)$ given in Theorem \ref{differentiable}, we obtain
		\begin{equation}\label{equnablaphi}
			\begin{aligned}
				&\E[\phi_{k}(x^{k+1})\mid \F_k]-\phi_{k}(x^k) + \left( \frac{1}{\alpha_k} - \frac{L_{\phi_k}}{2}\right) \E[\|x^{k+1}-x^{k}\|^2\mid \F_k] \\
				\le\,&\E[\langle\nabla_x\psi_{k}(x^k,y_{k}^*(x^k),z_{k}^*(x^k)) - {d}_{x}^k,x^{k+1}-x^{k}\rangle\mid \F_k] \\
				=\,&\E[\langle\nabla_x\psi_{k}(x^k,y_{k}^*(x^k),z_{k}^*(x^k)) - \nabla_x\psi_{k}(x^k,y^{k+1},z^{k+1}),x^{k+1}-x^{k}\rangle\mid \F_k] \\
				&+\E[\langle\nabla_x\psi_{k}(x^k,y^{k+1},z^{k+1})-d_{x}^k,x^{k+1}-x^{k}\rangle\mid \F_k]\\
			\end{aligned}
		\end{equation}
		For the first term in the right side of \eqref{equnablaphi},
		\begin{align*}
			&\E[\langle\nabla_x\psi_{k}(x^k,y_{k}^*(x^k),z_{k}^*(x^k)) - \nabla_x\psi_{k}(x^k,y^{k+1},z^{k+1}),x^{k+1}-x^{k}\rangle\mid \F_k] \\
			\le\, & \E[\left( (L_F + \rho_k L_f) \| y^{k+1}- y_{k}^*(x^k)\| + \rho_k L_f \| z^{k+1}- z_{k}^*(x^k)\| \right) \| x^{k+1}-x^{k}\|\mid \F_k] \\
			\le\, & \frac{\alpha_k}{2} (L_F + 2\rho_k L_f)^2\E[\| u^{k+1}- u_{k}^*(x^k)\|^2\mid \F_k] + \frac{1}{2\alpha_k} \E[\| x^{k+1}-x^{k}\|^2\mid \F_k] \\
			\le\, &\frac{\alpha_k}{2} (L_F + 2\rho_k L_f)^2 (1-\bar{\sigma}_k \beta_k) \E[\| u^{k}- u_{k}^*(x^k)\|^2\mid \F_k] \\
			&+ \frac{1}{2\alpha_k} \E[\| x^{k+1}-x^{k}\|^2\mid \F_k]+\frac{\alpha_k}{2} (L_F + 2\rho_k L_f)^2(1+2\rho_k+2\rho_k^2)\beta_k^2\Delta^2, 
		\end{align*}
		where the last inequality follows from Lemma \ref{udescent}. 
		For the second in the right side of  \eqref{equnablaphi},
		\begin{align*}
			&\E[\langle\nabla_x\psi_{k}(x^k,y^{k+1},z^{k+1})-d_{x}^k,x^{k+1}-x^{k}\rangle\mid \F_k] \\
			\le\, &\alpha_k \E[\|\nabla_x\psi_{k}(x^k,y^{k+1},z^{k+1})-d_{x}^k\|^2\mid \F_k]+\frac{1}{4\alpha_k}\E[\|x^{k+1}-x^{k}\|^2\mid \F_k]\\
			=&\alpha_k \E[\|e_x^k\|^2\mid \F_k]+\frac{1}{4\alpha_k}\E[\|x^{k+1}-x^{k}\|^2\mid \F_k].
		\end{align*}
		The conclusion follows by combining the above inequality with \eqref{lem_b7_eq1} and \eqref{lem_b7_eq3}.
	\end{proof}

	\subsection{Proof for Lemma \ref{variancereduction}}
	
	\begin{proof}
		For brevity, write
		\[
		\begin{aligned}
			&D_x^k:=\nabla_x\psi_k(x^k,y^{k+1},z^{k+1}),\quad D_x^{k-1}:=\nabla_x\psi_{k-1}(x^{k-1},y^{k},z^{k}),\\
			&G_x^k := \nabla_x\psi_k(x^k,y^{k+1},z^{k+1};\xi_k^x),
			\quad
			H_x^{k-1} := \nabla_x\psi_{k-1}(x^{k-1},y^k,z^k;\xi_k^x).
		\end{aligned}
		\]
		By \eqref{direction_x}, we have
		$d_x^k
		=
		G_x^k + (1-\eta_k)\big(d_x^{k-1}-H_x^{k-1}\big),$
		and therefore
		\[
		e_x^k
		=
		G_x^k-D_x^k
		+
		(1-\eta_k)\big(e_x^{k-1}+D_x^{k-1}-H_x^{k-1}\big).
		\]
		Set
		$A_k := G_x^k-D_x^k,
		\;
		B_k := D_x^{k-1}-H_x^{k-1}.$
		Then
		\[
		e_x^k = (1-\eta_k)e_x^{k-1}+A_k+(1-\eta_k)B_k.
		\]
		
		Since $x^k,x^{k-1},y^k,z^k,y^{k+1},z^{k+1}$ and $e_x^{k-1}$ are
		$\F_{k+\frac12}$-measurable, while $\xi_k^x$ is independent of $\F_{k+\frac12}$,
		Assumption~\ref{ass:oracle} yields
		\[
		\E[A_k\mid \F_{k+\frac12}] = 0,
		\qquad
		\E[B_k\mid \F_{k+\frac12}] = 0.
		\]
		Hence
		\[
		\begin{aligned}
			&\E\!\left[\|e_x^k\|^2 \mid \F_{k+\frac12}\right]\\
			=\;&
			(1-\eta_k)^2\|e_x^{k-1}\|^2
			+
			\E\!\left[\|A_k+(1-\eta_k)B_k\|^2 \mid \F_{k+\frac12}\right] \\
			\le\;&
			(1-\eta_k)^2\|e_x^{k-1}\|^2
			+
			2\E\!\left[\|A_k+B_k\|^2 \mid \F_{k+\frac12}\right]
			+
			2\eta_k^2\E\!\left[\|B_k\|^2 \mid \F_{k+\frac12}\right].
		\end{aligned}
		\]
		By Assumption~\ref{ass:oracle},
		\[
		\E\!\left[\|B_k\|^2 \mid \F_{k+\frac12}\right]
		=
		\E\!\left[\|H_x^{k-1}-D_x^{k-1}\|^2 \mid \F_{k+\frac12}\right]
		\le (3+6\rho_{k}^2)\Delta^2.
		\]
		
		Next, define
		$\widehat D_x^{k-1}:=\nabla_x\psi_k(x^{k-1},y^k,z^k),
		\;
		\widehat G_x^{k-1}:=\nabla_x\psi_k(x^{k-1},y^k,z^k;\xi_k^x)$, $S_k
		:=
		\big(G_x^k-\widehat G_x^{k-1}\big)
		-
		\big(D_x^k-\widehat D_x^{k-1}\big),$ and $R_k
		:=
		\big(\widehat G_x^{k-1}-H_x^{k-1}\big)
		-
		\big(\widehat D_x^{k-1}-D_x^{k-1}\big).$
		Then we have:
		\[
		A_k+B_k = S_k + R_k,
		\]
		We first bound $R_k$. By the definition of $\nabla_x\psi_k$,
		\[		\widehat G_x^{k-1}-H_x^{k-1}
		=
		-(\rho_k-\rho_{k-1})
		\Big(
		\nabla_xf(x^{k-1},y^k;\xi_k^x)
		-
		\nabla_xf(x^{k-1},z^k;\xi_k^x)
		\Big),\]
		and similarly,
		$
		\widehat D_x^{k-1}-D_x^{k-1}
		=
		-(\rho_k-\rho_{k-1})
		\Big(
		\nabla_xf(x^{k-1},y^k)
		-
		\nabla_xf(x^{k-1},z^k)
		\Big).
		$
		Therefore,
		\[
		\begin{aligned}
			R_k
			=-(\rho_k-\rho_{k-1})\Big[
			&\big(\nabla_xf(x^{k-1},y^k;\xi_k^x)-\nabla_xf(x^{k-1},z^k;\xi_k^x)\big) \\
			&-\big(\nabla_xf(x^{k-1},y^k)-\nabla_xf(x^{k-1},z^k)\big)
			\Big].
		\end{aligned}
		\]
		Using $\|a-b\|^2\le 2\|a\|^2+2\|b\|^2$ and Assumption~\ref{ass:oracle}, we obtain
		\[
		\E\!\left[\|R_k\|^2 \mid \F_{k+\frac12}\right]
		\le
		4(\rho_k-\rho_{k-1})^2\Delta^2.
		\]
		
		We next bound $S_k$. Since
		$
		S_k
		=
		\big(G_x^k-\widehat G_x^{k-1}\big)
		-
		\E\!\left[\big(G_x^k-\widehat G_x^{k-1}\big)\mid \F_{k+\frac12}\right],
		$
		we have
		\[
		\E\!\left[\|S_k\|^2\mid \F_{k+\frac12}\right]
		\le
		\E\!\left[\|G_x^k-\widehat G_x^{k-1}\|^2\mid \F_{k+\frac12}\right].
		\]
		By Assumption~\ref{ass:twopoint}, the definition of $\nabla_x\psi_k$, and
		$\|a+b+c\|^2\le 3(\|a\|^2+\|b\|^2+\|c\|^2)$,
		\[
		\begin{aligned}
			&\E\!\left[\|G_x^k-\widehat G_x^{k-1}\|^2\mid \F_{k+\frac12}\right]\\
			\le\;&
			3L_F^2\big(\|x^k-x^{k-1}\|^2\hspace{-0.1em}+\hspace{-0.1em}\|y^{k+1}-y^k\|^2\big) +3\rho_k^2L_f^2\big(\|x^k-x^{k-1}\|^2\hspace{-0.1em}+\hspace{-0.1em}\|y^{k+1}-y^k\|^2\big) \\
			&+3\rho_k^2L_f^2\big(\|x^k-x^{k-1}\|^2+\|z^{k+1}-z^k\|^2\big) \\
			\le\;&
			3(L_F+2\rho_kL_f)^2\big(\|x^k-x^{k-1}\|^2+\|u^{k+1}-u^k\|^2\big).
		\end{aligned}
		\]
		
		Combining the bounds for $S_k$ and $R_k$, we get
		\[
		\begin{aligned}
			\E\!\left[\|A_k+B_k\|^2\mid \F_{k+\frac12}\right]
			\le\;&
			6(L_F+2\rho_kL_f)^2\big(\|x^k-x^{k-1}\|^2+\|u^{k+1}-u^k\|^2\big)\\
			&+
			8(\rho_k-\rho_{k-1})^2\Delta^2.
		\end{aligned}
		\]
		Substituting the last two estimates into the recursion for
		$\E[\|e_x^k\|^2\mid \F_{k+\frac12}]$ and taking conditional expectation with respect to $\F_k$ and using the tower
		property gives
		\[
		\begin{aligned}
			&\E\!\left[\|e_x^k\|^2\mid \F_k\right]\\
			\le\;&
			(1-\eta_k)^2\|e_x^{k-1}\|^2
			+(6+12\rho_{k}^2)\eta_k^2\Delta^2+12(L_F+2\rho_kL_f)^2\|x^k-x^{k-1}\|^2\\
			&
			+16(\rho_k-\rho_{k-1})^2\Delta^2
			+12(L_F+2\rho_kL_f)^2\,\E\!\left[\|u^{k+1}-u^k\|^2 \mid \F_k\right].
		\end{aligned}
		\]
	Lemma~\ref{lem:bounded-saddles-monotone} ensures that $u_k^*(x^k)$ is uniformly bounded. As a result, by the definition of $T_{k}$, there exist a constant $M_T$ such that $\|T_k(x^{k},u_{k}^*(x^{k}))\|^2\le \rho_{k}^2M_T$.
		Finally, by the non-expansiveness of projection and boundedness of $u_k^*(x^k)$, we have 
		\begin{align*}
			&\E[\|u^{k+1}-u^{k}\|^2\mid \F_k]\\
			\le\; &\beta_{k}^2\E[\|\tilde{T}_{k}(x^{k},u^{k})\|^2\mid \F_k]\\
			\le \;&3\beta_{k}^2\E[\|\tilde{T}_{k}(x^{k},u^{k})-{T}_{k}(x^{k},u^{k})\|^2\mid \F_k]+3\beta_{k}^2\|T_{k}(x^{k},u^{k})-T_{k}(x^{k},u_{k}^*(x^{k}))\|^2\\
			&+3\beta_{k}^2\|T_k(x^{k},u_{k}^*(x^{k}))\|^2\\
			\le\;&(3+6\rho_k+6\rho_k^2)\beta_{k}^2\Delta^2+3\beta_{k}^2(L_F+\rho_{k}L_f+2\sigma_{k}+\delta_k )^2\|u^{k}-u_{k}^*(x^{k})\|^2\\
			&+3\beta_{k}^2 \rho_{k}^2 M_{T}.
		\end{align*}
		which implies the desired result.
	\end{proof}
	\subsection{Proof for Proposition \ref{meritfunctionproposition}}
	The next lemma provides a uniform lower bound needed in the merit function construction.
	
	\begin{lemma}
		\label{lem:det-concave-lower-bound}
		Let $\{\rho_k\}$, $\{\sigma_k\}$  and $\{\delta_k\}$ be sequences such that $\rho_{k+1} \ge \rho_{k}>0$, $\sigma_{k} \ge\sigma_{k+1}>0$ and $\delta_{k}\ge\delta_{k+1}>0$.
		Assume that $\phi$ is bounded below on $X$. Then there exists a constant $\underline\phi\in\mathbb R$ such that
		\[
		\phi_k(x)\ge \underline\phi,
		\qquad \forall x\in X,\ \forall k.
		\]
	\end{lemma}
	
	\begin{proof}
		Let
		$y^*(x):=\arg\min_{y\in \mathcal S_p(x)}\|y\|^2.$
		By Lemma~\ref{lem_a4} with $B=X$, the set $\bigcup_{x\in X}\mathcal S_p(x)$ is bounded, so there exists a constant $M_{y^*}>0$ such that
		\[
		\|y^*(x)\|\le M_{y^*},
		\qquad \forall x\in X.
		\]
		On the other hand, Lemma~\ref{lem_a5} gives
		\[
		\phi_k(x)
		\ge
		\phi(x)-\frac{\sigma_k+\delta_k}{2}\|y^*(x)\|^2
		\ge
		\inf_{x\in X}\phi(x)-\frac{\sup_k\sigma_k+\sup_k\delta_k}{2}M_{y^*}^2.
		\]
		Therefore the choice
		\[
		\underline\phi
		:=
		\inf_{x\in X}\phi(x)-\frac{\sup_k\sigma_k+\sup_k\delta_k}{2}M_{y^*}^2
		\]
		exists. In particular, for the polynomial schedules used below, one may simply take
		\[
		\underline\phi=
		\inf_{x\in X}\phi(x)-\frac{\sigma_0+\delta_0}{2}M_{y^*}^2.
		\]
	\end{proof}
	\begin{proof}[Proof for Proposition \ref{meritfunctionproposition}]
		Throughout the proof, $C>0$ denotes a generic constant independent of $k$,
		whose value may change from line to line. For simplicity, write
		\begin{align*}
					&L_k:=L_F+2\rho_kL_f,\; L_{T_k}=L_F+\rho_{k}L_f+2\sigma_{k}+\delta_k,\;\\
			&\Delta x^k:=x^{k+1}-x^k,\;
			E_k^u:=\|u^k-u_k^*(x^k)\|^2 .
		\end{align*}
		Under the parameter schedules \eqref{par},
		we have, 
		\[
		\bar\sigma_k=\min\{\sigma_0,\delta_0\}(k+1)^{-t}=O(k^{-t}),\qquad
		L_k=O(k^t),\qquad
		L_{\phi_k}=O(k^{3t}).
		\]
		Moreover, by choosing $\beta_0/\sigma_0$ sufficiently small if necessary, the
		stepsize condition
		$0<\beta_k<
		\frac{\bar\sigma_k}{(L_F+\rho_kL_f+2\sigma_k+\delta_k)^2}$ holds for every $k$. Hence Lemmas~\ref{udescentlemma},
		\ref{xdescentlemma}, and \ref{variancereduction} are applicable.
		
		Since $a_{k+1}\le a_k$, $b_{k+1}\le b_k$, $c_{k+1}\le c_k$, and $d_{k+1}\le d_k$,
		the definition of $V_k$ yields
		\[
		\begin{aligned}
			&\E[V_{k+1}\mid\F_k]-V_k\\
			\le\;&
			a_k\big(\E[\phi_{k+1}(x^{k+1})\mid\F_k]-\phi_k(x^k)\big)\\ &+b_k\big(\E[\|u^{k+1}-u_{k+1}^*(x^{k+1})\|^2\mid\F_k]
			-\|u^k-u_k^*(x^k)\|^2\big)\\
			&+c_k\big(\E[\|e_x^k\|^2\mid\F_k]-\|e_x^{k-1}\|^2\big)+d_k\big(\E[\|\Delta x^k\|^2\mid\F_k]-\|\Delta x^{k-1}\|^2\big).
		\end{aligned}
		\]
		Applying Lemmas~\ref{udescentlemma}, \ref{xdescentlemma}, and
		\ref{variancereduction}, and collecting the coefficients of
		$\E[\|\Delta x^k\|^2\mid\F_k]$, $E_k^u$, $\|e_x^{k-1}\|^2$, and
		$\|\Delta x^{k-1}\|^2$, we obtain
		\[
		\begin{aligned}
			\E[V_{k+1}\mid\F_k]-V_k
			\le\;&
			-A_k\,\E[\|\Delta x^k\|^2\mid\F_k]
			-B_k\,E_k^u
			-(2c_k\eta_k-c_k\eta_k^2)\|e_x^{k-1}\|^2\\
			&+a_k\alpha_k\,\E[\|e_x^k\|^2\mid\F_k]
			-D_k\|\Delta x^{k-1}\|^2
			+R_k^{\rm det}+R_k^{\rm st},
		\end{aligned}
		\]
		where
		\[
		\begin{aligned}
			A_k
			:=\;&
			a_k\left(\frac{1}{4\alpha_k}-\frac{L_{\phi_k}}{2}\right)
			-2b_k\left(1+\frac{2}{\beta_k\bar\sigma_k}\right)\frac{L_k^2}{\bar\sigma_k^2}
			-d_k,\\
			B_k
			:=\;&
			\frac12 b_k\beta_k\bar\sigma_k
			-\frac12 a_k\alpha_kL_k^2(1-\beta_k\bar\sigma_k)
			-36c_k\beta_k^2L_k^2L_{T_k}^2,\\
			D_k
			:=&d_k-12c_kL_k^2 ,\\
						R_k^{\rm det}:=\;& a_k\left( \sigma_k - \sigma_{k+1} \right)\frac{M_y^2 }{2} + 2a_k\left(\rho_{k+1}-\rho_{k}\right)M_f+ a_k\left( \delta_k - \delta_{k+1} \right)\frac{M_y^2 }{2}\\
			&+b_k(2+\frac{4}{\beta_k\bar{\sigma}_k}) ( \frac{12(\rho_{k+1} - \rho_k)^2}{\bar{\sigma}_k^2}M_{\nabla f}^2 + \frac{27(\sigma_{k} - \sigma_{k+1})^2}{\bar{\sigma}_k^2}M_y^2\\
			&+\frac{3(\delta_{k} - \delta_{k+1})^2}{\bar{\sigma}_k^2}M_y^2 ),
		\end{aligned}
		\]\[
			\begin{aligned}
			R_k^{\rm st}:=\;&
			\frac12 a_k\alpha_kL_k^2
			(1+2\rho_k+2\rho_k^2)\beta_k^2\Delta^2+c_k(6+12\rho_k^2)\eta_k^2\Delta^2\\
			&+b_k\left(1+\frac12\beta_k\bar\sigma_k\right)
			(1+2\rho_k+2\rho_k^2)\beta_k^2\Delta^2
			+16c_k(\rho_k-\rho_{k-1})^2\Delta^2
			\\
			&+36c_k\beta_k^2L_k^2
			(1+2\rho_k+2\rho_k^2)\Delta^2+36c_k\beta_k^2L_k^2\rho_k^2M_T.
		\end{aligned}
		\]
		
		From the schedules \eqref{par},
		\[
		\frac{a_k}{\alpha_k}=O(k^{7t+s}),\quad
		a_kL_{\phi_k}=O(k^t),\quad b_k\left(1+\frac{2}{\beta_k\bar\sigma_k}\right)\frac{L_k^2}{\bar\sigma_k^2}
		=O(k^{6t+s}).
		\]
		Therefore, for all sufficiently large $k$,
		$A_k\ge \frac{a_k}{8\alpha_k}.$
		
		We now convert the descent in $\Delta x^k$ into the projected-gradient mapping
		$\mathcal G_k$. By the nonexpansiveness of $P_X$ and the definition of
		$\mathcal G_k$, we have
		\[
		\begin{aligned}
			\|\mathcal G_k(x^k)\|^2
			=&
			\frac{1}{\alpha_k^2}
			\left\|x^k-P_X\big(x^k-\alpha_k\nabla\phi_k(x^k)\big)\right\|^2\\
			\le&
			\frac{3}{\alpha_k^2}\E[\|\Delta x^k\|^2\mid\F_k]
			+3L_k^2\E[\|u^{k+1}-u_k^*(x^k)\|^2\mid\F_k]\\&
			+3\E[\|e_x^k\|^2\mid\F_k].
		\end{aligned}
		\]
		Using Lemma~\ref{udescent} in the second term gives
		\begin{align*}
					\|\mathcal G_k(x^k)\|^2
			\le&
			\frac{3}{\alpha_k^2}\E[\|\Delta x^k\|^2\mid\F_k]
			+3L_k^2E_k^u
			+3\E[\|e_x^k\|^2\mid\F_k]\\&
			+3L_k^2(1+2\rho_k+2\rho_k^2)
			\beta_k^2\Delta^2 .
		\end{align*}
		Consequently,
		\[
		\begin{aligned}
			&-\frac{a_k}{8\alpha_k}\E[\|\Delta x^k\|^2\mid\F_k]
			\\\le\;&
			-\frac{a_k\alpha_k}{24}\|\mathcal G_k(x^k)\|^2
			+\frac{a_k\alpha_kL_k^2}{8}E_k^u
			+\frac{a_k\alpha_k}{8}\E[\|e_x^k\|^2\mid\F_k]\\
			&+\frac18 a_k\alpha_kL_k^2
			(1+2\rho_k+2\rho_k^2)\beta_k^2\Delta^2 .
		\end{aligned}
		\]
		Substituting this estimate into the previous bound yields
		\[
		\begin{aligned}
			\E[V_{k+1}\mid\F_k]-V_k
			\le\;&
			-\frac{a_k\alpha_k}{24}\|\mathcal G_k(x^k)\|^2
			-\widetilde B_k E_k^u
			-(2c_k\eta_k-c_k\eta_k^2)\|e_x^{k-1}\|^2\\
			&+\frac98a_k\alpha_k\E[\|e_x^k\|^2\mid\F_k]
			-\widetilde D_k\|\Delta x^{k-1}\|^2
			+\widetilde R_k^{\rm det}+\widetilde R_k^{\rm st},
		\end{aligned}
		\]
		where
		\begin{align*}
					&\widetilde B_k:=B_k-\frac18a_k\alpha_kL_k^2,\;
			\widetilde D_k:=D_k, \; \widetilde R_k^{\rm det}=R_k^{\rm det}, \;\\ &\widetilde R_k^{\rm st}=R_k^{\rm st}+\frac{1}{8} a_k\alpha_kL_k^2
			(1+2\rho_k+2\rho_k^2)\beta_k^2\Delta^2 .
		\end{align*}
		
		It remains to handle the term involving $\E[\|e_x^k\|^2\mid\F_k]$. Applying
		Lemma~\ref{variancereduction} once more and using $c_k\eta_k=O(k^{-10t-s}),\;
		a_k\alpha_k=O(k^{-11t-s}),$ we obtain
		\[
		\begin{aligned}
			&-(2c_k\eta_k-c_k\eta_k^2)\|e_x^{k-1}\|^2
			+\frac{9}{8}a_k\alpha_k\E[\|e_x^k\|^2\mid\F_k]\\
			\le\;&
			-c_k\eta_k\|e_x^{k-1}\|^2
			+\frac{27}{2} a_k\alpha_kL_k^2\|\Delta x^{k-1}\|^2
			+\frac{81}{2} a_k\alpha_k\beta_k^2L_k^2L_{T_k}^2E_k^u
			+\widehat R_k^{\rm st}
			+\widehat R_k^{\rm det},
		\end{aligned}
		\]
		where 
		\[\begin{aligned}
			\widehat R_k^{\rm det}:=\; &18 a_k\alpha_k(\rho_k-\rho_{k-1})^2\Delta^2\\
			\widehat R_k^{\rm st}:=\;&
			\frac98 a_k\alpha_k(6+12\rho_k^2)
			\eta_k^2\Delta^2+\frac{81}{2}a_k\alpha_k\beta_k^2L_k^2
			(1+2\rho_k+2\rho_k^2)\Delta^2\\
			&+\frac{81}{2}a_k\alpha_k\beta_k^2L_k^2
			\rho_k^2M_T.
		\end{aligned}
		\]
		Combining the last two displays gives
		\[
		\begin{aligned}
			\E[V_{k+1}\mid\F_k]-V_k
			\le\;&
			-\frac{a_k\alpha_k}{24}\|\mathcal G_k(x^k)\|^2
			-\bar{B}_kE_k^u
			-\bar{D}_k\|\Delta x^{k-1}\|^2
			\\
			&-c_k\eta_k\|e_x^{k-1}\|^2+\bar R_k^{\rm det}+\bar R_k^{\rm st},
		\end{aligned}
		\]
		where
		\begin{align*}
				&	\bar{B}_k
			= \widetilde B_k -\frac{81}{2}  a_k\alpha_k\beta_k^2L_k^2L_{T_k}^2,\;
			\bar{D}_k
			=\widetilde D_k-\frac{27}{2} a_k\alpha_kL_k^2, \\&\bar R_k^{\rm det}=\widehat R_k^{\rm det}+\widetilde R_k^{\rm det}, \; \bar R_k^{\rm st}=\widehat R_k^{\rm st}+\widetilde R_k^{\rm st}.
		\end{align*}
		Recalling the parameter schedules \eqref{par}, we have 
		\[
		\begin{aligned}
			&b_k\beta_k\bar\sigma_k=O(k^{-8t-s}),\quad
			a_k\alpha_kL_k^2=O(k^{-9t-s}),\quad\\
			&c_k\beta_k^2L_k^2L_{T_k}^2=O(k^{-9t-2s}), \quad a_k\alpha_k\beta_k^2L_k^2L_{T_k}^2= O(k^{-15t-3s}).
		\end{aligned}
		\]
		Thus, for all sufficiently large $k$,
		$\bar{B}_k\ge \frac{1}{4}b_k\beta_k\bar\sigma_k.$
		Moreover,
		\[
		d_k=O(k^{-2t}),\quad
		c_kL_k^2=O(k^{-3t}),\quad a_k\alpha_kL_k^2=O(k^{-9t-s}).
		\]
		so $\bar{D}_k\ge 0$ for all sufficiently large $k$.
		Consequently, we obtain
		\[
		\E[V_{k+1}\mid\F_k]-V_k
		\le
		-\frac{a_k\alpha_k}{24}\|\mathcal G_k(x^k)\|^2
		-\frac14b_k\beta_k\bar\sigma_kE_k^u
		+\bar{R}_k^{\rm det}+\bar R_k^{\rm st}.
		\]
		Finally, we estimate the two remainders. Using parameter schedules \eqref{par}. We can verify that:
		\[
		\bar{R}_k^{\rm det}=O(k^{-t-1})+O(k^{-2+s+6t}),
		\]
		Since $11t+s<1$, by defining
		$\zeta_k:=\bar R_k^{\rm det},$
		we have $\sum_{k=0}^{\infty}\zeta_k<\infty$.
		
		For the stochastic remainder, the schedules imply
		\[
		\bar R_k^{\rm st}=O(k^{-9t-2s}).
		\]
		Therefore, for all sufficiently large $k$,
		\[
		\E[V_{k+1}\mid\F_k]-V_k
		\le
		-\frac{a_k\alpha_k}{24}\|\mathcal G_k(x^k)\|^2
		-\frac14 b_k\beta_k\bar\sigma_k\|u^k-u_k^*(x^k)\|^2
		+Ck^{-9t-2s}+\zeta_k .
		\]
		This proves the desired descent inequality.
	\end{proof}

\end{appendices}

\bibliographystyle{plain}
\bibliography{sn-bibliography}

@inproceedings{lorraine2020optimizing,
  title = {Optimizing Millions of Hyperparameters by Implicit Differentiation},
  author = {Lorraine, Jonathan and Vicol, Paul and Duvenaud, David},
  booktitle = {International Conference on Artificial Intelligence and Statistics},
  year = {2020}
}

@inproceedings{cutkosky2019momentum,
	title={Momentum-based variance reduction in non-convex {SGD}},
	author={Cutkosky, Ashok and Orabona, Francesco},
	booktitle={Advances in Neural Information Processing Systems},
	year={2019}
}

@book{rockafellar2009variational,
  title = {Variational Analysis},
  author = {Rockafellar, Ralph Tyrrell and Wets, Roger J-B.},
  volume = {317},
  year = {2009},
  publisher = {Springer},
  address   = {Berlin, Heidelberg}
}

@inproceedings{macdonald2007overview,
  title={Overview of the {TREC} 2007 Blog Track.},
  author={Macdonald, Craig and Ounis, Iadh and Soboroff, Ian},
  booktitle={Text REtrieval Conference},
  year={2007}
}

@book{beck2017first,
  title={First-order methods in optimization},
  author={Beck, Amir},
  year={2017},
  publisher={SIAM},
  address   = {Philadelphia, PA}
}

@inproceedings{androutsopoulos2000evaluation,
  title={An Evaluation of Naive {B}ayesian Anti-Spam Filtering},
  author={Androutsopoulos, Ion and Koutsias, John and Chandrinos, Konstantinos V and Paliouras, George and Spyropoulos, Constantine D},
  booktitle={Workshop on Machine Learning in the New Information Age},
  year={2000}
}

@inproceedings{metsis2006spam,
  title={Spam filtering with naive {B}ayes-which naive {B}ayes?},
  author={Metsis, Vangelis and Androutsopoulos, Ion and Paliouras, Georgios},
  booktitle={Conference on Email and Anti-Spam},
  volume={17},
  pages={28--69},
  year={2006},
  organization={Mountain View, CA}
}

@inproceedings{ounis2006overview,
  title={Overview of the {TREC} 2006 Blog Track.},
  author={Ounis, Iadh and Macdonald, Craig and Soboroff, Ian},
  booktitle={Text REtrieval Conference},
  year={2006}
}

@article{guo2024sensitivity,
  title={Sensitivity analysis of the maximal value function with applications in nonconvex minimax programs},
  author={Guo, Lei and Ye, Jane J and Zhang, Jin},
  journal={Mathematics of Operations Research},
  volume={49},
  number={1},
  pages={536--556},
  year={2024},
  publisher={INFORMS}
}

@article{liu2018pessimistic,
  title = {Pessimistic Bilevel Optimization: A Survey},
  author = {Liu, June and Fan, Yuxin and Chen, Zhong and Zheng, Yue},
  journal = {International Journal of Computational Intelligence Systems},
  volume = {11},
  number = {1},
  pages = {725--736},
  year = {2018},
  publisher = {Springer}
}

@inproceedings{loridan1988approximate,
  title = {Approximate Solutions for Two-Level Optimization Problems},
  author = {Loridan, Pierre and Morgan, Jacqueline},
  booktitle = {French-German Conference on Optimization},
  year = {1988},
  publisher = {Springer}
}

@article{aboussoror2001existence,
  title = {Existence of Solutions to Two-Level Optimization Problems with Nonunique Lower-Level Solutions},
  author = {Aboussoror, Abdelmalek and Loridan, Pierre},
  journal = {Journal of Mathematical Analysis and Applications},
  volume = {254},
  number = {2},
  pages = {348--357},
  year = {2001},
  publisher = {Elsevier}
}

@inproceedings{lu2023slm,
  title={{S}lm: A smoothed first-order lagrangian method for structured constrained nonconvex optimization},
  author={Lu, Songtao},
  booktitle={Advances in Neural Information Processing Systems},
  year={2023}
}

@article{lu2024first,
  title = {First-Order Penalty Methods for Bilevel Optimization},
  author = {Lu, Zhaosong and Mei, Sanyou},
  journal = {SIAM Journal on Optimization},
  volume = {34},
  number = {2},
  pages = {1937--1969},
  year = {2024},
  publisher = {SIAM}
}

@inproceedings{kwon2023fully,
  title = {A Fully First-Order Method for Stochastic Bilevel Optimization},
  author = {Kwon, Jeongyeol and Kwon, Dohyun and Wright, Stephen J. and Nowak, Robert D.},
  booktitle = {International Conference on Machine Learning},
  year = {2023}
}

@article{pedregosa2011scikit,
  title={Scikit-learn: Machine learning in Python},
  author={Pedregosa, Fabian and Varoquaux, Ga{\"e}l and Gramfort, Alexandre and Michel, Vincent and Thirion, Bertrand and Grisel, Olivier and Blondel, Mathieu and Prettenhofer, Peter and Weiss, Ron and Dubourg, Vincent},
  journal={Journal of Machine Learning Research},
  volume={12},
  pages={2825--2830},
  year={2011}
}

@inproceedings{guanadaprox,
  title={{A}da{P}rox: A Novel Method for Bilevel Optimization under Pessimistic Framework},
  author={Guan, Ziwei and Sow, Daouda and Lin, Sen and Liang, Yingbin},
  booktitle={Conference on Parsimony and Learning},
  year={2025}
}

@inproceedings{shen2023penalty,
  title={On penalty-based bilevel gradient descent method},
  author={Shen, Han and Chen, Tianyi},
  booktitle={International Conference on Machine Learning},
  year={2023}
}

@article{benchouk2025scholtes,
  title={Scholtes relaxation method for pessimistic bilevel optimization},
  author={Benchouk, Imane and Jolaoso, Lateef O. and Nachi, Khadra and Zemkoho, Alain B},
  journal={Set-Valued and Variational Analysis},
  volume={33},
  number={2},
  pages={10},
  year={2025},
  doi={10.1007/s11228-025-00747-5},
  publisher={Springer}
}

@inproceedings{
kwon2023penalty,
title={On Penalty Methods for Nonconvex Bilevel Optimization and First-Order Stochastic Approximation},
author={Jeongyeol Kwon and Dohyun Kwon and Wright, Stephen J.  and Robert D. Nowak},
booktitle={International Conference on Learning Representations},
year={2024}
}

@article{dempe2019two,
  title={Two-level value function approach to non-smooth optimistic and pessimistic bilevel programs},
  author={Dempe, Stephan and Mordukhovich, Boris S and Zemkoho, Alain B},
  journal={Optimization},
  volume={68},
  number={2-3},
  pages={433--455},
  year={2019}
}

@inproceedings{liu2021towards,
  title={Towards gradient-based bilevel optimization with non-convex followers and beyond},
  author={Risheng Liu and Yaohua Liu and Shangzhi Zeng and Jin Zhang},
  booktitle={Advances in Neural Information Processing Systems},
  year={2021}
}

@article{allende2013solving,
  title={Solving bilevel programs with the {KKT}-approach},
  author={Gemayqzel Bouza Allende and Georg Still},
  journal={Mathematical Programming},
  volume={138},
  pages={309--332},
  year={2013},
  publisher={Springer}
}

@article{outrata1990numerical,
  title={On the numerical solution of a class of {S}tackelberg problems},
  author={Jiří V. Outrata},
  journal={Zeitschrift f{\"u}r Operations Research},
  volume={34},
  number={4},
  pages={255--277},
  year={1990}
}

@article{ye1995optimality,
  title={Optimality conditions for bilevel programming problems},
  author={Jane J. Ye and D. L Zhu},
  journal={Optimization},
  volume={33},
  number={1},
  pages={9--27},
  year={1995},
  publisher={Taylor and Francis}
}

@book{luo1996mathematical,
  title={Mathematical Programs with Equilibrium Constraints},
  author={Zhi-Quan Luo and Jong-Shi Pang and Daniel Ralph},
  year={1996},
  publisher={Cambridge University Press},
  address={Cambridge}
}

@article{dempe2013bilevel,
  title={The bilevel programming problem: reformulations, constraint qualifications and optimality conditions},
  author={Stephan Dempe and Alain B. Zemkoho},
  journal={Mathematical Programming},
  volume={138},
  pages={447--473},
  year={2013},
  publisher={Springer}
}

@article{colson2007overview,
  title={An overview of bilevel optimization},
  author={Benoît Colson and Patrice Marcotte and Gilles Savard},
  journal={Annals of Operations Research},
  volume={153},
  pages={235--256},
  year={2007},
  publisher={Springer}
}

@book{bonnans2013perturbation,
  title={Perturbation analysis of optimization problems},
  author = {Bonnans, J. Fr{\'e}d{\'e}ric and Shapiro, Alexander},
  year={2013},
  publisher={Springer},
  address = {New York}
}

@inproceedings{arbel2021amortized,
title={Amortized Implicit Differentiation for Stochastic Bilevel Optimization},
author={Michael Arbel and Julien Mairal},
booktitle={International Conference on Learning Representations},
year={2022}
}

@article{hong2023two,
  title={A two-timescale stochastic algorithm framework for bilevel optimization: Complexity analysis and application to actor-critic},
  author={Mingyi Hong and Hoi-To Wai and Zhaoran Wang and Zhuoran Yang},
  journal={SIAM Journal on Optimization},
  volume={33},
  number={1},
  pages={147--180},
  year={2023},
  publisher={SIAM}
}

@inproceedings{sow2022convergence,
  title={On the convergence theory for hessian-free bilevel algorithms},
  author={Daouda Sow and Kaiyi Ji and Yingbin Liang},
  booktitle={Advances in Neural Information Processing Systems},
  year={2022}
}

@inproceedings{franceschi2018bilevel,
  title={Bilevel programming for hyperparameter optimization and meta-learning},
  author={Luca Franceschi and Paolo Frasconi and Saverio Salzo and Riccardo Grazzi and Massimiliano Pontil},
  booktitle={International Conference on Machine Learning},
  year={2018}
}

@book{vonStackelbergHeinrich1953TTot,
  title     = {The Theory of the Market Economy},
  author    = {von Stackelberg, Heinrich},
  publisher = {Oxford University Press},
  address   = {London},
  year      = {1952}
}

@article{ustun2024hyperparameter,
  title={Hyperparameter Tuning Through Pessimistic Bilevel Optimization},
  author={Ustun, Meltem Apaydin and Xu, Liang and Zeng, Bo and Qian, Xiaoning},
  journal={arXiv preprint},
  note={arXiv:2412.03666},
  year={2024}
}

@inproceedings{liu2022bome,
  title={{B}ome! bilevel optimization made easy: A simple first-order approach},
  author={Bo Liu and Mao Ye and Stephen J. Wright and Peter Stone and Qiang Liu},
  booktitle={Advances in Neural Information Processing Systems},
  year={2022}
}

@inproceedings{ji2021bilevel,
  title={Bilevel optimization: Convergence analysis and enhanced design},
  author={Kaiyi Ji and Junjie Yang and Yingbin Liang},
  booktitle={International Conference on Machine Learning},
  year={2021}
}

@inproceedings{liu2020generic,
  title={A generic first-order algorithmic framework for bi-level programming beyond lower-level singleton},
  author={Risheng Liu and Pan Mu and Xiaoming Yuan and Shangzhi Zeng and Jin Zhang},
  booktitle={International Conference on Machine Learning},
  year={2020}
}

@inproceedings{pmlr-v70-franceschi17a,
  title = 	 {Forward and Reverse Gradient-Based Hyperparameter Optimization},
  author =       {Luca Franceschi and Michele Donini and Paolo Frasconi and Massimiliano Pontil},
  booktitle = 	 {International Conference on Machine Learning},
  year = 	 {2017},
}

@inproceedings{maclaurin2015gradient,
  title={Gradient-based hyperparameter optimization through reversible learning},
  author={Maclaurin, Dougal and Duvenaud, David and Adams, Ryan P},
  booktitle={International Conference on Machine Learning},
  year={2015},
}

@article{benchouk2026relaxation,
  title={Relaxation methods for pessimistic bilevel optimization},
  author={Benchouk, Imane and Jolaoso, Lateef O. and Nachi, Khadra and Zemkoho, Alain B},
  journal={Set-Valued and Variational Analysis},
  volume={34},
  number={1},
  pages={1},
  year={2026},
  publisher={Springer}
}

@article{zeng2020practical,
  title={A practical scheme to compute the pessimistic bilevel optimization problem},
  author={Bo Zeng},
  journal={INFORMS Journal on Computing},
  volume={32},
  number={4},
  pages={1128--1142},
  year={2020},
  publisher={INFORMS}
}

@article{aboussoror2005weak,
  title={Weak linear bilevel programming problems: existence of solutions via a penalty method},
  author={Abdelmalek Aboussoror and Abdelatif Mansouri},
  journal={Journal of Mathematical Analysis and Applications},
  volume={304},
  number={1},
  pages={399--408},
  year={2005},
  publisher={Elsevier}
}

@inproceedings{pedregosa2016hyperparameter,
  title={Hyperparameter optimization with approximate gradient},
  author={Pedregosa, Fabian},
  booktitle={International Conference on Machine Learning},
  year={2016}
}

@article{zheng2013exact,
  title={An exact penalty method for weak linear bilevel programming problem},
  author={Yue Zheng and Zhongping Wan and Kangtai Sun and Tao Zhang},
  journal={Journal of Applied Mathematics and Computing},
  volume={42},
  number={1},
  pages={41--49},
  year={2013},
  publisher={Springer}
}

@article{aussel2019pessimistic,
  title={Is pessimistic bilevel programming a special case of a mathematical program with complementarity constraints?},
  author={Didier Aussel and Anton Svensson},
  journal={Journal of Optimization Theory and Applications},
  volume={181},
  pages={504--520},
  year={2019},
  publisher={Springer}
}

@article{lampariello2019standard,
  title={The standard pessimistic bilevel problem},
  author={Lorenzo Lampariello and Simone Sagratella and Oliver Stein},
  journal={SIAM Journal on Optimization},
  volume={29},
  number={2},
  pages={1634--1656},
  year={2019},
  publisher={SIAM}
}

@article{dempe2014necessary,
  title={Necessary optimality conditions in pessimistic bilevel programming},
  author={Stephan Dempe and Boris S. Mordukhovich and Alain B. Zemkoho},
  journal={Optimization},
  volume={63},
  number={4},
  pages={505--533},
  year={2014},
  doi={10.1080/02331934.2012.696641},
  publisher={Taylor and Francis}
}

@article{wiesemann2013pessimistic,
  title={Pessimistic bilevel optimization},
  author={Wolfram Wiesemann and Angelos Tsoukalas and Polyxeni-Margarita Kleniati and Berç Rustem},
  journal={SIAM Journal on Optimization},
  volume={23},
  number={1},
  pages={353--380},
  year={2013},
  publisher={SIAM}
}

@article{benfield2024classification,
  title={Classification under strategic adversary manipulation using pessimistic bilevel optimisation},
  author={Benfield, David and Coniglio, Stefano and Kunc, Martin and Vuong, Phan Tu and Zemkoho, Alain B},
  journal={arXiv preprint},
  note={arXiv:2410.20284},
  year={2024}
}

@incollection{dempe2020bilevel,
  title={Bilevel optimization},
  author={Stephan Dempe and Alain B. Zemkoho},
  booktitle={Springer Optimization and its Applications},
  volume={161},
  year={2020},
  publisher={Springer},
  address={Cham}
}

@inproceedings{bruckner2011stackelberg,
  title={Stackelberg games for adversarial prediction problems},
  author={Michael Brückner and Tobias Scheffer},
  booktitle={International Conference on Knowledge Discovery and Data Mining},
  year={2011}
}

@incollection{ban2009risk,
	author    = {Ban, Xuegang and Lu, Shu and Ferris, Michael C. and Liu, Henry X.},
	title     = {Risk Averse Second Best Toll Pricing},
	booktitle = {Transportation and Traffic Theory 2009: Golden Jubilee},
	pages     = {197--218},
	publisher = {Springer},
	address   = {Boston, MA},
	year      = {2009},
}

@article{zheng2016pessimistic,
  title={Pessimistic bilevel optimization model for risk-averse production-distribution planning},
  author={Zheng, Yue and Zhang, Guangquan and Han, Jialin and Lu, Jie},
  journal={Information Sciences},
  volume={372},
  pages={677--689},
  year={2016},
  publisher={Elsevier}
}

@article{kis2021optimistic,
  title={On optimistic and pessimistic bilevel optimization models for demand response management},
  author={Tamás Kis and András Kovács and Csaba Mészáros},
  journal={Energies},
  volume={14},
  number={8},
  pages={2095},
  year={2021}
}

@article{calvete2024novel,
  title={A novel approach to pessimistic bilevel problems. {A}n application to the rank pricing problem with ties},
  author={Calvete, Herminia I and Gal{\'e}, Carmen and Hern{\'a}ndez, Aitor and Iranzo, Jos{\'e} A},
  journal={Optimization},
  pages={1--34},
  year={2024},
  publisher={Taylor \& Francis}
}

@inproceedings{qichao2025single,
  title={A Single-Loop Gradient Algorithm for Pessimistic Bilevel Optimization via Smooth Approximation},
  author={Cao, Qichao and Zeng, Shangzhi and Zhang, Jin},
  booktitle={Advances in Neural Information Processing Systems},
  year={2025}
}

@article{jimenez2025pessimistic,
  title={Pessimistic bilevel optimization approach for decision-focused learning},
  author={Jim{\'e}nez, Diego and Pagnoncelli, Bernardo K and Yaman, Hande},
  journal={arXiv preprint},
  note={arXiv:2501.16826},
  year={2025}
}

@inproceedings{zharmagambetov2023landscape,
  title={Landscape surrogate: Learning decision losses for mathematical optimization under partial information},
  author={Zharmagambetov, Arman and Amos, Brandon and Ferber, Aaron and Huang, Taoan and Dilkina, Bistra and Tian, Yuandong},
  booktitle={Advances in Neural Information Processing Systems},
  year={2023}
}

@article{bucareydecision,
	title={Decision-focused predictions via pessimistic bilevel optimization: complexity and algorithms},
	author={Bucarey L{\'o}pez, V{\'\i}ctor and Calder{\'o}n, Sophia and Mu{\~n}oz, Gonzalo and Semet, Fr{\'e}d{\'e}ric},
	journal={Computational Optimization and Applications},
	pages={1--31},
	year={2026},
	publisher={Springer}
}

@article{elmachtoub2022smart,
  title={Smart “predict, then optimize”},
  author={Elmachtoub, Adam N and Grigas, Paul},
  journal={Management Science},
  volume={68},
  number={1},
  pages={9--26},
  year={2022},
  publisher={INFORMS}
}

@article{stewart1977perturbation,
	title={On the perturbation of pseudo-inverses, projections and linear least squares problems},
	author={Stewart, Gilbert W},
	journal={SIAM Review},
	volume={19},
	number={4},
	pages={634--662},
	year={1977},
	publisher={SIAM}
}

@article{chen2024learning,
  title={Learning Decisions Offline from Censored Observations with $\epsilon$-insensitive Operational Costs},
  author={Chen, Minxia and Fu, Ke and Huang, Teng and Bai, Miao},
  journal={arXiv preprint},
  note={arXiv:2408.07305},
  year={2024}
}

@inproceedings{ferber2023surco,
  title={{S}urco: Learning linear surrogates for combinatorial nonlinear optimization problems},
  author={Ferber, Aaron M and Huang, Taoan and Zha, Daochen and Schubert, Martin and Steiner, Benoit and Dilkina, Bistra and Tian, Yuandong},
  booktitle={International Conference on Machine Learning},
  year={2023}
}

@article{steintrue,
	title={A true single--level reformulation for pessimistic bilevel optimization},
	author={Stein, Oliver and Zemkoho, Alain B},
	year={2026}
}

@article{fanaee2014event,
  title={Event labeling combining ensemble detectors and background knowledge},
  author={Fanaee-T, Hadi and Gama, Joao},
  journal={Progress in Artificial Intelligence},
  volume={2},
  number={2},
  pages={113--127},
  year={2014},
  publisher={Springer}
}

@article{ye1997exact,
  author = {Ye, J. J. and Zhu, D. L. and Zhu, Q. J.},
  title = {Exact Penalization and Necessary Optimality Conditions for Generalized Bilevel Programming Problems},
  journal = {SIAM Journal on Optimization},
  volume = {7},
  number = {2},
  pages = {481--507},
  year = {1997},
  doi = {10.1137/S1052623493257344}
}

\end{document}